\RequirePackage{fix-cm}
\documentclass[12pt]{article}
\usepackage[fontsize=11.5pt]{scrextend}
\usepackage[left=2.4cm,right=1.5cm,top=2.3cm,bottom=1cm]{geometry}
\usepackage{paralist}
\usepackage{marginnote}
\usepackage{amsfonts}
\usepackage{enumitem}
\usepackage{amsthm}
\usepackage{amsrefs}
\usepackage{amssymb}
\usepackage{amsmath}
\usepackage{graphicx}
\usepackage{varwidth}
\usepackage{hyperref}
\usepackage{fancyhdr}
\usepackage{stmaryrd}
\usepackage{multirow}
\usepackage{blkarray}
\usepackage[autostyle]{csquotes}
\usepackage[mathscr]{euscript}
\hypersetup{
    colorlinks=true,
    linkcolor=black,
    filecolor=magenta,      
    urlcolor=black,
    citecolor=black
}
\DeclareMathAlphabet{\mathcalligra}{T1}{calligra}{c}{h}
\providecommand{\U}[1]{\protect\rule{.1in}{.1in}}
\newtheorem*{theorem*}{Theorem} 
\newtheorem{theorem}{Theorem}
\newtheorem{proposition}[theorem]{Proposition}
\newtheorem{lemma}[theorem]{Lemma}

\newtheorem{corollary}[theorem]{Corollary}

\newtheorem{remark}[theorem]{Remark}

\newtheorem{example}[theorem]{Example}
\newtheorem{examples}[theorem]{Examples}

\DeclareMathOperator{\inter}{int}

\DeclareMathOperator{\Imag}{Im}
\DeclareMathOperator{\Real}{Re}
\newcommand{\R}{\mathbb{R}}
\newcommand{\Q}{\mathbb{Q}}

\newcommand{\C}{\mathbb{C}}

\newcommand{\grad}{\mbox{grad\,}}

\newcommand{\nap}{\nabla^{\perp}}

\def\<{{\langle}}
\def\>{{\rangle}}

\def\Jp{J^\perp}
\def\n{\nabla}
\def\d{\partial}
\def\dz{\partial}
\def\dzb{\bar{\partial}}
\def\a{\alpha}

\def\add{\a(\dz,\dz)}

\def\id{I}

\def\th{\theta}
\def\w{\omega}
\def\bea{\begin{eqnarray*} }
\def\eea{\end{eqnarray*} }
\def\be{\begin{equation} }
\def\ee{\end{equation} }

\def\nap{\nabla^\perp}

\def\qed{\ifhmode\unskip\nobreak\fi\ifmmode\ifinner
\else\hskip5 pt \fi\fi\hbox{\hskip5 pt \vrule width4 pt
height6 pt  depth1.5 pt \hskip 1pt }}

\begin{document}
\vspace{-3em}
\title{Bonnet and Isotropically Isothermic Surfaces in \\ 4-Dimensional Space Forms}
\author{Kleanthis Polymerakis}
\date{}
\maketitle
\renewcommand{\thefootnote}{\fnsymbol{footnote}}
\footnotetext{The author would like to acknowledge financial support by the General Secretariat for Research and Technology (GSRT) and the Hellenic Foundation for Research and Innovation (HFRI), Grant No: 133.} 
\renewcommand{\thefootnote}{\arabic{footnote}}
\vspace{-2em}
\begin{abstract}
We study the Bonnet problem for surfaces in 4-dimensional space forms, namely, 
to what extent a surface is determined by the metric and the mean curvature.
Two isometric surfaces have the same mean curvature
if there exists a parallel vector bundle isometry between their normal bundles that preserves the mean curvature vector fields. 
We deal with the structure of the moduli space of congruence classes of isometric surfaces 
with the same mean curvature, and with properties inherited on a surface by this structure.
The study of this problem led us to a new conformally invariant property, called isotropic isothermicity, 
that coincides with the usual concept of isothermicity for surfaces lying in totally umbilical hypersurfaces, 
and is related to lines of curvature and infinitesimal isometric deformations that preserve the mean curvature vector field. 
The class of isotropically isothermic surfaces includes the one of surfaces with a vertically harmonic Gauss lift 
and particularly the minimal surfaces, and overlaps with that of isothermic surfaces without containing the entire class.

We show that if a simply-connected surface is not proper Bonnet, which means that the moduli space is a finite set, 
then it admits either at most one, or exactly three Bonnet mates. 
For simply-connected proper Bonnet surfaces, the moduli space
is either 1-dimensional with at most two connected components diffeomorphic to the circle, or the 2-dimensional torus.
We prove that simply-connected Bonnet surfaces lying in totally geodesic hypersurfaces of the ambient space as surfaces of nonconstant mean curvature, 
always admit Bonnet mates that do not lie in any totally umbilical hypersurface. 
Such surfaces either admit exactly three Bonnet mates, or they are proper Bonnet with moduli space the torus.
We show that isotropic isothermicity characterizes the proper Bonnet surfaces, and we provide 
relevant conditions for non-existence of Bonnet mates for compact surfaces.
Moreover, we study compact surfaces that are locally proper Bonnet,
and we prove that the existence of a uniform substructure on the local moduli spaces, 
characterizes surfaces with a vertically harmonic Gauss lift that are neither minimal, nor superconformal.
In particular, we show that the only compact, locally proper Bonnet surfaces with moduli space the torus,
are those with nonvanishing parallel mean curvature vector field and positive genus.
\end{abstract}

\section{Introduction}

The theory of isometric or conformal immersions deals with the study of 
isometric or conformal invariants of immersions, 
aiming at the possible classification of the immersions with respect to these invariants.
In the classical 
theory of surfaces in a complete, simply-connected 3-dimensional space form $\Q^3_c$ of curvature $c$, 
a basic problem is to investigate to what extent several geometric data determine a surface up to congruence, 
and furthermore, to study and classify the exceptional surfaces that are not uniquely determined by certain data.

In 1867, Bonnet \cite{B} raised the problem to what extent a surface in $\Q^3_c$ is determined by the metric and the mean curvature. 
This naturally leads to the following question: given an isometric immersion $f\colon M\to \Q^3_c$ of a 2-dimensional oriented Riemannian manifold, 
how many noncongruent to $f$ isometric immersions of $M$ into $\Q^3_c$ can exist with the same mean curvature with $f$? 
Any noncongruent to $f$ such surface is called a Bonnet mate of $f$.
A generic surface in $\Q^3_c$ is uniquely determined by the metric and the mean curvature. The exceptions are called Bonnet surfaces.
Several aspects of the Bonnet problem have been studied by Bonnet \cite{B}, Cartan \cite{Ca}, Tribuzy \cite{Tr}, 
Chern \cite{Ch2}, Roussos-Hernandez \cite{RH}, Kenmotsu \cite{K}, and Smyth-Tinaglia \cite{ST} among many others.
It turns out that a simply-connected surface $f\colon M\to \Q^3_c$ either admits at most one Bonnet mate, or the moduli space of 
all isometric immersions of $M$ into $\Q^3_c$ that have the same mean curvature with $f$, is the circle
$\mathbb{S}^1\simeq \R/2\pi\mathbb{Z}$. In the latter case, the surface is called proper Bonnet. 
It has been shown by Bonnet \cite{B} and Lawson \cite{L}, that simply-connected surfaces 
with constant mean curvature are proper Bonnet, unless they are totally umbilical.
For compact surfaces, Lawson and Tribuzy \cite{LT} proved that a surface of non-constant mean curvature in $\Q^3_c$, admits at most one Bonnet mate. 
It still remains an open problem if there exist compact surfaces of non-constant mean curvature in $\R^3$ that do admit a Bonnet mate.

The Bonnet problem for surfaces in $\Q^3_c$ is closely related to the extensively studied class of isothermic surfaces in $\Q^3_c$.
It was shown by Raffy \cite{Ra} that a proper Bonnet surface is isothermic away from its isolated umbilic points. 
Afterwards, Graustein \cite{Grau} proved that an isothermic Bonnet surface is proper Bonnet. 
His characterization of proper Bonnet surfaces involving the isothermicity, has been used by Bobenko and Eitner \cite{BE} 
for the classification of simply-connected, umbilic-free proper Bonnet surfaces of non-constant mean curvature in $\R^3$. 
On the other hand, Kamberov, Pedit and Pinkall \cite{KPP} described all simply-connected, umbilic-free Bonnet pairs in $\R^3$ in terms of isothermic surfaces.
Recently, Jensen, Musso and Nicolodi \cite{JMN} provided sufficient conditions in terms of isothermicity, for non-existence of Bonnet mates
for compact surfaces.

An umbilic-free surface $f\colon M\to\Q^3_c$ is called isothermic if it admits a conformal curvature line parametrization
around every point. This is equivalent to the co-closeness of the principal connection form of the surface, 
which is a globally defined 1-form on $M$. 
Isothermicity is a conformally invariant property that appears in several problems where a surface
is not uniquely determined by certain geometric invariants. 
As a matter of fact, isothermic surfaces admit an amount of transformations that preserve geometric data,
and they are characterized by the existence of these transformations (cf. \cite{H-J}).
Classical examples of isothermic surfaces in $\Q^3_c$ are the umbilic-free surfaces with 
constant mean curvature, and particularly, the minimal surfaces, as well as their M\"obius transformations.
The notion of isothermicity has been extended for surfaces in the Euclidean space with arbitrary codimension by Palmer \cite{Pa}.
Isothermicity in arbitrary codimension is again a conformally invariant property, and such isothermic surfaces
inherit the most of the properties of those  in 3-dimensional space forms.
For instance, they are characterized by the existence of analogous transformations.
However, in codimension greater than one, the isothermicity implies flatness of the normal bundle of the surface,
and this restricts the class of isothermic surfaces from including the minimal surfaces and their M\"obius transformations.

There is a characterization of isothermic surfaces in $\R^3$ which has no higher codimensional analogue, namely, 
an umbilic-free surface in $\R^3$ is isothermic if and only if it locally admits a nontrivial infinitesimal isometric
deformation that preserves the mean curvature. 
Probably the only recent proof of this result can be found in \cite{LP}, where this  
characterization of isothermicity has been extended to discrete surfaces in $\R^3$.  
As mentioned in \cite{CGS}, this result dates back to  
the 19th century, and it seems to have been almost forgotten 
until this reference, since surfaces that admit such deformations  
had in the meantime been studied, without establishing a correlation with isothermicity (see Section 6.1 of the survey \cite{IS}). 
The theory of infinitesimal isometric deformations of surfaces and submanifolds in the Euclidean space has a long and rich history, 
as can be seen in the surveys \cite{IS,IMS}, and is still developing (cf. \cite{DV2, Jim, DJ}).
In particular, the relation of isothermicity with the Bonnet problem for surfaces in $\R^3$
verifies very elegantly the quote of Efimov stated in \cite{IMS}, that 
"the theory of infinitesimal isometric deformations is the differential of the theory of isometric deformations".

Besides space forms, the Bonnet problem has been studied for surfaces in homogeneous 3-manifolds \cite{GMM},
and it was recently raised for surfaces in static 3-manifolds \cite{LMW}.

The Bonnet problem for surfaces in 4-dimensional space forms $\Q^4_c$ has been studied in \cite{PV}.
Two isometric surfaces in $\Q^4_c$ are said to have {\emph{the same mean curvature}} if there exists a parallel vector bundle isometry
between their normal bundles that preserves the mean curvature vector fields.
Most of the results in \cite{PV} concern compact surfaces and are global in nature.

In this paper, we focus mainly on local aspects of the Bonnet problem for surfaces in $\Q^4_c$.
The local study of the problem led us to a new conformally invariant property,
which has a similar effect on the Bonnet problem for surfaces in $\Q^4_c$ 
with that of isothermicity on the classical Bonnet problem. 
This property is called isotropic isothermicity and we discuss it first.

We introduce the notion of isotropically isothermic surfaces in $\Q^4_c$, generalizing the one 
of isothermic surfaces in $\Q^3_c$, as follows:
using the two isotropic parts of the Hopf differential of an oriented surface $f\colon M\to \Q^4_c$,
we introduce two differential 1-forms $\Omega^+$ and $\Omega^-$, called the mixed connection forms of $f$. 
The form $\Omega^\pm$ is defined 
away from pseudo-umbilic points of $f$, that are the points where the curvature ellipse of $f$ is a circle,
at which the normal curvature satisfies $\pm K_N\geq0$.
For an umbilic-free surface lying in some totally umbilical hypersurface of $\Q^4_c$,
both mixed connection forms coincide with the principal connection form of the surface. 
Extending naturally the definition of isothermic surfaces in $\Q^3_c$,
we call a surface $f\colon M\to \Q^4_c$ {\emph{isotropically isothermic}} if 
at least one of the mixed connection forms is 
defined and co-closed on the whole $M$. 
If this occurs for both mixed connection forms, then $f$ is
called {\emph{strongly isotropically isothermic}}.

It turns out that isotropic isothermicity is a property invariant under conformal changes of the metric of the ambient space.
Examples of isotropically isothermic surfaces in $\Q^4_c$ are the non-superconformal surfaces with a vertically harmonic Gauss lift,
the minimal superconformal surfaces, and their M\"obius transformations, away from isolated points.
In particular, non-superconformal minimal surfaces are strongly isotropically isothermic away from 
pseudo-umbilic points.
We note that, as follows from \cite{Ha,PV}, surfaces with a vertically harmonic Gauss lift are the  
analogues in $\Q^4_c$ of constant mean curvature surfaces in $\Q^3_c$, and particularly,
superconformal surfaces with a vertically harmonic Gauss lift generalize the totally umbilical surfaces. 
The class of strongly isotropically isothermic surfaces includes the one of isothermic surfaces lying in totally umbilical hypersurfaces of the ambient space, 
however, we show that there exist isothermic surfaces in $\R^4$ which are not isotropically isothermic.

For surfaces in $\R^4$, we prove that isotropic isothermicity is related to infinitesimal isometric 
deformations that preserve the mean curvature,
and that strong isotropic isothermicity involves the principal curvature lines, studied in \cite{GGTG, GS},
along which the second fundamental form  of the surface points in the direction of a principal axis of the curvature ellipse.
For an infinitesimal isometric deformation of a surface $f\colon M\to\R^4$, we define  
the {\emph{parallel preservation in the normal bundle}} under the deformation, 
of quantities related to the second fundamental form of $f$, in such a way
that the deformation is trivial if and only if it preserves parallelly in the normal bundle the mean curvature vector field 
and the Hopf differential, i.e., the second fundamental form. 
We note that parallel preservation of the mean curvature vector field in the normal bundle, implies preservation of its length and of the normal curvature.
Our first result is the following.

\smallskip

\begin{theorem} \label{IIID}
Let $f\colon M\to \R^4$ be an oriented surface, free of pseudo-umbilic points.
The surface $f$ is isotropically isothermic if and only if it locally admits a nontrivial infinitesimal isometric deformation
that preserves parallelly in the normal bundle, the mean curvature vector field and an isotropic part of the Hopf differential. Moreover, $f$ is strongly
isotropically isothermic if and only if it is isotropically isothermic and admits a conformal principal curvature line parametrization around every point.
\end{theorem}

For umbilic-free superconformal surfaces with nowhere-vanishing mean curvature vector field, we show that isotropic isothermicity
is related to the mean-directional curvature lines, studied in \cite{Me}, along which the second fundamental form of the surface
points in the direction of the mean curvature vector. 
It is known that such surfaces have a holomorphic Gauss lift (cf. \cite{ES}). 
All these superconformal surfaces in $\R^4$ have been locally parametrized in terms of minimal surfaces by Dajczer-Tojeiro \cite{DT}
and Moriya \cite{Mo}.

\begin{theorem} \label{IIS}
Let $f\colon M\to \R^4$ be an oriented, umbilic-free superconformal surface with nowhere-vanishing mean curvature vector field.
The following are equivalent: 
\begin{enumerate}[topsep=0pt,itemsep=-1pt,partopsep=1ex,parsep=0.5ex,leftmargin=*, label=(\roman*), align=left, labelsep=0em]
\item The surface $f$ is isotropically isothermic. 
\item There exists a conformal mean-directional curvature line parametrization around every point of $M$.
\item Locally, the surface $f$ admits a nontrivial infinitesimal isometric deformation that preserves,
parallelly in the normal bundle the mean curvature vector field, and the holomorphicity of  
a Gauss lift of $f$.
\end{enumerate}
\end{theorem}

Isotropically isothermic superconformal surfaces in $\R^4$ satisfying the conditions of the above theorem
can be obtained as compositions, either of superminimal surfaces in the 4-sphere 
with a stereographic projection, or of holomorphic curves in $\R^4$ with inversions.

To the best of our knowledge, the notion of isotropic isothermicity is the only generalization of 
isothermicity for surfaces in $\Q^3_c$ that allows surfaces with nonflat normal bundle. 
Moreover, apart from 
the parallel preservation in the normal bundle, 
there is no other known concept of preservation of exterior geometric data 
under infinitesimal deformations of submanifolds in codimension greater than one. 
As far as we know, this is also the first time that the aforementioned curvature lines appear in a problem that is not related exclusively to their own interest.

Transformations of isotropically isothermic surfaces is the subject of a forthcoming paper.

The rest of our results concern the Bonnet problem. For an isometric immersion 
$f\colon M\to \Q^4_c$, we denote by $\mathcal{M}(f)$ the moduli space of congruence classes
of all isometric immersions of $M$ into $\Q^4_c$ that have the same mean curvature with $f$. 
Every nontrivial class in $\mathcal{M}(f)$ is called a {\emph{Bonnet mate}} of $f$, and 
the surface $f$ is called {\emph{proper Bonnet}} if it admits infinitely many Bonnet mates.
The structure of the moduli space for compact surfaces has been studied in \cite{PV}.
The following result determines the possible structure of $\mathcal{M}(f)$ for simply-connected surfaces.

\begin{theorem} \label{MS}
Let $f\colon M\to \Q^4_c$ be a simply-connected oriented surface.
\begin{enumerate}[topsep=0pt,itemsep=-1pt,partopsep=1ex,parsep=0.5ex,leftmargin=*, label=(\roman*), align=left, labelsep=-0.3em]
\item If $f$ is not proper Bonnet, then it admits either at most one Bonnet mate, or exactly three.
\item If $f$ is proper Bonnet, then the moduli space $\mathcal{M}(f)$ is a space diffeomorphic to a manifold.
Moreover, $f$ is characterized according to the structure of $\mathcal{M}(f)$ as follows:
\begin{description}[leftmargin=2cm, style=nextline]
\item[Tight:] The moduli space is 1-dimensional with at most two connected components, each one diffeomorphic to $\mathbb{S}^1\simeq \R/2\pi \mathbb{Z}$.
\item[Flexible:] The moduli space is diffeomorphic to the torus $\mathbb{S}^1\times\mathbb{S}^1$.
\end{description}
\end{enumerate}
In particular, $f$ admits at most one Bonnet mate if $M$ is homeomorphic to $\mathbb{S}^2$.
\end{theorem}

It has been proved in \cite{PV} 
that simply-connected
surfaces in $\Q^4_c$ with a vertically harmonic Gauss lift, which are neither minimal, nor superconformal, are proper Bonnet.
In particular, it was shown that  
non-minimal surfaces with parallel mean curvature vector field which are not totally umbilical, are flexible.
Surfaces with nonvanishing parallel mean curvature vector field lie as constant mean curvature surfaces in some totally umbilical hypersurface of  
$\Q^4_c$ (cf. \cite{Chen, Yau}).

The following theorem implies that there exist flexible proper Bonnet surfaces in $\Q^4_c$ 
that do not lie in any totally umbilical hypersurface. 
Such surfaces in our result arise as Bonnet mates of surfaces in $\Q^4_c$, which are given by the composition
of a proper Bonnet surface with non-constant mean curvature in $\Q^3_c$ with a totally geodesic inclusion. 
The following theorem also shows that the simply-connected Bonnet pairs in $\Q^3_c$,
give rise to Bonnet quadruples in $\Q^4_c$.

\begin{theorem} \label{Q3}
Let $f\colon M\to \Q^4_c$ be a simply-connected oriented surface, which is the composition of 
a non-minimal Bonnet surface $F\colon M\to\Q^3_c$ with a totally geodesic inclusion. 
Every Bonnet mate of $F$ in $\Q^3_c$ determines two Bonnet mates $f^-$ and $f^+$ of $f$ in $\Q^4_c$, that do not lie in any totally geodesic hypersurface. 
The surface $f^\pm$ lies in some totally umbilical  
hypersurface of $\Q^4_c$ if and only if $F$ has constant mean curvature.
Moreover, either $f$ admits exactly three Bonnet mates, or it is a flexible proper Bonnet surface.
\end{theorem}

We show that proper Bonnet surfaces are isotropically isothermic away from isolated points, and that
strong isotropic isothermicity characterizes the flexible surfaces away from their isolated pseudo-umbilic points.
In particular, the umbilic-free flexible surfaces obtained by the above theorem are furthermore isothermic.
We also prove a result analogous to that of Graustein \cite{Grau}, which implies 
that a simply-connected, Bonnet and strongly isotropically isothermic surface is proper Bonnet. 
This result indicates that the most natural class to look for simply-connected Bonnet surfaces which are not proper Bonnet,
is that of {\emph{half or strongly totally non isotropically isothermic surfaces}}, 
that are surfaces whose 
either at least one, or both of mixed connection forms, respectively, are everywhere defined and nowhere co-closed.

In the sequel we deal with compact surfaces.
It has been proved in \cite{PV} that compact surfaces in $\Q^4_c$ whose both Gauss lifts are not vertically harmonic, admit at most three Bonnet mates. 
The following theorem shows that for such surfaces, and in contrast to the simply-connected case,
additional assumptions involving isotropic isothermicity are restrictive for the existence of Bonnet mates. 
It is inspired by a recent result of Jensen-Musso-Nicolodi \cite{JMN} for surfaces in $\R^3$.

\begin{theorem}\label{QIC}
Let $f\colon M\to \Q^4_c$ be a compact oriented surface whose both Gauss lifts are not vertically harmonic.
If $f$ is either isotropically isothermic, or half totally non isotropically isothermic, on an open dense and connected subset 
of $M$, then it admits at most one Bonnet mate. 
In particular, $f$ does not admit any Bonnet mate, if it is either strongly isotropically isothermic, 
or strongly totally non isotropically isothermic, on such a subset of $M$. 
\end{theorem}

Thereafter, we study locally proper Bonnet surfaces.
A surface $f\colon M\to \Q^4_c$ is called locally proper Bonnet if every point of $M$ has a neighbourhood, restricted to which $f$ is proper Bonnet. 
If such a surface is non-minimal, then for any sufficiently small neighbourhood $U$ of every $p\in M$, 
there exists a submanifold $L^n(p)$, $1\leq n\leq2$, of the torus $\mathbb{S}^1\times\mathbb{S}^1$,
that is also a submanifold of the moduli space $\mathcal{M}(f|_U)$. 
The surface $f$ is called {\emph{uniformly locally proper Bonnet}} if there exists a submanifold $L^n$, $1\leq n\leq2$, of the torus,
having the above property for every $p\in M$.
In particular, if this submanifold is the torus itself, then $f$ is called {\emph{locally flexible}}.

The following results concern compact surfaces that are locally proper Bonnet.
A basic ingredient of their proofs is an index theorem that we obtain using the mixed connection forms,
which extends the Poincar\'e-Hopf index theorem for surfaces in $\Q^3_c$ with isolated umbilics.
The following theorem characterizes compact surfaces with a vertically harmonic Gauss lift that are neither minimal,
nor superconformal, as the only compact, uniformly locally proper Bonnet surfaces in $\Q^4_c$.

\begin{theorem} \label{ULPB}
Let $f\colon M\to \Q^4_c$ be a non-minimal, compact oriented surface. 
The surface $f$ is uniformly locally proper Bonnet if and only if it has a vertically harmonic and non-holomorphic Gauss lift.
\end{theorem}

Our next result concerns superconformal surfaces. We mention that Fujioka \cite{Fu} found a class of simply-connected surfaces 
with nonflat normal bundle in the hyperbolic 4-space, that can be deformed by preserving the length of the mean curvature vector field. 
A careful look on the conditions that he imposed in order to obtain this class, shows that these 
surfaces are superconformal
and proper Bonnet in our sense. For compact 
surfaces, we prove the following.

\begin{theorem}\label{SPB}
There do not exist compact oriented superconformal surfaces in $\Q^4_c$ that are locally proper Bonnet.
\end{theorem}

The following theorem shows that 
the compact, locally flexible proper Bonnet surfaces in $\Q^4_c$ have parallel mean curvature vector field.
From \cite{Chen, Yau}, it follows that such a surface lies as a constant mean curvature surface 
in some totally umbilical hypersurface of $\Q^4_c$. 
Jointly with Theorem \ref{Q3}, this gives a strong generalization of a result due to Umehara \cite{U}.

\begin{theorem}\label{LFPB}
A compact oriented surface $f\colon M\to \Q^4_c$ is locally flexible proper Bonnet if and only if
it has nonvanishing parallel mean curvature vector field, and  $genus(M)>0$.
\end{theorem}

The paper is organized as follows: In Section \ref{s2}, we fix the notation and we give some preliminaries.
In Section \ref{s3}, we introduce the mixed connection forms of surfaces in $\Q^4_c$ and 
we prove an index theorem
that will be used for the proofs of Theorems \ref{ULPB}-\ref{LFPB}. 
We also provide some applications, among them, a short proof of a result due to Asperti \cite{As}. 
In Section \ref{s4}, we introduce the concept of isotropic isothermicity, we prove that it is a conformally invariant property, and we give some examples. 
We also investigate its relation with isothermicity and with lines of curvature. 
The last part of the section concerns infinitesimal isometric deformations, and there 
we prove Theorems \ref{IIID} and \ref{IIS}.
In Section \ref{s5}, we set up the framework for the study of the Bonnet problem.
Section \ref{s6} is devoted to simply-connected surfaces. We prove a theorem that provides detailed information
about the structure of the moduli space, and we give the proofs of Theorems \ref{MS} and \ref{Q3}.
In the last part of the section, we study proper Bonnet surfaces and we prove that they are isotropically isothermic. 
We also show that such surfaces 
admit conformal
metrics of constant curvature $-1$, away from points 
at which some Gauss lift is vertically harmonic. 
Section \ref{s7} deals with compact surfaces. We investigate the effect of isotropic isothermicity
on the structure of the moduli space and we give the proof of Theorem \ref{QIC}. 
Finally, we study locally proper Bonnet surfaces and we prove Theorems \ref{ULPB}, \ref{SPB} and \ref{LFPB}.

\medskip

\noindent{\bf{Acknowledgements:}} A large part of this work is also part of the Ph.D. thesis of the author,
accomplished with financial support of the Alexander S. Onassis Public Benefit Foundation.
The author would like to thank his supervisor, Professor Theodoros Vlachos, for his constant encouragement and for 
valuable comments and remarks.


\section{Preliminaries}\label{s2}

Throughout the paper, $M$ is a connected, oriented 2-dimensional Riemannian manifold.
A surface $f\colon M\to \Q_c^n$, $n=3,4,$ is an isometric immersion into the complete, 
simply-connected $n$-dimensional space form of curvature $c$.

Let $f\colon M\to \Q_c^4$ be a surface. Denote by $N_fM$ the normal bundle of $f$ and by 
$\nap, R^\perp$ the normal connection and its curvature tensor, respectively.
The orientations of $M$ and $\Q^4_c$ induce an orientation on the normal bundle of $f$.
The \emph{normal curvature} $K_N$ of $f$ is given by
$K_N=\<R^\perp(e_1,e_2)e_4,e_3\>$, 
where $\{e_1, e_2\}$ and $\{e_3, e_4\}$ are positively oriented orthonormal frame fields of $TM$ and $N_fM$, respectively,
and $\<\cdot,\cdot\>$ stands for the Riemannian metric of $\Q_c^4$.
Notice that if $\tau$ is an orientation-reversing isometry of $\Q_c^4$, then $f$ and $\tau\circ f$ have opposite normal curvatures.
The Gaussian curvature $K$ of $M$ and the normal curvature satisfy the equations
\be \label{normcf}
d\w_{12}=-K \w_1\wedge\w_2,\;\;\;\;d\w_{34}=-K_N \w_1\wedge\w_2,
\ee
where $\{\w_k\}$ is the dual frame field of $\{e_k\}$, $1\leq k\leq4$, and its corresponding connection forms $\w_{kl}= -\w_{lk},\; 1\leq k,l\leq 4$, are given by
\be \label{connection forms}
d\w_k=\sum_{m=1}^{4}\w_{km}\wedge \w_m,\;\;\; 1\leq k \leq 4.
\ee
If $M$ is compact, the Euler-Poincar{\'e} characteristics $\mathcal{\chi}, \mathcal{\chi}_N$ of $TM$ and
$N_fM$, are respectively given by
\begin{equation*} \label{char}
\mathcal{\chi}=\frac{1}{2\pi}\int_{M}K,\;\;\;\;\;\mathcal{\chi}_N=\frac{1}{2\pi}\int_{M}K_N.
\end{equation*}

Let $\a\colon TM\times TM\to N_fM$ be the second fundamental form of $f$.
The shape operator $A_{\xi}$ of $f$ with respect to $\xi \in N_fM$ is the symmetric endomorphism
of $TM$ defined by $\<A_{\xi}X,Y\>=\<\a(X,Y),\xi\>$. The surface $f$ is said to have {\emph{flat normal bundle}} if $K_N\equiv0$ on $M$. 
This is equivalent to the existence for every $p\in M$,
of an orthonormal basis of $T_pM$ that diagonalizes  
all shape operators of $f$ at $p$.

The \emph{curvature ellipse of} $f$ at each $p\in M$
is defined by  $${\cal E}_{f}(p)=\left\{\a(X,X):X\in T_pM, \|X\|=1\right\}.$$
It is indeed an ellipse on $N_{f}M(p)$ centered at the mean curvature vector $H(p)=\mbox{trace}\a(p)/2$, which may degenerate into a line segment or a point.
It is parametrized by
\bea
\a(X_{\theta},X_{\theta})=H(p)+\cos 2\theta \frac{(\a_{11}-\a_{22})}{2}+\sin 2\theta \a_{12},
\eea
where $X_{\theta}=\cos \theta e_1 +\sin \theta e_2$, $\a_{kl}=\a(e_k,e_l)$, $k,l=1,2$, and $\{e_1,e_2\}$ is an orthonormal basis of $T_pM$.
The ellipse degenerates into a line segment or a point if and only if the
vectors $(\a_{11}-\a_{22})/2$ and $\a_{12}$ are linearly dependent, or equivalently, if $R^{\perp}=0$ at $p$ (cf. \cite{GR}).
Moreover, at a point where the curvature ellipse is nondegenerate, $K_N$ is positive if and only if
the orientation induced on the ellipse as $X_{\theta}$ traverses positively the unit tangent 
circle, coincides with the orientation of the normal plane. 
The lengths $\lambda_1,\lambda_2$ of the semiaxes of ${\cal E}_{f}$,
satisfy at any point the relations (cf. \cite{Little})
\be \label{axes}
\lambda_1^2+\lambda_2^2=\|H\|^2-(K-c),\;\;\ \lambda_1 \lambda_2=\frac{1}{\pi}A({\cal E}_{f})=\frac{1}{2}|K_N|,
\ee
where $A({\cal E}_{f})$ is the area of the curvature ellipse.
Therefore, at every point of $M$ we have that
$$\|H\|^2-(K-c)\geq|K_N|.$$
A point $p\in M$ is called \emph{pseudo-umbilic} if the curvature ellipse is a circle at $p$,
and the set $M_0(f)$ of pseudo-umbilic points of $f$ is characterized as
$$M_0(f)=\left\{p\in M: \|H\|^2-(K-c)=|K_N|\right\}.$$
A surface for which any point is pseudo-umbilic is called \emph{superconformal}.
A pseudo-umbilic point is called \emph{umbilic} if the circle degenerates into a point.
By setting $$M_0^{\pm}(f)=\{p\in M_0(f):\pm K_N\geq0 \},$$ 
it follows that $M_0(f)=M_0^{+}(f)\cup M_0^{-}(f)$, and that the set $M_1(f)$ of umbilic points is 
$$M_1(f)=M_0^{+}(f)\cap M_0^{-}(f)=\{p\in M: \|H\|^2=K-c\}.$$

\subsection{Complexification and Associated Differentials}

The complexified tangent bundle $TM\otimes\mathbb{C}$ of a 2-dimensional oriented Riemannian manifold $M$, decomposes 
into the eigenspaces of the complex structure $J$, denoted by $T^{(1,0)}M$ and $T^{(0,1)}M$, corresponding to the eigenvalues
$i$ and $-i$, respectively. 

The second fundamental form of a surface $f\colon M\to \Q_c^4$ can be 
$\C$-bilinearly extended to $TM\otimes\mathbb{C}$ with values in the complexified
normal bundle $N_fM\otimes\mathbb{C}$, and then decomposed into its $(k,l)$-components $\a^{(k,l)}$, $k+l=2$,
which are tensors of $k$ many 1-forms vanishing on $T^{(0,1)}M$ and $l$ many 1-forms vanishing on $T^{(1,0)}M$.
For a positively oriented local orthonormal frame field
$\{e_1,e_2\}$ of $TM$, the \emph{Hopf invariant} $\mathcal{H}(e_1,e_2)$ of $f$ with respect to $\{e_1,e_2\}$ is the local section of $N_fM\otimes\C$ defined by
\be \label{Hopf I}
\mathcal{H}(e_1,e_2)=2\a^{(2,0)}(e_1,e_1) 
=\frac{\a_{11}-\a_{22}}{2}-i\a_{12},\;\;\;\a_{kl}=\a(e_k,e_l),\; k,l=1,2.
\ee

Let $\Jp$ be the complex structure of $N_fM$ defined by the metric and the orientation.
The complexified normal bundle decomposes as 
$$N_fM\otimes\mathbb{C}=N_f^{-}M\oplus N_f^{+}M$$ 
into the eigenspaces $N_f^{-}M$ and $N_f^{+}M$ of $\Jp$, corresponding to the eigenvalues $i$ and $-i$, respectively. 
Any section $\xi \in N_fM\otimes\mathbb{C}$ is decomposed as $\xi=\xi^{-}+\xi^{+}$, where
$$\xi^{\pm}=\pi^{\pm}(\xi),$$
and the projection $\pi^{\pm}\colon N_fM\otimes\mathbb{C}\to N_f^{\pm}M$ is given by
$$\pi^{\pm}(\xi)=\frac{1}{2}(\xi\pm i\Jp \xi),\;\;\; \xi \in N_fM\otimes\mathbb{C}.$$
A section $\xi$ of $N_fM\otimes\mathbb{C}$ is called \emph{isotropic} if at any point of $M$, either
$\xi=\xi^{-}$, or $\xi=\xi^{+}$. This is equivalent to $\langle \xi,\xi\rangle =0$, where
$\langle \cdot,\cdot\rangle$ is the $\C$-bilinear extension of the metric.
Notice that $\langle \zeta, \eta \rangle =0$ for $\zeta \in N^{-}_fM$ and $\eta\in N^{+}_fM$, implies that either $\zeta=0$, or $\eta=0$.
According to the above decomposition, the Hopf invariant of $f$ with respect to $\{e_1,e_2\}$ 
splits into isotropic parts as $\mathcal{H}(e_1,e_2)=\mathcal{H}^-(e_1,e_2)+\mathcal{H}^+(e_1,e_2)$, where 
\begin{eqnarray}
\mathcal{H}^{\pm}(e_1,e_2)=\frac{1}{2}\left(\frac{\a_{11}-\a_{22}}{2}\pm \Jp \a_{12}\pm i \Jp\left(\frac{\a_{11}-\a_{22}}{2}\pm \Jp \a_{12}\right)\right).
\label{Hopf Invariants}
\end{eqnarray}
The length of $\mathcal{H}^{\pm}(e_1,e_2)$ is independent of the frame field $\{e_1,e_2\}$, and the function 
$\|\mathcal{H}^{\pm}\|$ given by
\be \label{Bpm}
\|\mathcal{H}^{\pm}\|= \sqrt{2}\left\|\mathcal{H}^{\pm}(e_1,e_2)\right\|=\sqrt{\|H\|^2-(K-c)\mp K_N}
\ee
vanishes precisely on $M^{\pm}_0(f)$.

Let $E$ be a complex vector bundle over $M$ equipped with a connection $\n^E$. An \emph{$E$-valued differential $\Psi$ of $r$-order}
is an $E$-valued $r$-covariant tensor field on $M$ of holomorphic type $(r,0)$. The $r$-differential $\Psi$
is called \emph{holomorphic} (cf. \cite{BWW}) if its covariant derivative $\n^E \Psi$ has holomorphic type $(r+1,0)$. 
Let $(U,z=x+iy)$ be 
a local complex coordinate on $M$. The Wirtinger operators are defined on $U$ by $\d=\d_z=(\d_x-i\d_y)/2$, 
$\bar{\d}=\d_{\bar z}=(\d_x+i\d_y)/2$, where $\d_x=\d/\d x$ and $\d_y=\d/\d y$.
On $U$, the differential $\Psi$ has the form $\Psi=\psi dz^r$, where $\psi \colon U\to E$ is given by $\psi=\Psi(\dz,\dots,\dz)$.
Then, $\Psi$ is holomorphic if and only if $$\n^E_{\dzb}\psi=0,$$
i.e., $\psi$ is a holomorphic local section. For later use we need the following result (cf. \cite{Ch,BWW}).

\begin{lemma}\label{zeros}
Assume that the $E$-valued differential $\Psi$ is holomorphic and let $p\in M$ be such that $\Psi(p)=0$.
Let $(U,z)$ be a local complex coordinate with $z(p)=0$. Then either $\Psi \equiv 0$ on $U$; or $\Psi=z^m\Psi^{*}$,
where $m$ is a positive integer and $\Psi^{*}(p)\neq0$.
\end{lemma}

Let $f\colon M\to \Q^4_c$ be an oriented surface.
In terms of a local complex coordinate $(U,z=x+iy)$,
the metric $ds^2$ of $M$ is written as $ds^2=\lambda^2|dz|^2$,
where $\lambda > 0$ is the conformal factor.
Setting $e_1=\d_x/\lambda$ and  $e_2=\d_y/\lambda$, the components of $\a$ are given by
$$\a^{(2,0)} = \a(\dz,\dz)dz^2,\;\; \a^{(0,2)} = \overline {\a^{(2,0)}},\;\; \a^{(1,1)} = \a(\dz,\dzb)(dz\otimes d\bar{z} + d\bar{z}\otimes dz),$$
where
\begin{eqnarray} \label{defadd}
\a(\dz,\dz) =\frac{\lambda^2}{2}\mathcal{H}(e_1,e_2)\;\;\; \mbox{and}\;\;\; \a(\dz,\dzb)= \frac{\lambda^2}{2}H.
\end{eqnarray}
The {\emph{Hopf differential of $f$}} is the quadratic $N_fM\otimes\mathbb{C}$-valued differential $\Phi=\a^{(2,0)}$ with local expression
$\Phi=\add dz^2$.
According to the decomposition of $N_fM\otimes\mathbb{C}$, the Hopf differential splits into isotropic parts as
\begin{equation*}
\Phi=\Phi^{-}+\Phi^{+},\;\;\mbox{where}\;\; \Phi^{\pm}=\pi^{\pm}\circ\Phi.
\end{equation*}
The following has been proved in \cite[Lemma 8]{PV}.
\begin{lemma}\label{pseudo}
\begin{enumerate}[topsep=0pt,itemsep=-1pt,partopsep=1ex,parsep=0.5ex,leftmargin=*, label=(\roman*), align=left, labelsep=-0.4em]
\item The zero-sets of $\Phi^{\pm}$ and $\Phi$, are $M_{0}^{\pm}(f)$ and $M_1(f)$, respectively.
\item The surface $f$ is superconformal with normal curvature $\pm K_N\geq0$ if and only if $\Phi^{\pm}\equiv0$. 
In particular, if $f$ is superconformal, then $K_N$ vanishes precisely on $M_1(f)$.
\end{enumerate}
\end{lemma}

On $(U,z)$ the differential $\Phi^{\pm}$ has the expression
\be \label{phipm}
\Phi^{\pm}=\phi^{\pm}dz^2,
\ee
and the compatibility equations for $f$ can be written as
\begin{eqnarray}
&\mbox{(Gauss)}&  \;\;\; (\log\lambda^2)_{z\bar z}-\frac{2}{\lambda^2}\left(\langle\phi^-,\overline{\phi^-}\rangle +\langle\phi^+,\overline{\phi^+}\rangle\right)
+\frac{\lambda^2}{2}(\|H\|^2+c)=0,\;\;\;\label{Gauss} \\
&\mbox{(Codazzi)}&  \;\;\; \nap_{\dzb}\phi^-=\frac{\lambda^2}{2}\nap_{\d}H^-,\;\;\; \nap_{\dzb}\phi^+=\frac{\lambda^2}{2}\nap_{\d}H^+,\;\;\;\label{Codazzi} \\
&\mbox{(Ricci)}&  \;\;\; R^{\perp}(\d,\dzb)=\frac{2}{\lambda^2}(\phi^-\wedge\overline{\phi^-}+\phi^+\wedge\overline{\phi^+}),\;\;\;\label{Ricci}
\end{eqnarray}
where $R^{\perp}$ is the $\C$-trilinear extension of the normal curvature tensor 
and $(\xi\wedge \zeta)\eta=\<\zeta,\eta\>\xi-\<\xi,\eta\>\zeta$, for $\xi,\zeta,\eta \in N_fM\otimes\C$. 
It follows from \eqref{phipm} and (\ref{Codazzi}) that $\Phi$ 
is holomorphic if and only if the mean curvature vector field $H$ is parallel in the normal connection.

\subsection{Twistor Spaces and Gauss Lifts}

Let $f\colon M\to \R^4$ be an oriented surface. We recall that (see for instance \cite[Section 4.2]{PV})
the Grassmannian $Gr(2,4)$ of oriented 2-planes in $\R^4$,
is isometric to the product $\mathbb{S}^2_+\times \mathbb{S}^2_-$ of two spheres of radius $1/\sqrt{2}$. 
Accordingly, the Gauss map $g\colon M\to Gr(2,4)$ of $f$, decomposes into a pair of maps as 
$g=(g_+,g_-)\colon M\to \mathbb{S}^2_+\times \mathbb{S}^2_-$. For surfaces in not necessarily flat space forms $\Q^4_c$, the geometric information
encoded in the components $g_+$ and $g_-$ of the Gauss map of a surface in $\R^4$, is encoded in the Gauss lifts of the surface to the twistor bundle of $\Q^4_c$.

We briefly recall some facts about the twistor theory of 4-dimensional space forms (cf. \cite{ES, Fr, JR}).
The twistor bundle ${\mathcal Z}$ of $\Q_c^4$ is the set of all pairs $(p,\tilde{J})$, where $p\in \Q_c^4$
and $\tilde{J}$ is an orthogonal complex structure on $T_p\Q_c^4$, endowed with the twistor projection $\varrho \colon {\mathcal Z}\to \Q_c^4$,
defined by $\varrho (p,\tilde{J})=p$. The twistor bundle is a $O(4)/U(2)$-fiber bundle over $\Q_c^4$ associated to
$O(\Q_c^4)$, the principal $O(4)$-bundle of orthonormal frames in $\Q_c^4$, which has two connected components.
More precisely, at a point $p\in \Q_c^4$, any orthonormal frame $e=(e_1,e_2,e_3,e_4)$ of $T_p\Q_c^4$
determines an orthogonal complex structure $\tilde{J}_e$, given by $$\tilde{J}_e e_1=e_2,\ \tilde{J}_e e_3=e_4,\ \tilde{J}_e^2=-\id.$$
Every orthogonal complex structure on $T_p\Q_c^4$ 
can be written in the above form for some orthonormal frame of $T_p\Q_c^4$. 
In particular, $\tilde{J}_e=\tilde{J}_{\tilde{e}}$ if and only if $\tilde{e}=eA$ for some $A\in U(2)$. 
Therefore, the set of all
orthogonal complex structures on $T_p\Q_c^4$ is $O(4)/U(2)$ and has two connected components diffeomorphic to
$SO(4)/U(2)=\{\tilde{J}_e: e\ \mbox{is a}\pm \mbox{oriented frame of}\ T_p\Q_c^4\}$. Hence, the twistor bundle is
$${\mathcal Z}=O(\Q_c^4)\times_{O(4)}O(4)/U(2)=O(\Q_c^4)/U(2)$$
and its two connected components are denoted by ${\mathcal Z}_{+}$ and ${\mathcal Z}_{-}$. 
Each projection $\varrho_{\pm}\colon {\mathcal Z}_{\pm}\to \Q_c^4$ is a $\mathbb{S}^2$-fiber bundle over $\Q_c^4$,
where $\varrho_{\pm}$ is the restriction of $\varrho$ on ${\mathcal Z}_{\pm}$.

There is a one-parameter family of Riemannian metrics $g_{t},\ t>0$, defined on ${\mathcal Z}$, that make 
$\varrho_{+}$ and $\varrho_{-}$ Riemannian submersions. 
With respect to the decomposition of the tangent bundle of ${\mathcal Z}_{\pm}$ into horizontal and vertical subbundles as
$T{\mathcal Z}_{\pm}=T^h{\mathcal Z}_{\pm}\oplus T^v{\mathcal Z}_{\pm},$
the metric $g_{t}$ is given by the pull-back of the metric of $\Q_c^4$ to the
horizontal subspaces, and by adding the $t^2$-fold of the canonical metric of the fibers.

Let $Gr_2(T\Q_c^4)$ be the Grassmann bundle of oriented 2-planes tangent to $\Q_c^4$. There are projections
$\Pi_{+}\colon Gr_2(T\Q_c^4)\to {\mathcal Z}_{+}$ and $\Pi_{-}\colon Gr_2(T\Q_c^4)\to {\mathcal Z}_{-}$
defined as follows; if $\zeta \subset T_p\Q_c^4$ is an oriented
2-plane, then $\Pi_{\pm}(p,\zeta)$ is the complex structure on $T_p\Q_c^4$ corresponding to the rotation by
$+\pi/2$ on $\zeta$ and the rotation by $\pm \pi/2$ on $\zeta^{\perp}$.
The Gauss lift $G_f\colon M\to Gr_2(T\Q_c^4)$ of an oriented surface $f\colon M\to \Q^4_c$ is defined by $G_f(p)=(f(p),f_{*}T_pM)$.
The \emph{Gauss lifts of $f$ to the twistor bundle} are the maps
$$G_{+}\colon M\to {\mathcal Z}_{+}\;\; \mbox{and}\;\;G_{-}\colon M\to {\mathcal Z}_{-},\;\;\mbox{where}\;\;  G_{\pm}=\Pi_{\pm}\circ G_f.$$
At any point $p\in M$, the Gauss lift $G_\pm$ is given by $G_{\pm}(p)=(f(p),\tilde{J}_{\pm}(f(p)))$, where 
\begin{equation*}
\tilde{J}_{\pm}(f(p))= \left\{
\begin{array}{rll}
f_{*}\circ J(p), & \mbox{on}\ f_{*}T_pM,\\
\pm \Jp(p), & \mbox{on}\ N_fM(p).
\end{array}\right.
\end{equation*}

Let $\{e_k\}_{1\leq k\leq4}$ be a positively oriented, local adapted orthonormal frame field of $\Q^4_c$, where $\{e_1,e_2\}$ is in the orientation of $TM$.
Denote by $\{\w_k\}_{1\leq k\leq4}$ the corresponding coframe and by $\w_{kl},\; 1\leq k,l\leq 4$, the connection forms 
given by (\ref{connection forms}).
Locally, the pull-back of $g_t$ on $M$ under $G_{\pm}$, is related to the metric $ds^2$ of $M$ (cf. \cite{Fr,JR}) as follows 
\be \label{Gconf}
G_{\pm}^{*}(g_t)=ds^2 + \frac{t^2}{4}\left((\w_{13}\mp\w_{24})^2 + (\w_{23}\pm\w_{14})^2\right).
\ee
The Gauss lift $G_{\pm}\colon M\to ({\mathcal Z}_{\pm},g_t)$ is called \emph{conformal} if its induced metric $G_{\pm}^{*}(g_t)$ is conformal to $ds^2$,
and is called \emph{isometric} if $G_{\pm}^{*}(g_t)=ds^2$.
The following has been proved in \cite[Prop. 8.2]{JR}.

\begin{proposition} \label{conformal Gl}
Let $f\colon M\to \Q^4_c$ be an oriented surface. The Gauss lift $G_{\pm}\colon M\to ({\mathcal Z}_{\pm},g_t)$ of $f$ is
either conformal, or isometric, if and only if either (i), or (ii), respectively, holds:
\begin{enumerate}[topsep=0pt,itemsep=-1pt,partopsep=1ex,parsep=0.5ex,leftmargin=*, label=(\roman*), align=left, labelsep=0em]
\item The surface $f$ is either minimal, or superconformal with normal curvature $\pm K_N\geq0$.
\item The surface $f$ is minimal and superconformal with normal curvature $\pm K_N\geq0$.
\end{enumerate}
\end{proposition}

Adopting the notation of \cite{JR}, there exists an almost complex structure 
$\mathcal{J}_+$ on $\mathcal Z$, that makes $({\mathcal Z}_{\pm},g_t)$ a Hermitian manifold.
The Gauss lift $G_{\pm}\colon M\to ({\mathcal Z}_{\pm},g_t)$ is called \emph{holomorphic} if it is holomorphic with respect to $\mathcal{J}_+$.
The following has been proved in \cite[Prop. 8.1]{JR}.

\begin{proposition} \label{holomorphic Gl}
Let $f\colon M\to \Q^4_c$ be an oriented surface. The Gauss lift $G_{\pm}\colon M\to ({\mathcal Z}_{\pm},g_t)$ 
of $f$ is holomorphic if and only if $f$ is superconformal
with normal curvature $\pm K_N\geq0$.
\end{proposition}

Immediate consequence of Propositions \ref{conformal Gl} and \ref{holomorphic Gl} is the following.

\begin{proposition} \label{CH Gl}
Let $f\colon M\to \Q^4_c$ be an oriented surface with nowhere-vanishing mean curvature vector field. 
The Gauss lift $G_{\pm}\colon M\to ({\mathcal Z}_{\pm},g_t)$ of $f$ is holomorphic if and only if it is conformal.
\end{proposition}

The Gauss lift $G_{\pm}\colon M\to ({\mathcal Z}_{\pm},g_t)$ is called \emph{vertically harmonic} if its tension field has
vanishing vertical component with respect to the decomposition $T{\mathcal Z}_{\pm}=T^h{\mathcal Z}_{\pm}\oplus T^v{\mathcal Z}_{\pm}$.
The following has been proved in \cite[Prop. 9]{PV}.

\begin{proposition}\label{glphi}
Let $f\colon M \to \Q^4_c$ be an oriented surface with mean curvature vector field $H$. The following are equivalent:
\begin{enumerate}[topsep=0pt,itemsep=-1pt,partopsep=1ex,parsep=0.5ex,leftmargin=*, label=(\roman*), align=left, labelsep=0em]
\item The Gauss lift $G_{\pm}\colon M\to ({\mathcal Z}_{\pm},g_t)$ of $f$ is vertically harmonic.
\item The differential $\Phi^{\pm}$ is holomorphic.
\item The section $H^{\pm}$ is anti-holomorphic.
\item $\nap_{JX}H=\pm\Jp \nap_{X}H$, for any $X\in TM$.
\end{enumerate}
\end{proposition}

For later use, we need the following consequence of Theorem 8.1. in \cite{JR}.
Notice that for a local orthonormal frame field $\{e_3,e_4\}$ of $N_fM$, the covariant differential of the mean curvature vector field $H=H^3e_3+H^4e_4$ is given by
\begin{eqnarray}
\nap H &=& \sum_{a=3}^{4}\big(dH^a + \sum_{b=3}^{4}H^b \w_{ba}\big)\otimes e_a = \sum_{a=3}^{4}\sum_{j=1}^{2}H^a_j \w_{j}\otimes e_a. \label{Hij}
\end{eqnarray}

\begin{proposition}\label{tension field}
Let $f\colon M \to \Q^4_c$ be an oriented surface. The squared length of the vertical component $\tau^v(G_{\pm})$ of the tension field of 
the Gauss lift $G_{\pm}\colon M\to ({\mathcal Z}_{\pm},g_1)$ of $f$, is given by
\bea
\|\tau^v(G_{\pm})\|^2=4\left((H^{3}_1 \mp H^{4}_2)^2+(H^{3}_2\pm H^{4}_1)^2\right),
\eea
where $\{e_1,e_2\}$ and $\{e_3,e_4\}$ are positively oriented local orthonormal frame fields of $TM$ and $N_fM$, respectively,
and $H^{a}_j$, $j=1,2,\ a=3,4$, is given by \eqref{Hij}.
\end{proposition}

\begin{proof}
It follows immediately from the proof of \cite[Thm. 8.1]{JR}, where the components of the tension field of $G_{\pm}$ have been computed
(see also the proof of \cite[Prop. 9]{PV}).
\qed
\end{proof} 
\medskip

\begin{remark}
{\emph{
\begin{enumerate}[topsep=0pt,itemsep=-1pt,partopsep=1ex,parsep=0.5ex,leftmargin=*, label=(\roman*), align=left, labelsep=0em]
\item Proposition \ref{glphi} and Lemma \ref{pseudo}(ii) imply that any superconformal surface $f\colon M\to \Q^4_c$
with $\pm K_N\geq0$ has vertically harmonic Gauss lift $G_{\pm}$. 
\item From Proposition \ref{glphi} it follows that both Gauss lifts are vertically harmonic if and only if the 
mean curvature vector field of the surface is parallel in the normal connection.
\item In the case of $\R^4$, $({\mathcal Z}_{\pm},g_t)$ is isometric to the product $\R^4 \times \mathbb{S}^2(t)$. The Grassmann bundle is trivial
$Gr_2(\R^4)\simeq \R^4 \times Gr(2,4)$ and the Gauss lift of $f$ to the Grassmann bundle is given by $G_f=(f,g)$, where
$g=(g_{+},g_{-})\colon M \to \mathbb{S}^2_{+} \times \mathbb{S}^2_{-}$ is the Gauss map of $f$. 
The Gauss lift $G_{\pm}$ of $f$ to the twistor bundle is then given by $G_{\pm}=(f,\sqrt2tg_{\pm})$ and it is vertically harmonic if and only if
$g_{\pm}$ is harmonic.
\item Lagrangian surfaces in $\R^4$ with conformal or harmonic Maslov form, constitute examples of surfaces with 
the component $g_+$ or $g_-$, respectively, harmonic (cf. \cite{CU}).
\end{enumerate}
}}\end{remark}

\section{The Mixed Connection Forms of Surfaces in $\Q^4_c$} \label{s3}

Let $f\colon M\to \Q^4_c$ be an oriented surface with $M_0^{\pm}(f)$ isolated, and consider 
a positively oriented local orthonormal frame field $\{e_1,e_2\}$ of $TM$ defined on an open $U\subset M\smallsetminus M_0^{\pm}(f)$.
By virtue of \eqref{Hopf Invariants} and \eqref{Bpm}, 
the frame field $\{e_1,e_2\}$ determines a unique orthonormal frame field $\{e^{\pm}_3,e^{\pm}_4\}$ of $N_fU$ such that
\be \label{HI}
\mathcal{H}^{\pm}(e_1,e_2)=\frac{1}{2}\|\mathcal{H}^{\pm}\|(e^{\pm}_3\pm ie^{\pm}_4),
\ee
where
\be \label{e34pm}
e^{\pm}_3=\|\mathcal{H}^{\pm}\|^{-1}\left(\frac{\a_{11}-\a_{22}}{2}\pm \Jp \a_{12}\right),\;\;\; e^{\pm}_4=\Jp e^{\pm}_3,
\ee
and $\a_{kl}=\a(e_k,e_l)$, $k,l=1,2$.
Define the 1-form $\Omega^{\pm}(e_1,e_2)$ on $U$ by
\be \label{Om}
\Omega^{\pm}(e_1,e_2)=2\w_{12} \pm \w_{34}^{\pm},
\ee
where the connection forms $\w_{12}$ and $\w_{34}^{\pm}$, correspond to the dual frame field of 
$\{e_1,e_2,e_3^{\pm},e_4^{\pm}\}$ and are given by \eqref{connection forms}.
The following proposition shows that $\Omega^{\pm}(e_1,e_2)$ is independent of the frame field $\{e_1,e_2\}$ and thus, 
well-defined on $M\smallsetminus M_0^{\pm}(f)$.

\begin{proposition} \label{Criterion}
Let $f\colon M\to \Q^4_c$ be an oriented surface. If $M_0^{\pm}(f)$ is isolated, then:
\begin{enumerate}[topsep=0pt,itemsep=-1pt,partopsep=1ex,parsep=0.5ex,leftmargin=*, label=(\roman*), align=left, labelsep=0em]
\item There exists a 1-form $\Omega^{\pm}$ on $M\smallsetminus M_0^{\pm}(f)$ such that
\be \label{wres}
\Omega^{\pm}|_U=\Omega^{\pm}(e_1,e_2)
\ee
for every positively oriented local orthonormal frame field $\{e_1,e_2\}$ of $TM$, defined on an open $U\subset M\smallsetminus M_0^{\pm}(f)$.
\item The exterior derivative of $\Omega^{\pm}$ is globally defined on $M$ and satisfies
\be \label{dW}
d\Omega^{\pm}=-(2K\pm K_N)dM,
\ee
where $dM$ is the volume element of $M$. 
\item For every $p\in M_0^{\pm}(f)$, the limit
\be \label{ind}
I^{\pm}(p)=\lim_{r\to 0}\frac{1}{2\pi}\int_{S_r(p)}\Omega^{\pm}
\ee
exists, where $S_r(p)$ is a positively oriented geodesic circle of radius $r$ centered at $p$.
\end{enumerate}
\end{proposition}

\begin{proof}
(i) Let $\{e_1,e_2\}$ and $\{\tilde{e}_1,\tilde{e}_2\}$ be positively oriented
orthonormal frame fields on an open, simply-connected $U\subset M\smallsetminus M_0^{\pm}(f)$.
Consider the frame fields $\{e^{\pm}_3,e^{\pm}_4\}$ and $\{\tilde{e}^{\pm}_3,\tilde{e}^{\pm}_4\}$ of $N_fU$ determined by 
$\{e_1,e_2\}$ and $\{\tilde{e}_1,\tilde{e}_2\}$, respectively, from \eqref{HI}.
Since $U$ is simply-connected, it follows that there exists $\tau \in \mathcal{C}^{\infty}(U)$ such that
$\tilde{e}_1-i\tilde{e}_2=\exp{(i\tau)}(e_1-ie_2)$.
Moreover, from \eqref{Hopf I} and \eqref{HI} we obtain that
$\tilde{e}^{\pm}_3\pm i\tilde{e}^{\pm}_4=\exp{(2i\tau)}(e^{\pm}_3\pm ie^{\pm}_4)$.
These relations imply that
$$\tilde{\w}_{12}=\w_{12} +d\tau\;\;\;\;\mbox{and}\;\;\;\; \tilde{\w}^{\pm}_{34}=\w_{34} \mp 2d\tau.$$
Therefore, from \eqref{Om} it follows that
$$\Omega^{\pm}(\tilde{e}_1,\tilde{e}_2)=\Omega^{\pm}(e_1,e_2).$$
By virtue of the above, we define $\Omega^{\pm}$ by \eqref{wres},
for an arbitrary positively oriented orthonormal frame field $\{e_1,e_2\}$, on every simply-connected $U\subset M\smallsetminus M_0^{\pm}(f)$. 
It is clear that $\Omega^{\pm}$ is globally defined on $M\smallsetminus M_0^{\pm}(f)$,
and that \eqref{wres} also holds for frame fields defined on non-simply-connected subsets $U\subset M\smallsetminus M_0^{\pm}(f)$.

(ii) Using part (i) and \eqref{normcf}, exterior differentiation of \eqref{Om} yields that \eqref{dW} holds on $M\smallsetminus M_0^{\pm}(f)$.
Since the right-hand side of \eqref{dW} is defined globally on $M$, the proof follows.

(iii) Let $p\in M_0^{\pm}(f)$. Consider positively oriented geodesic circles $S_{r_1}(p)$ and $S_{r_2}(p)$ centered at $p$, with $r_2<r_1$,
and denote by $D$ the annular region bounded by these circles. 
Stokes' theorem yields that 
$$\int_{S_{r_1}(p)}\Omega^{\pm}-\int_{S_{r_2}(p)} \Omega^{\pm}=\int_{D}d \Omega^{\pm}.$$
Part (ii) implies that the right hand side of the above tends to zero as $r_1,r_2\to0$.
Therefore, any sequence $\int_{S_{r_n}(p)} \Omega^{\pm}$ with $r_n\to0$, is a Cauchy sequence and thus, it converges.
This completes the proof.
\qed
\end{proof}
\medskip

\begin{remark} \label{PCF}
{\emph{Let $F\colon M\to \Q^3_c$ be an umbilic-free oriented surface with shape operator $A$ and corresponding principal curvatures $k_1,k_2$, with $k_1>k_2$.
Every point of $M$ has a neighbourhood $U$ on which there exists a principal frame field $\{e_1,e_2\}$ of $F$, i.e., 
a positively oriented orthonormal frame field of $TU$ such that $Ae_l=k_le_l$, $l=1,2$. Since a principal frame field is
unique up to sign on its domain, 
there exists a 1-form $\Omega$ on $M$ such that $\Omega|_U=\w_{12}$,
where $\w_{12}$ is the connection form corresponding to the dual coframe of a principal frame field $\{e_1,e_2\}$ 
on $U\subset M$. We call $\Omega$}} the principal connection form of $F$.
\end{remark}

The following proposition shows that the mixed connection forms $\Omega^-$ and $\Omega^+$ are the natural generalizations to surfaces in 4-dimensional space forms,
of the principal connection form $\Omega$ of surfaces in 3-dimensional space forms.

\begin{proposition} \label{MCF-PCF}
Assume that $f\colon M\to \Q^4_c$ is the composition of an umbilic-free oriented surface $F\colon M\to \Q^3_{\tilde c}$, ${\tilde c}\geq c,$ 
with a totally umbilical inclusion $j\colon \Q^3_{\tilde c}\to \Q^4_c$. 
Then, $\Omega^-=\Omega^+=2\Omega$, where $\Omega$ is the principal connection form of $F$.
\end{proposition}

\begin{proof}
Let $\xi$ be the unit normal vector field of $F$ in $\Q^3_{\tilde c}$, and $A$ the shape operator of $F$ with respect to $\xi$. 
As in Remark \ref{PCF}, let $k_1,k_2$, with $k_1>k_2$ be the corresponding principal curvatures of $F$ and
consider a principal frame field $\{e_1,e_2\}$ of $F$ on $U\subset M$.
Proposition \ref{Criterion}(i) and \eqref{Om} imply that $\Omega^{\pm}|_U=\Omega^{\pm}(e_1,e_2)=2\w_{12}\pm \w_{34}^{\pm}$.
Moreover, for the second fundamental form $\alpha$ of $f$ we have that $\a_{11}-\a_{22}=(k_1-k_2)j_*\xi$ and $\a_{12}=0$,
where $\a_{kl}=\a(e_k,e_l)$, $k,l=1,2$. Then, from \eqref{e34pm} it follows that $e_3^-=e_3^+=j_*\xi$. 
Since $j_*\xi$ is parallel in the normal connection of $f$, we obtain that $\w_{34}^-=\w_{34}^+=0$. 
Then, Proposition \ref{Criterion}(i) and Remark \ref{PCF} imply that $\Omega^-|_U=\Omega^+|_U=2\Omega|_U$, and this completes the proof.
\qed
\end{proof}
\medskip

Assume that $f\colon M\to \Q^4_c$ is a surface with $M_0^{\pm}(f)$ isolated.
Proposition \ref{Criterion}(i) allows us to express locally the mixed connection form $\Omega^{\pm}$, by using \eqref{Om}
for the normalized basic vectors fields corresponding to a complex coordinate.
Let $(U,z=x+iy)$ be a local complex coordinate on $M$ and set $e_1=\d_x/\lambda$, $e_2=\d_y/\lambda$, where $\lambda > 0$ is the conformal factor.
From \eqref{defadd} and \eqref{phipm} it follows that $\phi^{\pm}=(\lambda^2/2)\mathcal{H}^{\pm}(e_1,e_2)$.
By virtue of \eqref{HI}, this implies that
\be \label{fpm}
\phi^{\pm}=\frac{\lambda^2}{4}\|\mathcal{H}^{\pm}\|(e^{\pm}_3\pm ie^{\pm}_4)\;\;\;\; \mbox{on}\;\;\;\; U\smallsetminus M_0^{\pm}(f).
\ee
The connection form $\w_{12}$ of the dual frame field of $\{e_1,e_2\}$ 
is given by $\w_{12}=\star d\log\lambda$, where $\star$ is the Hodge star operator. 
In particular, exterior differentiation gives 
$d\w_{12}=\Delta\log\lambda \w_1\wedge\w_2$,
where $\Delta=4\lambda^{-2}\d\dzb$ is the Laplacian on $M$,
and \eqref{normcf} implies that the Gaussian curvature of $M$ is given by $K=-\Delta\log\lambda$.
If $\w_{34}^{\pm}$ is the connection form of the dual frame field of $\{e_3^{\pm},e_4^{\pm}\}$,
then according to Proposition \ref{Criterion}(i), {\emph{the expression of $\Omega^{\pm}$ in terms of the complex coordinate $z$}} is 
\be \label{lambda}
\Omega^{\pm}=\star d\log\lambda^2\pm \w_{34}^{\pm}\;\;\;\; \mbox{on}\;\;\;\; U\smallsetminus M_0^{\pm}(f).
\ee

\begin{proposition} \label{IndPU}
Let $f\colon M\to \Q^4_c$ be an oriented surface with $M_0^{\pm}(f)$ isolated. 
Consider a simply-connected complex chart $(U,z)$ on $M$, with $U\cap M^{\pm}_0(f)=\{p\}$ and $z(p)=0$.
If there exists a positive integer $m$ such that the differential $\Phi^{\pm}$ satisfies
\be \label{index1}
\Phi^{\pm}=z^m\hat{\Phi}^{\pm}\;\;\; \mbox{on}\;\;\; U,\;\;\; \hat{\Phi}^{\pm}(p)\neq0,
\ee
then $I^{\pm}(p)=-m$.
\end{proposition}

\begin{proof}
Let $\Phi^{\pm}=\phi^{\pm}dz^2$ on $U$, where $\phi^{\pm}$ is given by \eqref{fpm} on $U\smallsetminus \{p\}$.
For $r>0$, consider a positively oriented geodesic circle $S_r(p)=\d B_r(p)\subset U$. 
Stokes' theorem implies that 
$\int_{S_r(p)}\star d\log\lambda=-\int_{B_r(p)}K\w_1\wedge\w_2$, and since the Gaussian curvature is bounded on $B_r(p)$, 
from Proposition \ref{Criterion}(iii) and \eqref{lambda} we obtain that
\be \label{Ww}
I^{\pm}(p)=\pm\lim_{r\to0}\frac{1}{2\pi}\int_{S_r(p)}\w_{34}^{\pm}.
\ee

Assume that $\hat{\Phi}^{\pm}$ is given by $\hat{\Phi}^{\pm}=\hat{\phi}^{\pm}dz^2$  on $U$.
Since $\hat{\phi}^{\pm}\in N^{\pm}_fU$ and $\hat{\phi}^{\pm}\neq0$ everywhere on $U$, there exist 
$R\in \mathcal{C}^{\infty}(U;(0,+\infty))$
and an orthonormal frame field $\{e_3,e_4\}$ of $N_fU$, such that $\hat{\phi}^{\pm}=R(e_3\pm ie_4)$.
Then, from \eqref{fpm} and \eqref{index1} it follows that
\be \label{winding}
\frac{\lambda^2}{2}\|\mathcal{H}^{\pm}\|(e_3^{\pm}\pm ie_4^{\pm})=z^m R(e_3\pm ie_4)\;\;\; \mbox{on}\;\;\; U\smallsetminus\{p\}.
\ee
Let $c(s), s\in [0,2\pi]$, be a parametrization of $S_r(p)$ as a simple closed curve.
There exists a smooth function $\tau(s), s\in [0,2\pi]$, such that along $c$, the frame fields $\{e_3^\pm,e_4^\pm\}$ and $\{e_3,e_4\}$
are related by
\be \label{angle tau}
e_3^{\pm}(s)\pm ie_4^{\pm}(s)=e^{\mp i\tau(s)}(e_3(s)\pm ie_4(s)).
\ee
Therefore,
\be \label{integral tau}
\frac{1}{2\pi}\int_{S_r(p)}\w_{34}^{\pm}-\frac{1}{2\pi}\int_{S_r(p)}\w_{34}=\frac{1}{2\pi}\int_{S_r(p)}d\tau.
\ee

We argue that the right hand side of \eqref{integral tau} is equal to $\mp m$. From \eqref{winding} and \eqref{angle tau}
it follows that along $c$ we have
\bea
\frac{(\lambda(s))^2\|\mathcal{H}^{\pm}\|(s)}{2R(s)}=(z(s))^me^{\pm i\tau(s)}.
\eea
Let $k(s)$ be the function at the left hand side of the above. 
Since $k(s)>0$, $s\in[0,2\pi]$, 
it follows that
\bea
\log k(s)=\log((z(s))^me^{\pm i\tau(s)}).
\eea
Differentiating the above with respect to $s$, then integrating from $0$ to $2\pi$, and taking into account that $k(0)=k(2\pi)$,
we obtain that
\bea
0=\log k(2\pi)-\log k(0)=m\int_{0}^{2\pi}\frac{z'(s)}{z(s)}ds \pm i\int_{0}^{2\pi}\tau'(s)ds,
\eea
or, equivalently
\be \label{w34m}
\frac{1}{2\pi}\int_{S_r(p)}d\tau=\mp \frac{m}{2\pi i}\int_{z(S_r(p))}\frac{dw}{w}=\mp m.
\ee

Since $\w_{34}$ is defined everywhere on $U$ and $K_N$ is bounded on $B_r(p)$, by using \eqref{normcf} we obtain that
$\lim_{r\to0}\int_{S_r(p)}\w_{34}=\lim_{r\to0}\int_{B_r(p)}d\w_{34}=-\lim_{r\to0}\int_{B_r(p)}K_N\w_1\wedge\w_2=0$.
Therefore, by taking limits in \eqref{integral tau} and using \eqref{Ww} and \eqref{w34m}, the proof follows.
\qed
\end{proof}
\medskip

\begin{theorem}\label{Hopf type}
Let $f\colon M\to \Q^4_c$ be a compact oriented surface. If $M_0^{\pm}(f)$ is isolated, then
$$2\chi \pm \chi_N=\sum_{p\in M_0^{\pm}(f)}I^{\pm}(p).$$
\end{theorem}

\begin{proof}
Assume that $M_0^{\pm}(f)\neq\emptyset$ and let $M_0^{\pm}(f)=\{p_1,\dots,p_k\}$, where $k$ is a positive integer. 
For a sufficiently small $r>0$, let $M_r=M\smallsetminus\left(B_r(p_1)\cup\dots\cup B_r(p_k)\right)$,
where $B_r(p_j)$ is the geodesic ball of radius $r$, centered at $p_j$, $j=1,\dots,k$.
Stokes' theorem implies that
\bea
\int_{M_r}d \Omega^{\pm}=-\sum_{j=1}^{k}\int_{S_r(p_j)}\Omega^{\pm},
\eea
where $\Omega^{\pm}$ is the form of Proposition \ref{Criterion}(i), and $S_r(p_j)=\d B_r(p_j)$ is positively oriented with respect to its interior.
From the above and \eqref{dW} we obtain that
\bea
2\chi \pm \chi_N=-\frac{1}{2\pi}\lim_{r\to0}\int_{M_r}d\Omega^{\pm}=\sum_{j=1}^{k}\frac{1}{2\pi}\lim_{r\to0}\int_{S_r(p_j)}\Omega^{\pm}
\eea
and the proof follows from \eqref{ind}. If $M_0^{\pm}(f)=\emptyset$, the proof follows by integrating \eqref{dW} on $M$.
\qed
\end{proof}
\medskip

In the sequel, we provide some applications of Theorem \ref{Hopf type}.
The first one is a short proof of the following result due to Asperti \cite{As}.

\begin{theorem} \label{Asperti}
If a compact 2-dimensional Riemannian manifold immerses isometrically into $\Q^4_c$ with everywhere nonvanishing normal curvature,
then it is homeomorphic either to the sphere $\mathbb{S}^2$, or to the real projective space $\R P^2$.
\end{theorem}

\begin{proof}
Let $\tilde M$ be a compact 2-dimensional Riemannian manifold, and $f\colon \tilde{M}\to \Q^4_c$ an isometric immersion with 
$K_N\neq0$ everywhere. Assume that $\tilde M$ is oriented and that $\pm K_N>0$.
Then, $M_0^{\mp}(f)=\emptyset$ and Theorem \ref{Hopf type} implies that $2\chi =\pm \chi_N$. Since $\pm\chi_N>0$, 
it follows that $\chi>0$ and thus, $\tilde M$ is homeomorphic to $\mathbb{S}^2$. 
If $\tilde M$ is non-orientable, then we apply the previous procedure to the lift of $f$ to
the orientable double covering of $\tilde M$, and 
we conclude that $\tilde M$ is homeomorphic to $\R P^2$.
\qed
\end{proof}
\medskip

We mention here that a long-standing open problem posed by S.S. Chern \cite[p. 45]{Ch0} is to investigate the existence of compact surfaces
of negative Gaussian curvature in $\R^4$. In this direction, we obtain the following result.

\begin{theorem}
Let $f\colon M\to \Q^4_c$ be an isometric immersion of a compact, oriented 2-dimensional Riemannian manifold $M$.
If $c\geq 0$ and the normal curvature of $f$ does not change sign, then the Gaussian curvature $K$ of $M$ satisfies
$\max K\geq 0$.
\end{theorem}

\begin{proof}
Arguing indirectly, suppose that $\max K< 0$. Since $c\geq 0$, this implies that $M_1(f)=\emptyset$. 
Since $K_N$ does not change sign, we may assume that $\pm K_N\geq0$. Therefore $M_0^{\mp}(f)=\emptyset$,
and as in the proof of Theorem \ref{Asperti}, we obtain that $M$ is homeomorphic to $\mathbb{S}^2$.
Then, the Gauss-Bonnet theorem implies that there exist points of $M$ with positive Gaussian curvature, and this is a contradiction.
\qed
\end{proof}
\medskip

Immediate consequences of the above theorem are the following corollaries; the first one has been proved 
by Peng and Tang \cite{PT} for surfaces in $\R^4$. 

\begin{corollary}
Let $f\colon M\to \Q^4_c$, $c\geq0$, be an isometric immersion of a compact, oriented 2-dimensional Riemannian manifold $M$.
If the normal curvature of $f$ is constant, then there exists a point of $M$ with nonnegative Gaussian curvature.
\end{corollary}

\begin{corollary}
Let $M$ be a compact, oriented 2-dimensional Riemannian manifold with Gaussian curvature $K<0$.
If there exists an isometric immersion $f\colon M\to \Q^4_c$, $c\geq0,$ then its normal curvature
satisfies $\min K_N<0<\max K_N$.
\end{corollary}

\section{Isotropically Isothermic Surfaces} \label{s4}

We introduce here the notion of isotropically isothermic surfaces in 4-dimensional space forms, as a generalization of the notion
of isothermic surfaces in 3-dimensional space forms. We recall that an umbilic-free surface $F\colon M\to \Q^3_c$ is called isothermic
if it admits a conformal curvature line parametrization around every point. This is equivalent (see for instance \cite{JMNB}) with
the co-closeness of the principal connection form $\Omega$ of $F$. Inspired by Proposition \ref{MCF-PCF} we give the following 
definitions.

Let $f\colon M\to \Q^4_c$ be an oriented surface with $M^{\pm}_0(f)=\emptyset$. 
A point $p\in M$ is called a $\pm$ {\emph{isotropically isothermic point}} for $f$ if $d\star\Omega^{\pm}(p)=0$.
The surface $f$ is called $\pm$ {\emph{(totally non) isotropically isothermic}} if every point is $\pm$ (non) isotropically isothermic.
Moreover, $f$ is called {\emph{strongly (totally non) isotropically isothermic}} if it is both $+$ and $-$ (totally non) isotropically isothermic. 
In the sequel, a $\pm$ (totally non) isotropically isothermic surface is simply called {\emph{(half totally non) isotropically isothermic}}, 
whenever we do not need to distinguish between the signs.

The following lemma provides a characterization of isotropically isothermic points in terms of a complex coordinate.
Notice that if $f\colon M\to \Q^4_c$ is a surface with $M_0^{\pm}(f)=\emptyset$, then for every complex chart $(U,z)$ on $M$
there exists a smooth complex function $h^{\pm}$ on $U$, such that the section $\phi^{\pm}$ of $N^\pm_fU$ 
given by \eqref{phipm} satisfies on $U$ the relation
\be \label{hpm}
\nap_{\dzb}\phi^{\pm}=h^{\pm}\phi^{\pm}.
\ee

\begin{lemma} \label{qiz}
Let $f\colon M\to \Q^4_c$ be an oriented surface with $M^{\pm}_0(f)=\emptyset$. A point $p\in M$ is $\pm$ isotropically isothermic for $f$
if and only if for every complex chart $(U,z)$ around $p$, the function $h^{\pm}$ satisfies 
$$\Imag h^{\pm}_z(p)=0.$$ 
\end{lemma}

\begin{proof}
Let $(U,z=x+iy)$ be a complex chart around $p$ and set $e_1=\d_x/\lambda$, $e_2=\d_y/\lambda$, where $\lambda > 0$ is the conformal factor.
Consider the frame field $\{e^{\pm}_3, e^{\pm}_4\}$ of $N_fU$ determined by $\{e_1,e_2\}$ from \eqref{HI}.
Then \eqref{fpm} and \eqref{lambda} hold on $U$.
From \eqref{hpm} and \eqref{fpm} it follows that
\be \label{phiH}
\nap_{\dzb}\phi^{\pm}=\frac{\lambda^2}{4}\|\mathcal{H}^{\pm}\|h^{\pm}(e^{\pm}_3\pm ie^{\pm}_4)\;\;\;\; \mbox{on}\;\;\;\; U.
\ee
Differentiating \eqref{fpm} with respect to $\dzb$ in the normal connection, we obtain
\bea
\nap_{\dzb}\phi^{\pm}=\frac{1}{4}\left(\dzb(\lambda^2\|\mathcal{H}^{\pm}\|)\mp i\lambda^2\|\mathcal{H}^{\pm}\|\w_{34}^{\pm}(\dzb)\right)(e^{\pm}_3\pm ie^{\pm}_4).
\eea
The above and \eqref{phiH} yield that
\begin{equation}\label{hpmB}
h^{\pm}= \dzb\log(\lambda^2\|\mathcal{H}^{\pm}\|)\mp i\w_{34}^{\pm}(\dzb).
\end{equation}
Differentiating \eqref{hpmB} with respect to $z$, and taking the imaginary part yields
$$\frac{4}{\lambda^2}\Imag h^{\pm}_z=
\mp\left( e_1(\log\lambda)\w_{34}^{\pm}(e_1)+e_2(\log\lambda)\w_{34}^{\pm}(e_2)+e_1(\w_{34}^{\pm}(e_1))+e_2(\w_{34}^{\pm}(e_2))\right).$$
From \eqref{lambda} and the above, we obtain that $d\star\Omega^{\pm}=-(4/\lambda^2)\Imag h^{\pm}_z\w_1\wedge\w_2$, and this completes the proof.
\qed
\end{proof}
\medskip

\begin{proposition} \label{RP}
Let $f\colon M\to \Q^4_c$ be an oriented surface with $M^{\pm}_0(f)=\emptyset$. The surface $f$ is $\pm$ isotropically isothermic if and only
if for every simply-connected complex chart $(U,z)$, the section $\phi^{\pm}$ given by \eqref{phipm} has the form 
\be \label{real parallel}
\phi^{\pm}=D^{\pm}\xi^{\pm}, 
\ee
where $D^{\pm}\in \mathcal{C}^\infty(U;(0,+\infty))$, and $\xi^{\pm}$ is a nowhere-vanishing holomorphic section of $N^\pm_fU$.
\end{proposition}

\begin{proof}
Let $(U,z)$ be a simply-connected complex chart. Since $M^{\pm}_0(f)=\emptyset$, the section $\phi^{\pm}$ is given on $U$ by \eqref{fpm}.
Appealing to Proposition \ref{Criterion}(i), we express $\Omega^{\pm}$ on $U$ in terms of $z$, by \eqref{lambda}.

Assume that $f$ is $\pm$ isotropically isothermic. From \eqref{lambda} it follows that $d\star\w_{34}^{\pm}=0$ and thus, 
there exists a smooth positive function $r^{\pm}$ on $U$ such that
\be \label{rpm}
\w_{34}^{\pm}=\mp\star d\log r^{\pm}.
\ee
Taking into account \eqref{fpm}, we define $D^{\pm}$ and $\xi^{\pm}$, respectively, by
\be \label{xipm}
D^{\pm}=\frac{\lambda^2\|\mathcal{H}^{\pm}\|}{4r^{\pm}}\;\;\;\;\;\mbox{and}\;\;\;\;\; \xi^{\pm}=r^{\pm}(e^{\pm}_3\pm ie^{\pm}_4).
\ee
Differentiating $\xi^{\pm}$ with respect to $\dzb$ in the normal connection, yields 
\be \label{xihol}
\nap_{\dzb}\xi^{\pm}=\frac{1}{r^{\pm}}\left((\log r^{\pm})_{\bar z}\mp i\w_{34}^{\pm}(\dzb)\right)(e^{\pm}_3\pm ie^{\pm}_4).
\ee
From the above and \eqref{rpm}, it follows that $\xi^{\pm}$ is holomorphic.

Conversely, assume that \eqref{real parallel} holds on $U$.
By setting $r^{\pm}=\|\xi^{\pm}\|/\sqrt 2$, from \eqref{real parallel} and \eqref{fpm} we obtain \eqref{xipm}.
Therefore, \eqref{xihol} is valid.
Since $\xi^{\pm}$ is holomorphic, from \eqref{xihol} we obtain \eqref{rpm}.
Equations \eqref{lambda} and \eqref{rpm} imply that $d\star\Omega^{\pm}=0$ on $U$.
Since $U$ is arbitrary, it follows that $f$ is $\pm$ isotropically isothermic.
\qed
\end{proof}
\medskip

The characterization of isotropic isothermicity provided by Proposition \ref{RP}, also makes sense 
for oriented surfaces immersed in orientable 4-dimensional Riemannian manifolds of not necessarily constant sectional curvature,
and can be used as the definition of isotropic isothermicity for such surfaces.

\begin{proposition}\label{conformal}
Let $N$ be a Riemann surface and $F\colon N\to \Q^4_c$ a conformal immersion.
The property of $F$ equipped with its induced metric being isotropically isothermic is invariant under conformal changes of the metric of $\Q^4_c$.
\end{proposition}

\begin{proof}
Let $f\colon M\to \Q^4_c$ be the isometric immersion induced by $F$, where $M=(N,ds^2)$ and $ds^2=F^*\langle \cdot,\cdot\rangle$. 
Consider the Riemannian manifold $\tilde{\Q}^4_c$, obtained from $\Q^4_c$ by the conformal change $\langle \cdot,\cdot\rangle_{\mu}=\mu^2\langle \cdot,\cdot\rangle$
of its metric, where $\mu\in \mathcal{C}^\infty(\Q^4_c;(0,+\infty))$, equipped with the same orientation with $\Q^4_c$. 
The conformal immersion $F$ induces the isometric immersion
$\tilde{f}\colon \tilde{M}\to \tilde{\Q}^4_c$, where $\tilde{M}=(N,d\tilde{s}^2)$ and $d\tilde{s}^2=\mu^2 ds^2$.

Assume that $f$ is $\pm$ isotropically isothermic. We argue that $\tilde{f}$ is also $\pm$ isotropically isothermic.
The normal bundles of $f$ and $\tilde{f}$ coincide as vector bundles over $N$,
and since their bundle metrics are conformal, they have the same complex structure $J^\perp$.
It follows easily (see for instance \cite{DTB}) that
the second fundamental forms $\a,\tilde\a$, and the normal connections $\nap, \tilde{\nabla}^\perp$, of $f$ and $\tilde{f}$, respectively, are related by
\be \label{sfforms}
\tilde\a(X,Y)=\a(X,Y)-\frac{1}{\mu}\langle X,Y\rangle(\grad\mu)^\perp\;\;\;\;\;\;\mbox{and}\;\;\;\;\;\; 
\tilde{\nabla}^\perp_X\eta=\nap_X\eta+\frac{1}{\mu}\langle \grad\mu, X\rangle \eta,
\ee
for all $X,Y\in TN$ and $\eta\in N_fM=N_{\tilde f}{\tilde M}$, where $\grad$ denotes the gradient with respect to $\langle \cdot,\cdot\rangle$.
Let $(U,z)$ be a complex chart on $\tilde M$ with conformal factor $\tilde\lambda$. Then, $(U,z)$ is also a complex chart on $M$ with conformal factor 
$\lambda=\tilde\lambda/\mu$.
From the first equation in \eqref{sfforms}, it follows that the Hopf differentials $\Phi,\tilde{\Phi}$ of $f,\tilde{f}$, respectively, coincide. In particular,
if $\Phi^\pm$ is given by \eqref{phipm} and $\tilde{\Phi}^\pm=\tilde{\phi}^\pm dz^2$ on $U$, then $\phi^{\pm}=\tilde{\phi}^\pm$.
Proposition \ref{RP} implies that $\phi^{\pm}=D^\pm\xi^\pm$, where $D^\pm$ is a smooth positive function on $U$ and $\xi^\pm$ a
nowhere-vanishing $\nap$-holomorphic local section. Then, we have that
$$\tilde{\phi}^\pm=\phi^{\pm}=\tilde{D}^\pm\tilde{\xi}^\pm,\;\;\; \mbox{where}\;\;\; \tilde{D}^\pm=\mu D^\pm\;\;\mbox{and} \;\;\tilde{\xi}^\pm=\frac{1}{\mu}\xi^\pm.$$
Since $\xi^\pm$ is $\nap$-holomorphic, from the second equation in \eqref{sfforms} we obtain that $\tilde{\xi}^\pm$ is $\tilde{\nabla}^\perp$-holomorphic.
From Proposition \ref{RP} it follows that $\tilde{f}$ is $\pm$ isotropically isothermic.
\qed
\end{proof}
\medskip

\begin{remark}
{\emph{
Adopting the notation of the proof of Proposition \ref{conformal}, 
by using \eqref{e34pm}, \eqref{Om} and Proposition \ref{Criterion}(i),
it is easy to see that if the metric $\langle \cdot,\cdot\rangle_{\mu}$
has constant curvature, then the corresponding mixed connection forms $\Omega^\pm, \tilde{\Omega}^\pm$
of $f$ and $\tilde f$, are related by $\tilde{\Omega}^\pm=\Omega^\pm+2\star d\log\mu$.
}}
\end{remark}

\begin{examples}\label{examples}
{\emph{We provide some classes of isotropically isothermic surfaces $f\colon M\to \Q^4_c$.
The surfaces in the classes (iii) and (iv) below, are always strongly isotropically isothermic.
}}

(i) Surfaces with a vertically harmonic Gauss lift (neither minimal, nor superconformal)
\\
{\emph{Assume that the Gauss lift $G_{\pm}$ of $f$ is vertically harmonic and that $M_0^{\pm}(f)=\emptyset$. 
Proposition \ref{glphi} implies that $\Phi^{\pm}$
is holomorphic and from Proposition \ref{RP} it follows that $f$ is $\pm$ isotropically isothermic.
According to Proposition \ref{conformal}, by appropriate conformal changes of the metric of (possibly part of) $\Q^4_c$, we obtain from $f$ other
$\pm$ isotropically isothermic surfaces in $\Q^4_{\tilde c}$ whose corresponding Gauss lift $\tilde{G}_\pm$ is not vertically harmonic.}}

(ii) Minimal superconformal surfaces. 
\\
{\emph{Assume that $f$ is minimal and superconformal, with $M_0^{\pm}(f)=\emptyset$. For the Hopf
differential $\Phi$ of $f$, Lemma \ref{pseudo}(i) implies that $\Phi^\mp\equiv0$ and thus, $\Phi\equiv\Phi^\pm$. 
The Codazzi equation yields that $\Phi$ is holomorphic and from Proposition \ref{RP} it follows that $f$ is $\pm$ isotropically isothermic.
Since the superconformal property is conformally invariant, by virtue of Proposition \ref{conformal}, we obtain 
from $f$, non-minimal superconformal surfaces in $\Q^4_{\tilde c}$ that are isotropically isothermic.}}

(iii) Non-superconformal minimal surfaces. 
\\
{\emph{Assume that $f$ is minimal with $M_0(f)=\emptyset$. The Codazzi equation implies that the Hopf differential of $f$ is 
holomorphic and Proposition \ref{RP} yields that $f$ is strongly isotropically isothermic. 
Proposition \ref{conformal} implies that under appropriate conformal changes of the metric of $\Q^4_c$,
the surface $f$ gives rise to non-minimal, strongly isotropically isothermic surfaces in $\Q^4_{\tilde c}$.
In particular, since the flatness of the normal bundle of a surface in $\Q^4_c$
is a conformally invariant property, it follows that such a surface has nonflat normal bundle, if
the normal bundle of $f$ is nonflat.}}

(iv) Isothermic surfaces in totally umbilical hypersurfaces.
\\
{\emph{Assume that $f$ is the composition of an umbilic-free surface $F\colon M\to \Q^3_{\tilde c}$, $\tilde{c}\geq c$, with a totally
umbilical inclusion. Proposition \ref{MCF-PCF} implies that $f$ is strongly isotropically isothermic if and only if 
$F$ is isothermic.}}
\end{examples}

\subsection{Lines of Curvature}

We recall that (cf. \cite{GGTG,GS}) a principal direction of an oriented surface $f\colon M\to \Q^4_c$
at $p\in M$, is a line in $T_pM$
generated by a unit vector which makes extremal the length of $\a(X,X)$, where $\a$ is the second fundamental form of $f$, 
and $X$ varies on the unit circle of $T_pM$. If $p\in M\smallsetminus M_0(f)$, then there exist four principal directions 
of $f$ at $p$. The {\emph{principal curvature lines of $f$}} are those curves on $M\smallsetminus M_0(f)$ which are tangent to principal
directions.

An oriented surface $f\colon M\to \Q^4_c$ is called {\emph{isothermic}} (cf. \cite{Pa}) if around every point of $M$
there exists a complex chart with the property that its corresponding basic vector fields 
diagonalize at every point of its domain, all shape operators of $f$. 
It is straightforward to show that a surface is isothermic if and only if around every point of $M$
there exists a complex chart $(U,z=x+iy)$ such that $\a(e_1,e_2)=0$ at every point of $U$, where
$e_1=\d_x/\lambda, e_2=\d_y/\lambda$, and $\lambda>0$ is the conformal factor.

\begin{proposition} \label{LoC}
Let $f\colon M\to \Q^4_c$ be an oriented surface with $M_0(f)=\emptyset$. 
\begin{enumerate}[topsep=0pt,itemsep=-1pt,partopsep=1ex,parsep=0.5ex,leftmargin=*, label=(\roman*), align=left, labelsep=-0.4em]
\item Assume that $f$ is strongly isotropically isothermic. Then it admits a conformal principal curvature line parametrization around every point.
In particular, $f$ is isothermic if it has flat normal bundle.
\item If $f$ is isotropically isothermic and admits a conformal principal curvature line 
parametrization around every point, then it is strongly isotropically isothermic. 
In particular, if $f$ is isothermic and isotropically isothermic, then it is strongly isotropically isothermic.
\end{enumerate}
\end{proposition}

\begin{proof}
Let $p\in M$. Since $\mathcal{E}(p)$ is not a circle,
from \cite[Lemma 6]{PV} it follows that there exist positively oriented local orthonormal frame fields
$\{e_1,e_2\}$ of $TM$, $\{e_3,e_4\}$ of $N_fM$, on a neighbourhood $U$ of $p$, and $\kappa, \mu\in \mathcal{C}^\infty(U)$ with
$\kappa>|\mu|$, such that $\a_{11}-\a_{22}=2\kappa e_3$ and $\a_{12}=\mu e_4$, where $\a_{kl}=\a(e_k,e_l), k,l=1,2$. 
In particular, from the proof of \cite[Lemma 6]{PV} it follows that $e_3$ is in the direction of the major axis of $\mathcal{E}_f$, 
and $\kappa$, $|\mu|$ are the lengths its semi-axes at every point of $U$.
Then, \eqref{e34pm} implies that $e_3^+=e_3^-$ and thus, $\w_{34}^+=\w_{34}^-$. From Proposition \ref{Criterion}(i) it follows that 
\be \label{O++O-}
\Omega^++\Omega^-=4\w_{12}\;\;\; \mbox{on}\;\;\; U,
\ee
where $\w_{12}$ is the connection form corresponding to the dual frame field of $\{e_1,e_2\}$.

(i) Since $\Omega^+$ and $\Omega^-$ are both co-closed, from \eqref{O++O-} it follows that $d\star\w_{12}=0$. 
Therefore, there exists a positive function $\lambda$ on $U$ such that $\star\w_{12}=-d\log\lambda$.
This implies that the forms $\lambda^{-1}\w_1, \lambda^{-1}\w_2$ are closed and thus, there exist smooth functions
$x,y$ on $U$ such that $dx=\lambda^{-1}\w_1, dy=\lambda^{-1}\w_2$. Then, $z=x+iy$ is a complex coordinate on $U$
with conformal factor $\lambda$, such that $e_1=\d_x/\lambda$, $e_2=\d_y/\lambda$.
In particular, if $f$ has flat normal bundle, then \eqref{axes} implies that $\mu=0$. Therefore $a_{12}=0$ and thus, $f$ is isothermic.

(ii) Suppose that $f$ is $\pm$ isotropically isothermic and consider a conformal principal curvature line parametrization $(U,z=x+iy)$ around $p\in M$,
with conformal factor $\lambda>0$.
Then, the connection form of the dual frame field of $\{\tilde{e}_1=\d_x/\lambda, \tilde{e}_2=\d_y/\lambda\}$ is given by $\tilde\w_{12}=\star d\log\lambda$.

We claim that there exists a conformal principal curvature line parametrization on $U$, with normalized basic vector fields $e_1$, $e_2=Je_1$, such that
$\a_{11}$ is a vertex of $\mathcal{E}_f$ determined by the major axis at any point of $U$. Indeed, in the case where 
$\a(\tilde{e}_1,\tilde{e}_1)$ is a vertex of $\mathcal{E}_f$ determined by the minor axis, we consider the frame
field $\{e_1, e_2\}$ given by $e_1-ie_2=\exp(i\pi/4)(\tilde{e}_1-i\tilde{e}_2)$. 
Then, the connection form of its dual frame field is given by $\w_{12}=\tilde\w_{12}$, and the vector field $\a_{11}$
is a vertex of $\mathcal{E}_f$ determined by the major axis. Since $\w_{12}$ is co-closed, as in the proof of part (i),
it follows that there exists a complex coordinate with normalized basic vector fields $e_1$ and $e_2$.

For the frame field $\{e_1,e_2\}$, equation \eqref{O++O-} is valid. Since $d\star\Omega^{\pm}=0$, 
from \eqref{O++O-} it follows that $d\star\Omega^{\mp}=0$ and thus, $f$ is strongly isotropically isothermic.
The rest of the proof is obvious.
\qed
\end{proof}
\medskip

The following example shows that the converse of Proposition \ref{LoC}(i) is not true in general. 
Bearing in mind Example \ref{examples}(iv), it also shows that the classes of isothermic and isotropically isothermic surfaces overlap, 
but no one of these classes is contained in the other.

\begin{example} \label{ISTNII}
Isothermic surfaces in $\R^4$ that are strongly totally non isotropically isothermic.
\\
{\emph{Let $\gamma_j:I_j\to \R^2$ be a smooth curve parametrized by its arc length $s_j$, where $I_j$ is an open interval, $j=1,2$.
Let $n_j$ be the normal vector field of $\gamma_j$ such that $\{t_j=\dot{\gamma}_j,n_j\}$ is positively oriented, 
where the dot denotes the derivative with respect to $s_j$, $j=1,2$.
By setting $M=I_1\times I_2$ and $z=s_1+is_2$, it is clear that 
$z$ is a global complex coordinate on $M$ with basic vector fields $e_1,e_2$, where $e_j=\d/\d s_j$, $j=1,2$.
Moreover, the connection form of the corresponding coframe of $\{e_1,e_2\}$ satisfies $\w_{12}=0$.
We consider the product surface $f\colon M\to \R^4$, $f=\gamma_1\times\gamma_2$. Then, the adapted to $f$ frame field 
$$\{f_*e_1=(t_1,0), N_1=(n_1,0), f_*e_2=(0,t_2), N_2=(0,n_2)\}$$
is positively oriented in $\R^4$. Therefore, $\Jp N_1=-N_2$. 
Let $k_j$ be the curvature of $\gamma_j$, $j=1,2$. For the second fundamental form $\a$
of $f$ we have $\a_{11}=k_1N_1$, $\a_{22}=k_2N_2$ and $\a_{12}=0$, where $\a_{kl}=\a(e_k,e_l)$, $k,l=1,2$. 
Since $\a_{12}=0$, it follows that $f$ is isothermic.}}

{\emph{Assume furthermore that $f$ is umbilic-free, or equivalently, that there do not exist points $(s_1,s_2)$ on $M$ such that $k_1(s_1)=k_2(s_2)=0$,
and set
$$e_3=\frac{\a_{11}-\a_{22}}{\|\a_{11}-\a_{22}\|}=\frac{1}{\sqrt{k_1^2+k_2^2}}(k_1N_1-k_2N_2),\;\; e_4=\Jp e_3.$$
Then, \eqref{e34pm} implies that $e_3=e_3^-=e_3^+$. Since $\w_{12}=0$, from Proposition \ref{Criterion}(i) and \eqref{Om}
it follows that $f$ is strongly isotropically isothermic if and only if $\w_{34}$ is co-closed. An easy computation shows that 
at every point of $M$, the equation $d\star\w_{34}=0$ is equivalent to the differential equation
\be \label{curves}
k_1\ddot{k}_2-\ddot{k}_1k_2+ 2k_1k_2\frac{(\dot{k}_1)^2-(\dot{k}_2)^2}{k_1^2+k_2^2}=0
\ee
for the curvatures of $\gamma_1$ and $\gamma_2$, where each dot denotes a derivative of $k_j$ with respect to $s_j$, $j=1,2$. 
Clearly, if $k_j(s_j)=c_js_j$, $0\neq c_j\in\R$, $j=1,2,$ and $c_1\neq c_2$, then for $s_1s_2>0$ it follows from \eqref{curves} that $f$ is 
strongly totally non isotropically isothermic.}}
\end{example}
\medskip

We recall (cf. \cite{Me}) that {\emph{a mean-directional
curvature line}} of an oriented surface $f\colon M\to \Q^4_c$, is a curve on $M$ which is tangent at every point to a unit vector field, 
whose image under the second fundamental form of $f$
is parallel to the mean curvature vector field. There exist two families of mean-directional curvature lines, whose common singularities are the minimal
points of $f$ and the points where the ellipse of curvature $\mathcal{E}_f$ degenerates into a line segment, parallel to the mean curvature vector.

\begin{proposition}\label{mdi}
Let $f\colon M\to \Q^4_c$ be an umbilic-free superconformal surface with nowhere-vanishing mean curvature vector field. 
The surface $f$ is isotropically isothermic if and only if it admits a conformal mean-directional curvature line parametrization around every point.
\end{proposition}

\begin{proof}
Since $M_1(f)=\emptyset$, by virtue of Lemma \ref{pseudo}(ii), we may assume that $\pm K_N<0$. Then, Lemma \ref{pseudo}(ii) implies
that $\Phi^\mp\equiv0$ and from Proposition \ref{glphi} it follows that the Gauss lift $G_\mp$ of $f$ is vertically harmonic.
Consider the orthonormal frame field $\{e_3=H/\|H\|, e_4=\Jp e_3\}$ of the normal bundle. Using Proposition \ref{glphi}(iv) 
and \eqref{connection forms}, we obtain that the connection form of its dual frame field 
is given by
\be \label{w34sc}
\w_{34}=\mp\star d\log\|H\|.
\ee
Let $r>0$ be the radius of $\mathcal{E}_f$ at every point of $M$, and 
consider a positively oriented local orthonormal frame field $\{e_1, e_2\}$ of $TM$, such that
$$\a_{11}=(\|H\|+r)e_3,\;\;\; \a_{22}=(\|H\|-r)e_3.$$
Since $\mathcal{H}^\mp(e_1,e_2)\equiv0$, from \eqref{Hopf Invariants} and the above it follows that 
$\a_{12}=\mp re_4.$
Then, \eqref{e34pm} implies that $e_3^\pm=e_3$ and $e_4^\pm=e_4$. From Proposition \ref{Criterion}(i), \eqref{Om} and \eqref{w34sc} we obtain that
\be \label{OmH}
\Omega^\pm=2\w_{12}-\star d\log\|H\|,
\ee
where $\w_{12}$ is the connection form corresponding to the dual frame field of $\{e_1, e_2\}$.

A conformal parametrization whose coordinate curves are mean-directional curvature lines exists around every point of $M$, if and only if $\w_{12}$ is co-closed.
The proof follows immediately from \eqref{OmH}.
\qed
\end{proof}

\subsection{Infinitesimal Deformations}

Let $f\colon M\to \R^4$ be an oriented surface and denote by $N$ the underlying Riemann surface of $M$, such that $M=(N,ds^2)$.
A deformation of $f$ is a smooth map $F\colon I\times N\to \R^4$ with $F(0,\cdot)=f$, where $I\subset \R$ is an open interval containing $0$.
For every $t\in I$, we denote by $f_t$ the isometric immersion $F(t,\cdot)\colon M_t\to \R^4$, where $M_t=(N,ds_t^2)$. 
At any point of $N$, the Taylor expansion of $f_t$ around $t=0$ is $f_t=f+t\mathcal{T}+o(t)$, 
where 
$$\mathcal{T}=F_*\partial/\partial t|_{t=0}=\delta f_t,$$
and $\delta=(d/dt)|_{t=0}$ is the {\emph{variational operator}}.
The deformation $F$ is called {\emph{isometric}} if $ds_t^2=ds^2$ for every $t\in I$, 
and is called {\emph{infinitesimal isometric}} if $\delta ds^2_t=0$.
If $F$ is infinitesimal isometric then
the section $\mathcal{T}\in \Gamma(f^*(T\R^4))$ defined above 
is called the {\emph{bending field of $F$}},
and by using the Taylor expansion of $f_t$, it follows that it satisfies
\be \label{bfld}
\langle\tilde{\nabla}_X\mathcal{T},f_*Y\rangle +\langle f_*X,\tilde{\nabla}_Y\mathcal{T}\rangle=0,\;\;\; X,Y\in TM,
\ee
where $\tilde{\nabla}$ is the connection of $f^*(T\R^4)$.

Every section $\mathcal{T}$ of $f^*(T\R^4)$ satisfying \eqref{bfld} is called {\emph{a bending field}}, and such sections always exist;
the variational vector field $\mathcal{T}$ of an isometric deformation of $f$ produced by a smooth one-parameter family of isometries of $\R^4$,
satisfies \eqref{bfld} and is called a {\emph{trivial bending field}}.
A bending field $\mathcal{T}$ is trivial (cf. \cite{DTB}) if and only if there exist constant vectors $C\in \Lambda^2\R^4$ and $v\in\R^4$, such that
$$\mathcal{T}=C\cdot f+v,$$
where the dot multiplication of a simple 2-vector $X\wedge Y\in\Lambda^2\R^4$ with $Z\in \R^4$,
is defined by $X\wedge Y\cdot Z=\langle Y,Z\rangle X-\langle X,Z\rangle Y$, extends linearly to every element of $\Lambda^2\R^4$,
and is skew-symmetric with respect to the inner product of $\R^4$.
An infinitesimal isometric deformation is called either {\emph{trivial}}, or {\emph{nontrivial}}, 
if its bending field is either trivial on $M$, or nontrivial on an open and dense subset of $M$, respectively.
Two bending fields $\mathcal{T}_1, \mathcal{T}_2\in\Gamma(f^*(T\R^4))$ 
are equivalent and we identify them,
if there exist $0\neq c\in\R$ and a trivial bending field $\mathcal{T}_0$,
such that $\mathcal{T}_2= c\mathcal{T}_1 +\mathcal{T}_0$.

Every bending field $\mathcal{T}\in\Gamma(f^*(T\R^4))$ determines a unique infinitesimal 
isometric deformation of the form
\be \label{form inf}
f_t=f+t\mathcal{T},\;\;\;\;\; 
\ee
for $t$ in some fixed interval $I$,
which is always assumed to be sufficiently small for our purposes.
In the sequel, {\emph{we deal only with deformations of the above form}}, 
and we write $F\colon I\times M\to \R^4$ to denote such a deformation of the surface $f\colon M\to \R^4$. 

For the proofs of Theorems \ref{IIID} and \ref{IIS}, we need a version of the fundamental theorem of infinitesimal isometric deformations,
recently proved in \cite{DJ} in invariant form, in terms of moving frames. 
A statement of the fundamental theorem in this context, and also some auxiliary results,
can be found in the survey paper \cite{IMS}. Because it turned out to be impossible for the author to find  
detailed proofs, or even proofs of some of these results (some references in \cite{IMS} are in Russian, and others are really hard to find), 
and arguments involving moving frames jointly with Taylor expansions are quite delicate, 
we also provide neat proofs of everything that we use to obtain our results.

\begin{lemma} \label{invfr}
Let $F\colon I\times M\to \R^4$ be an infinitesimal isometric deformation.
\begin{enumerate}[topsep=0pt,itemsep=-1pt,partopsep=1ex,parsep=0.5ex,leftmargin=*, label=(\roman*), align=left, labelsep=-0.4em]
\item If $M$ is simply-connected, then every orthonormal frame field $\{e_1,e_2\}$ on $M$, extends to a smooth with respect to $t$, orthonormal frame field
$\{e_1(t),e_2(t)\}$ on $M_t$, with dual frame field $\{\w_1(t),\w_2(t)\}$ and corresponding connection form $\w_{12}(t)$,
such that 
\be \label{delta int}
\delta e_j(t)= \delta \w_j(t)=0,\;\; j=1,2,\;\;\;\mbox{and}\;\;\;\delta\w_{12}(t)=0.
\ee
\item Every orthonormal frame field $\{e_3,e_4\}$ of $N_fM$, 
locally extends to a smooth with respect to $t$,
local orthonormal frame field $\{e_3(t),e_4(t)\}$ of $N_{f_t}M_t$.
\end{enumerate}
\end{lemma}

\begin{proof}
(i) Let $\{e_1,e_2\}$ be an orthonormal frame field on $M$. 
Applying the Gram-Schmidt process with respect to the metric $ds^2_t$, to the frame field $\{e_1,e_2\}$, we 
obtain a smooth with respect to $t$, orthonormal frame field $\{\tilde{e}_1(t),\tilde{e}_2(t)\}$ on $M_t$, 
with $\tilde{e}_j(0)=e_j$, $j=1,2$. 
Since $\delta ds^2_t=0$, from the Taylor expansions $\tilde{e}_j(t)=e_j+t\delta\tilde{e}_j(t)+o(t)$, $j=1,2$,
it follows that $\delta\tilde{e}_1(t)=ue_2$ and $\delta\tilde{e}_2(t)=-ue_1$, for some $u\in \mathcal{C}^\infty(M)$.
Then, the orthonormal frame field on $M_t$ defined by 
$e_1(t)+ie_2(t)=\exp{(itu)}(\tilde{e}_1(t)+i\tilde{e}_2(t))$,
depends smoothly on $t$ and satisfies $e_j(0)=e_j$ and $\delta e_j(t)=0$, $j=1,2$. 
Let $\{\w_1(t),\w_2(t)\}$ be the dual frame field of $\{e_1(t),e_2(t)\}$ and $\w_{12}(t)$ its connection form.
Since the coefficients of the corresponding powers of $t$ in the Taylor expansions of $e_j(t)$ and $\w_j(t)$ are dual,
from $\delta e_j(t)=0$ it follows that $\delta \w_j(t)=0$, $j=1,2$.
Moreover, using the Taylor expansions of all the involved forms in $d\w_j(t)=\w_{jr}(t)\wedge\w_r(t)$, $j,r=1,2$, 
and comparing the coefficients of $t$, we obtain that $\delta\w_{12}(t)=0$.

(ii) We claim that every point of $M$ has a neighbourhood on which, there exists a smooth with respect to $t$, 
local orthonormal frame field $\{\tilde{e}_3(t),\tilde{e}_4(t)\}$ of $N_{f_t}M_t$.
Indeed, for $p\in M$ there exists a neighbourhood $\tilde{U}$ of $p$ on which, the surface $f$  
is a graph over a coordinate plane of $\R^4$.
Assume that $f(x,y)=(x,y,R(x,y),S(x,y))$ on $\tilde{U}$.
Expressing the bending field of $F$ by using the coordinates $(x,y)$, and
substituting into \eqref{form inf}, it follows that for sufficiently small $t$,
the vector $E_3=(0,0,1,0)\in \R^4$ is non-tangent to
$f_t$ in a neighbourhood $U$ of $p$, compactly contained in $\tilde{U}$. 
From part (i), there exists a smooth with respect to $t$, positively oriented orthonormal frame field $\{e_1(t),e_2(t)\}$ on $U_t=(U,ds_t^2)$.
Then, the local section $N_3(t)=\star\left({f_t}_*e_1(t)\wedge{f_t}_*e_2(t)\wedge E_3\right)$
of $f_t^*(T\R^4)$, where $\star$ is the Hodge star operator, depends smoothly on $t$ and is nowhere-vanishing.
Applying the Gram-Schmidt process to the frame field $\{{f_t}_*e_1(t),{f_t}_*e_2(t),N_3(t)\}$, we obtain a smooth with
respect to $t$, unit vector field $\tilde{e}_3(t)$ of $N_{f_t}U_t$.
Then, the orthonormal frame field $\{\tilde{e}_3(t),\tilde{e}_4(t)\}$ of $N_{f_t}U_t$, where 
$\tilde{e}_4(t)=\star\left({f_t}_*e_1(t)\wedge{f_t}_*e_2(t)\wedge \tilde{e}_3(t)\right)$,
depends smoothly on $t$ and the claim follows.

Let $\{e_3,e_4\}$ be a $\pm$ oriented orthonormal frame field of $N_fM$.  
For $p\in M$, consider a frame field $\{\tilde{e}_3(t),\tilde{e}_4(t)\}$ of $N_{f_t}U_t$ as in the claim proved above.
By setting $\tilde{e}_a(0)=\tilde{e}_a$, $a=3,4$, there exists $\tau\in \mathcal{C}^\infty(U)$ such that
$e_3\mp ie_4=\exp(i\tau)(\tilde{e}_3-i\tilde{e}_4)$.
Then, the orthonormal frame field $\{e_3(t),e_4(t)\}$ of $N_{f_t}U_t$ given by $e_3(t)\mp ie_4(t)=\exp(i\tau)(\tilde{e}_3(t)-i\tilde{e}_4(t))$,
depends smoothly on $t$, and $e_a(0)=e_a$, $a=3,4$.
\qed
\end{proof}

\medskip

Let $F\colon I\times M\to \R^4$ be an infinitesimal isometric deformation. An {\emph{adapted to $F$ orthonormal frame field}},
is a smooth with respect to $t$ frame field $\{e_k(t)\}_{1\leq k\leq 4}$, 
such that $\{e_1(t),e_2(t)\}$ and $\{e_3(t),e_4(t)\}$ are positively oriented orthonormal frame fields of $TM_t$ and $N_{f_t}M_t$, respectively,
and the former satisfies \eqref{delta int}.
For such a frame field, we denote by
$\w_{kl}(t)$, $1\leq k,l\leq 4$, 
the connection forms of its corresponding coframe, and by
$\{\varepsilon_k(t)\}_{1\leq k\leq 4}$ the adapted to $f_t$ frame field given by
$$\varepsilon_j(t)=f_{t_*}e_j(t),\;\; j=1,2,\;\;\; \mbox{and}\;\;\; \varepsilon_a(t)=e_a(t),\;\; a=3,4.$$
Then, the Gauss and Weingarten formulae for $f_t$ imply that
\be \label{G-W}
\tilde{\nabla}^t\varepsilon_k(t)=\sum_{l=1}^{4}\w_{kl}(t)\varepsilon_l(t),\;\;\; 1\leq k\leq4,
\ee
where $\tilde{\nabla}^t$ stands for the connection of $f_t^*(T\R^4)$, and $\tilde{\nabla}^0=\tilde{\nabla}$.
In order to simplify the notation, we also set 
$e_k(0)=e_k$, $\varepsilon_k(0)=\varepsilon_k$ and $\w_{kl}(0)=\w_{kl}$, for $1\leq k,l\leq 4$.

\begin{lemma} \label{varfraij}
Let $F\colon I\times M\to \R^4$ be an infinitesimal isometric deformation with bending field $\mathcal{T}$. 
If $\{e_k(t)\}_{1\leq k\leq 4}$ is an adapted to $F$ orthonormal frame field, then
there exists a unique section $W$ of $f^*(\Lambda^2T\R^4)$ such that the variations 
of $\varepsilon_k(t)$ are given by
\be \label{vframe}
\delta \varepsilon_k(t)=W\cdot \varepsilon_k, \;\;\; 1\leq k\leq 4, \;\;\;\; \mbox{with}\;\;\;\; \delta \varepsilon_j(t)=\tilde{\nabla}_{e_j}\mathcal{T},\;\;\;j=1,2,
\ee
and the variations $\varphi_{kl}=\delta\w_{kl}(t)$ of the connection forms, by
\be \label{fklW}
\varphi_{kl}=\langle \hat{\nabla}W\cdot \varepsilon_k,\varepsilon_l\rangle,\;\;\;\; 1\leq k,l\leq 4,
\ee
where $\hat{\nabla}$ is the connection of $f^*(\Lambda^2T\R^4)$.
\end{lemma}

\begin{proof}
Differentiating the relations $\langle \varepsilon_k(t), \varepsilon_l(t)\rangle=\delta_{kl}$, where $\delta_{kl}$
is the Kronecker's delta, we obtain that 
$\langle \delta\varepsilon_k(t), \varepsilon_l\rangle=-\langle \varepsilon_k, \delta\varepsilon_l(t)\rangle$, $1\leq k,l\leq 4$.
By setting $w_{kl}=\langle \delta\varepsilon_k(t), \varepsilon_l\rangle$, $1\leq k,l\leq 4$, it follows that the section $W$ of $f^*(\Lambda^2T\R^4)$ given by 
$$W=-\sum_{1\leq k<l\leq 4}w_{kl}\varepsilon_k\wedge\varepsilon_l$$
satisfies $\delta \varepsilon_k(t)=W\cdot \varepsilon_k$, $1\leq k\leq 4$, and is clearly unique. Furthermore, Lemma \ref{invfr}(i) yields that $e_j(t)=e_j+o(t)$
and thus, $f_{t_*}e_j(t)=f_{t_*}e_j+o(t)$, $j=1,2$. Using \eqref{form inf} in the right-hand side of the last relation, we obtain that
$$\varepsilon_j(t)=f_*e_j+t\tilde{\nabla}_{e_j}\mathcal{T}+o(t), \;\;j=1,2.$$
The above implies that $\delta \varepsilon_j(t)=\tilde{\nabla}_{e_j}\mathcal{T}$, $j=1,2$, and \eqref{vframe} follows.

Moreover, from \eqref{G-W} we have that 
$$\w_{kl}(t)=\langle \tilde{\nabla}^t\varepsilon_k(t), \varepsilon_l(t)\rangle,\;\;\; 1\leq k,l\leq 4,$$
and from the Taylor expansions $\varepsilon_k(t)=\varepsilon_k+t\delta \varepsilon_k(t)+o(t)$, $1\leq k\leq 4$, 
we obtain 
$$\tilde{\nabla}^t\varepsilon_k(t)=\tilde{\nabla}\varepsilon_k+ t\tilde{\nabla}\delta\varepsilon_k(t)+o(t),\;\;\; 1\leq k\leq 4.$$
Using again the Taylor expansions of $\varepsilon_l(t)$, $1\leq l\leq 4$, the above two relations give
$$\w_{kl}(t)=\w_{kl}+t\left(\langle \tilde{\nabla}\delta\varepsilon_k(t), \varepsilon_l\rangle+\langle \tilde{\nabla}\varepsilon_k, \delta\varepsilon_l(t)\rangle\right)
+o(t),\;\;\; 1\leq k,l\leq 4.$$
The above and \eqref{vframe} imply that
\bea
\varphi_{kl}=\langle \tilde{\nabla}(W\cdot \varepsilon_k), \varepsilon_l\rangle+\langle \tilde{\nabla}\varepsilon_k, W\cdot \varepsilon_l\rangle,\;\;\; 1\leq k,l\leq 4.
\eea
Equation \eqref{fklW} follows immediately from the above, by using that the formulae 
\begin{eqnarray}
\tilde{\nabla}_X(V\cdot f_*Y)=(\hat{\nabla}_XV)\cdot f_*Y+V\cdot \tilde{\nabla}_X f_*Y \;\;\;\;\;\mbox{and}\;\;\;\;\;
\langle V\cdot f_*X,f_*Y \rangle=-\langle f_*X,V\cdot f_*Y \rangle  \label{skew} 
\end{eqnarray}
hold for any $V\in\Gamma(f^*(\Lambda^2T\R^4))$ and $X,Y\in TM$.
\qed
\end{proof}

\medskip

The following is the fundamental theorem of infinitesimal isometric deformations in terms of moving frames. 
The main idea of the proof is contained in \cite{Ma}, where the theorem has been proved in terms of local coordinates.

\begin{theorem} \label{FTHMINF}
Assume that $f\colon M\to \R^4$ is a simply-connected oriented surface.
\begin{enumerate}[topsep=0pt,itemsep=-1pt,partopsep=1ex,parsep=0.5ex,leftmargin=*, label=(\roman*), align=left, labelsep=-0.4em]
\item Let $F\colon I\times M\to \R^4$ be an infinitesimal isometric deformation. 
If $\{e_k(t)\}_{1\leq k\leq 4}$ is
an adapted to $F$ orthonormal frame field, then the variations $\{\varphi_{kl}\}_{1\leq k,l\leq 4}$ of the 
connection forms of its dual frame field
satisfy the fundamental system
\begin{eqnarray}
\varphi_{12}=0\;\;\; \mbox{and} \;\;\; \varphi_{kl}=-\varphi_{lk},\;\;\;\;\;\;\;\;\;\;\;\;\;\;\;\;\;\;\;\;\;\;\;\;\;\;\;\;\;\;\;\;\;\;\;\;\;\;\;\;    
\;\;\; 1\leq k,l\leq4, \;\;\;\;\;\;\;\;\;\;                    \label{vf12}                 \\
\sum_{j=1}^2\w_j\wedge\varphi_{ja}=0,\;\;\;\;\;\;\;\;\;\;\;\;\;\;\;\;\;\;\;\;\;\;\;\;\;\;\;\;\;\;\;\;\;\;\;\;\;\;\;\;\;\;\;\;\;\;\;\;\;\;\;\;\;\;  
\;\;\;    a=3,4,           \;\;\;\;\;\;\;\;\;\;\;\;           \label{vfja}                  \\
d\varphi_{ja}=\sum_{r=1}^{2}\w_{jr}\wedge\varphi_{ra}+\sum_{b=3}^{4}(\varphi_{jb}\wedge \w_{ba}+\w_{jb}\wedge\varphi_{ba}),\;\;\;\;\;\;\;\;\;\;\;\;\;\;\;  
\;\;\;   j=1,2,   \;\;\;    a=3,4,     \;\;\;\;\;\;            \label{dfja}                   \\
\sum_{a=3}^4(\varphi_{1a}\wedge\w_{a2}+\w_{1a}\wedge\varphi_{a2})=0, \;\;\;\;\;\;\;\;\;\;\;\;\;     \;\;\;
d\varphi_{34}=\sum_{j=1}^2(\varphi_{3j}\wedge\w_{j4}+\w_{3j}\wedge\varphi_{j4}).\;\;\; \label{df1234}
\end{eqnarray}

\item Let $\{e_1,e_2\}$, $\{e_3,e_4\}$ be positively oriented 
orthonormal frame fields of $TM$ and $N_fM$, respectively, 
and $\{\w_{kl}\}_{1\leq k,l\leq 4}$ the connection forms of the dual frame field of $\{e_{k}\}_{1\leq k\leq 4}$. 
To every solution $\{\varphi_{kl}\}_{1\leq k,l\leq 4}$ of the fundamental system corresponds a unique  
bending field $\mathcal{T}$. Moreover, for the infinitesimal isometric deformation $F$ determined by $\mathcal{T}$,
the frame field $\{e_{k}\}_{1\leq k\leq 4}$ locally extends to an adapted to $F$ local orthonormal frame field,
such that the variations of the connection forms of its corresponding coframe are the $\{\varphi_{kl}\}_{1\leq k,l\leq 4}$.
\end{enumerate}
\end{theorem}

\begin{proof}
(i) The Taylor expansions of the connection forms are 
\be \label{wklt}
\w_{kl}(t)=\w_{kl}+t\varphi_{kl}+o(t),\;\;\; 1\leq k,l\leq 4,
\ee
which imply that $\varphi_{kl}=-\varphi_{lk}$, $1\leq k,l\leq 4$.
In particular, Lemma \ref{invfr}(i) yields that $\varphi_{12}=0$ and \eqref{vf12} follows.
Taking into account that $e_j(t)=e_j+o(t)$, $j=1,2$, and
using \eqref{wklt} to compare the coefficients of $t$ in the relations $\w_{ja}(t)(e_r(t))=\w_{ra}(t)(e_j(t))$, $j,r=1,2$, for $a=3,4$,
we obtain \eqref{vfja}.
The remaining equations of the fundamental system follow by using \eqref{wklt} and comparing the $t$-terms in the relations
$d\w_{kl}(t)=\sum_{m=1}^4 \w_{km}(t)\wedge \w_{ml}(t)$, for $1\leq k,l\leq 4$.

(ii) For a solution $\{\varphi_{kl}\}_{1\leq k,l\leq 4}$ of the fundamental system, consider the sections
$$V_j=-\sum_{1\leq k<l\leq 4}\varphi_{kl}(e_j)\varepsilon_k\wedge\varepsilon_l, \;\;\;\;\; j=1,2,$$
of $f^*(\Lambda^2T\R^4)$, where $\varepsilon_j=f_*e_j$, $j=1,2$, and $\varepsilon_a=e_a$, $a=3,4$. Since the bundle $f^*(\Lambda^2T\R^4)$ is flat, there exists
a parallel vector bundle isometry $P\colon f^*(\Lambda^2T\R^4)\to M\times \R^6$, where $M\times \R^6$ is the trivial bundle over $M$,
equipped with its canonical connection $\bar{\nabla}$.
Consider the 1-form $\omega \in \Gamma(T^*M\otimes \R^6)$
given by $\omega=\tilde{V}_1\w_1+\tilde{V}_2\w_2$, where $\tilde{V}_j=PV_j$, $j=1,2$. 
Its exterior derivative satisfies
\bea
d\omega(e_1,e_2)&=& \bar{\nabla}_{e_1}\w(e_2)-\bar{\nabla}_{e_2}\w(e_1)-\w([e_1,e_2])\\
&=& \bar{\nabla}_{e_1}\tilde{V}_2-\bar{\nabla}_{e_2}\tilde{V}_1+\w_{12}(e_1)\tilde{V}_1+\w_{12}(e_2)\tilde{V}_2\\
&=& P\left(\hat{\nabla}_{e_1}V_2-\hat{\nabla}_{e_2}V_1 +\w_{12}(e_1)V_1+\w_{12}(e_2)V_2\right).
\eea
Using \eqref{G-W} and all the equations of the fundamental system apart from \eqref{vfja}, 
it follows that the quantity in the last parenthesis is
equal to zero and therefore, $\w$ is closed. Since $M$ is simply-connected, there exists a unique, up to a constant vector in $\R^6$,
section $\tilde{V}\colon M\to \R^6$ such that $d\tilde{V}=\w$.
This implies that $V=P^{-1}\tilde{V}$ is the unique, up to a constant vector in $\Lambda^2\R^4$, section of $f^*(\Lambda^2T\R^4)$ satisfying
\be \label{Vis}
\hat{\nabla}_{e_j}V=V_j, \;\;\; j=1,2.
\ee

Consider furthermore the sections $\mathcal{T}_j=V\cdot \varepsilon_j$, $j=1,2$, of $f^*(T\R^4)$.
Using \eqref{vfja} it follows that
$$\tilde{\nabla}_{e_1}\mathcal{T}_2-\tilde{\nabla}_{e_2}\mathcal{T}_1 +\w_{12}(e_1)\mathcal{T}_1+\w_{12}(e_2)\mathcal{T}_2=0,$$
and since the bundle $f^*(T\R^4)$ is flat, arguing as above, we conclude that there exists a unique, up to a constant vector in $\R^4$,
section $\mathcal{T}$ of $f^*(T\R^4)$ such that 
\be \label{bfldi}
\tilde{\nabla}_{e_j}\mathcal{T}=\mathcal{T}_j=V\cdot \varepsilon_j,\;\;\; j=1,2.
\ee
Using the second equation in \eqref{skew}, the above implies that $\mathcal{T}$ is a bending field. 
In particular, $\mathcal{T}$ is uniquely determined up to a trivial bending field.

Let $F\colon I\times M\to \R^4$ be the infinitesimal isometric deformation determined by $\mathcal{T}$. 
Lemma \ref{invfr} implies that  
$\{e_k\}_{1\leq k\leq4}$ locally extends to an adapted to $F$ local orthonormal frame field
$\{\tilde{e}_k(t)\}_{1\leq k\leq4}$. For simplicity, we may assume that this occurs globally. 
Let $W$ be the section of Lemma \ref{varfraij} corresponding to $\{\tilde{e}_k(t)\}_{1\leq k\leq4}$. 
From \eqref{vframe} and \eqref{bfldi} it follows that $(W-V)\cdot \varepsilon_j=0$, $j=1,2$. 
Therefore, $W=V-u\varepsilon_3\wedge \varepsilon_4$ for some $u\in\mathcal{C}^\infty(M)$. 
From \eqref{fklW}, by differentiating the last relation and using \eqref{Vis} and \eqref{G-W},
we obtain that the variations
$\{\tilde{\varphi}_{kl}\}_{1\leq k,l\leq 4}$ of the connection forms $\{\tilde{\w}_{kl}(t)\}_{1\leq k,l\leq 4}$
of the dual frame field of $\{\tilde{e}_k(t)\}_{1\leq k\leq4}$, are given by 
\be \label{tilde fkl}
\tilde{\varphi}_{j3}=\varphi_{j3}+u\w_{j4},\;\;\;\;\; \tilde{\varphi}_{j4}=\varphi_{j4}-u\w_{j3},
\;\; j=1,2,\;\;\;\mbox{and}\;\;\;\;\; \tilde{\varphi}_{34}=\varphi_{34}+du.
\ee
Consider the adapted to $F$ orthonormal frame field $\{e_k(t)\}_{1\leq k\leq4}$, given by $e_j(t)=\tilde{e}_j(t)$, $j=1,2$,
and $e_3(t)+ie_4(t)=\exp(itu)(\tilde{e}_3(t)+i\tilde{e}_4(t))$. 
For the connection forms $\{\w_{kl}(t)\}_{1\leq k,l\leq 4}$ of its corresponding coframe we have
$$\w_{j3}(t)+i\w_{j4}(t)=e^{itu}(\tilde{\w}_{j3}(t)+i\tilde{\w}_{j4}(t)),\;\;j=1,2,\;\;\;\;\mbox{and}\;\;\;\; \w_{34}(t)=\tilde{\w}_{34}(t)-tdu.$$
Differentiating the above relations with respect to $t$ and using \eqref{tilde fkl}, it follows that $\delta \w_{kl}(t)=\varphi_{kl}$, $1\leq k,l\leq4$,
and this completes the proof.
\qed
\end{proof}

\medskip

\begin{corollary} \label{Trivial Inf}
Let $F\colon I\times M\to \R^4$ be an infinitesimal isometric deformation of a simply-connected oriented surface.
The deformation is trivial if and only if for every adapted to $F$ orthonormal frame field
$\{e_k(t)\}_{1\leq k\leq 4}$, the variations 
of the connection forms of its corresponding coframe have the form 
$$\varphi_{j3}=u\w_{j4},\;\;\;\;\; \varphi_{j4}=-u\w_{j3},\;\; j=1,2,\;\;\;\mbox{and}\;\;\;\;\; \varphi_{34}=du,$$
for some $u\in\mathcal{C}^\infty(M)$. 
Moreover, if $F$ is trivial and $\{e_1,e_2\}$, $\{e_3,e_4\}$ are positively oriented orthonormal frame fields of $TM$ and $N_fM$, respectively,
then the frame field $\{e_k\}_{1\leq k\leq 4}$ locally extends to an adapted to $F$ local orthonormal frame field,
such that the variations of all of the connection forms of its corresponding coframe vanish.
In particular, $\varphi_{34}=0$ implies that $\varphi_{kl}=0$, $1\leq k,l\leq4$.
\end{corollary}

\begin{proof}
Let $\{e_k(t)\}_{1\leq k\leq 4}$ be an adapted to $F$ orthonormal frame field, and consider the corresponding section $W$ of Lemma \ref{varfraij}.
The bending field $\mathcal{T}$ of $F$ is trivial if and only if $\tilde{\nabla}_{e_j}\mathcal{T}=C\cdot\varepsilon_j$, $j=1,2$,
where $C$ is a constant vector in $\Lambda^2\R^4$.
Using \eqref{vframe}, this is equivalent to $(W-C)\cdot \varepsilon_j=0$, $j=1,2$. 
The last relation holds if and only if $W=C-u\varepsilon_3\wedge\varepsilon_4$, for some $u\in\mathcal{C}^\infty(M)$. 
The rest of the proof follows by repeating the part of the proof of Theorem \ref{FTHMINF}(ii), concerning the infinitesimal deformation.
\qed
\end{proof}

\medskip

Let $F\colon I\times M\to \R^4$ be an infinitesimal isometric deformation. Consider a smooth with respect to $t$,
section $\xi(t)\in N_{f_t}M_t$.
We say that $F$ {\emph{preserves $\xi(t)$ parallelly in the normal bundle}}, if around every point of $M$,
there exists a local orthonormal frame field $\{e_3(t),e_4(t)\}$ of $N_{f_t}M_t$ that depends smoothly on $t$, such that
$$\delta\w_{34}(t)=0\;\;\;\;\; \mbox{and}\;\;\;\;\; \delta \langle \xi(t), e_a(t)\rangle=0,\;\;\; a=3,4,$$
where $\w_{34}(t)$ is the connection form of the dual frame field of $\{e_3(t),e_4(t)\}$. 
If $\xi(t)=H_{f_t}$, we say that $F$ preserves the mean curvature vector field parallelly  in the normal bundle.

Let $\Psi(t)$ be a $N_{f_t}M_t\otimes\mathbb{C}$-valued quadratic differential that depends smoothly on $t$.
If $\{e_1(t),e_2(t)\}$ is a smooth with respect to $t$, positively oriented local orthonormal frame field on $M_t$ 
with dual frame field $\{\w_1(t),\w_2(t)\}$, then $\Psi(t)$ has the local expression
$$\Psi(t)=\psi(t)(\w_1(t)+i\w_2(t))^2,\;\;\; \psi(t)\in N_{f_t}M_t\otimes\mathbb{C}.$$
We say that $F$ {\emph{preserves $\Psi(t)$ parallelly in the normal bundle}}, if for every
local orthonormal frame field $\{e_1(t),e_2(t)\}$ on $M_t$ satisfying \eqref{delta int},
$F$ preserves parallelly in the normal bundle the real and the imaginary parts of $\psi(t)$.
In particular, if $\Psi(t)=\Phi^\pm(t)$, where $\Phi(t)$ is the Hopf differential of $f_t$,
we say that $F$ preserves parallelly in the normal bundle, the differential $\Phi^\pm$.

\begin{proposition} \label{IDP}
Let $F\colon I\times M\to \R^4$ be an infinitesimal isometric deformation of a simply-connected oriented surface.
Suppose that $\{e_k(t)\}_{1\leq k\leq 4}$ is an adapted to $F$ orthonormal frame field with $\varphi_{34}=0$. Then:
\begin{enumerate}[topsep=0pt,itemsep=-1pt,partopsep=1ex,parsep=0.5ex,leftmargin=*, label=(\roman*), align=left, labelsep=-0.4em]
\item The deformation $F$ preserves parallelly in the normal bundle, the mean curvature vector field and 
the differential $\Phi^\pm$ if and only if
\be \label{ff23}
\varphi_{13}=\star\varphi_{23}\;\;\;\;\; \mbox{and} \;\;\;\;\; \varphi_{14}=\star\varphi_{24}=\mp\varphi_{23}.
\ee
\item Assume that $M_0^\mp(f)=\emptyset$. If \eqref{ff23} holds and $\varphi_{23}$ is nowhere-vanishing, then
\be \label{f23L}
\varphi_{23}=L\left(\cos\phi\w_1\mp \sin\phi\w_2\right),
\ee
for functions $L>0$ and $\phi$ on $M$, satisfying
\be \label{logLOm}
d\log L=\star\Omega^\mp\;\;\;\;\;\mbox{and}\;\;\;\;\; e_3^\mp=\cos\phi e_3+\sin\phi e_4,
\ee
where $\{\w_1,\w_2\}$ is the dual frame field of $\{e_1,e_2\}$, and $e_3^\mp$ is given by \eqref{HI}. 
\end{enumerate}
\end{proposition}

\begin{proof}
(i) From \eqref{Hopf Invariants}, it follows that the differential $\Phi^\pm(t)$ is given by 
$$\Phi^\pm(t)=\frac{1}{2}\left(\psi^\pm(t)\pm iJ_t^\perp \psi^\pm(t) \right)(\w_1(t)+i\w_2(t))^2,$$
where $J_t^\perp$ is the complex structure of $N_{f_t}M_t$,
$$\psi^\pm(t)=\sum_{a=3}^{4}\psi^\pm_a(t)e_a(t)=\frac{\a_{11}(t)-\a_{22}(t)}{2}\pm J_t^\perp \a_{12}(t),$$
and $\a_{jr}(t)=\a_{f_t}(e_j(t),e_r(t))$, $j,r=1,2$, where $\a_{f_t}$ is the second fundamental form of $f_t$.
Equation \eqref{G-W} yields that 
\be \label{ajrt}
\a_{jr}(t)=\sum_{a=3}^{4}\w_{ja}(t)(e_r(t))e_a(t),\;\;\; j,r=1,2.
\ee
Using the above, it follows that
$$\psi^\pm_a(t)=\frac{\w_{1a}(t)(e_1(t))-\w_{2a}(t)(e_2(t))}{2}\pm (-1)^a\w_{1b}(e_2(t)),\;\;\;\;\; a,b=3,4,\;\; b\neq a,$$
and that the components $H_a(t)$, $a=3,4$, of the mean curvature vector field $H_{f_t}$ of $f_t$,
with respect to the frame field $\{e_3(t), e_4(t)\}$, are given by
$$H_a(t)=\langle H_{f_t},e_a(t)\rangle=\frac{\w_{1a}(t)(e_1(t))+\w_{2a}(t)(e_2(t))}{2},\;\;\; a=3,4.$$
Therefore, we have that
\bea
\delta \psi^\pm_a(t)&=&\frac{\varphi_{1a}(e_1)-\varphi_{2a}(e_2)}{2}\pm (-1)^a\varphi_{1b}(e_2),\;\;\;\;\; a,b=3,4,\;\; b\neq a,\\
\delta H_a(t)&=&\frac{\varphi_{1a}(e_1)+\varphi_{2a}(e_2)}{2},\;\;\;\;\;\;\;\;\;\;\;\;\;\;\;\;\;\;\;\;\;\;\;\;\;\;\;\;\; a=3,4.
\eea
Taking into account \eqref{vfja}, it follows that $\delta H_a(t)=0$ is equivalent to $\varphi_{1a}=\star\varphi_{2a}$, $a=3,4$.
Moreover, it is clear that $F$ preserves $\Phi^\pm$ if and only if it preserves the section $\psi^\pm(t)$, parallelly in the normal bundle.
Provided $\delta H_a(t)=0$, it follows that the equations $\delta \psi^\pm_a(t)=0$, $a=3,4$, are equivalent to $\varphi_{14}=\mp\varphi_{23}$, and this completes the proof.

(ii) Consider $\phi\in \mathcal{C}^\infty(M)$ that satisfy the second equation in \eqref{logLOm}. By substituting $\a_{jr}$, $j,r=1,2$, from 
\eqref{ajrt} for $t=0$, into \eqref{e34pm}, we obtain that
\bea
2\|\mathcal{H}^\mp\|\cos\phi&=&(\w_{13}\pm\w_{24})(e_1)-(\w_{23}\mp\w_{14})(e_2),\\
2\|\mathcal{H}^\mp\|\sin\phi&=&(\w_{14}\mp\w_{23})(e_1)-(\w_{24}\pm\w_{13})(e_2).
\eea
Using \eqref{ff23} to express all the variations $\varphi_{ja}$, $j=1,2$, $a=3,4$, in terms of  
$\varphi_{23}$, and taking into account that $\varphi_{34}=0$, it follows that the equations in \eqref{df1234} are equivalent.
By virtue of the above relations, it follows that \eqref{df1234} is equivalent to
$$\varphi_{23}(e_1)\sin\phi \pm \varphi_{23}(e_2)\cos\phi=0.$$
Since $\varphi_{23}$ is nowhere-vanishing, the above
implies that there exists a positive $L\in\mathcal{C}^\infty(M)$  
such that \eqref{f23L} is valid. 
It remains to prove that $L$ satisfies the first equation in \eqref{logLOm}.

Consider the coframe $\{\theta_1,\theta_2\}$ on $M$, 
given by $$\theta_1=\varphi_{23},\;\;\;\;\; \theta_2=\star\theta_1,$$
with corresponding connection form $\th_{12}$ determined by the relations 
$d\th_1=\th_{12}\wedge\th_2$ and $d\th_2=-\th_{12}\wedge\th_{1}$.
Using \eqref{f23L}, 
it can be easily deduced that 
$$\th_{12}=\w_{12}\mp d\phi +\star d\log L.$$
On the other hand, by using the first equation in \eqref{ff23}, from \eqref{dfja} we obtain that
$$d\th_j=d\varphi_{r3}=(-1)^r(-\w_{12}\pm \w_{34})\wedge\th_r,\;\;\;\; j,r=1,2,\;\;j\neq r,$$
and thus, $\th_{12}=-\w_{12}\pm \w_{34}$. 
The last relation and the above expression of $\th_{12}$ imply that
$$2\w_{12}\mp (\w_{34}+d\phi) =-\star d\log L.$$
From the second equation in \eqref{logLOm}, we obtain that $\w_{34}+d\phi=\w_{34}^\mp$,
where $\w_{34}^\mp$ is the connection form of the dual frame field of $\{e_3^\mp, e_4^\mp\}$.
By virtue of Proposition \ref{Criterion}(i), the first equation in \eqref{logLOm} follows from the above relation.
\qed
\end{proof}

\medskip

The following is the infinitesimal analogue of the uniqueness part of the fundamental theorem of surfaces in $\R^4$.

\begin{theorem}
An infinitesimal isometric deformation $F\colon I\times M\to \R^4$ of an oriented surface
is trivial if and only if it preserves parallelly in the normal bundle, the mean curvature vector field and the Hopf differential.
\end{theorem}

\begin{proof}
If $F$ is trivial, then Corollary \ref{Trivial Inf} implies that around every point of $M$ there exists an adapted to $F$ local orthonormal 
frame field, such that the variations of all of the connection forms of its corresponding coframe vanish. 
Proposition \ref{IDP}(i) yields that $F$ preserves parallelly in the normal bundle, the mean curvature vector field and both 
isotropic parts of the Hopf differential.

Conversely, assume that $F$ preserves parallelly in the normal bundle, the mean curvature vector field and the Hopf differential.
Then, around every point of $M$ there exists an adapted to $F$ local orthonormal frame field with $\varphi_{34}=0$. 
Since $F$ preserves both isotropic parts of the Hopf differential,
from \eqref{ff23} it follows that $\varphi_{kl}=0$, $1\leq k,l\leq4$. Then, Corollary \ref{Trivial Inf} implies that $F$ is trivial.
\qed
\end{proof}

\medskip

\noindent{\emph{Proof of Theorem \ref{IIID}:}}
Without loss of generality, suppose that $M$ is simply-connected;
otherwise, we argue on a simply-connected neighbourhood around every point of $M$.

Assume that $F\colon I\times M\to \R^4$ is a nontrivial infinitesimal isometric deformation
that preserves parallelly in the normal bundle, the mean curvature vector field and the isotropic part $\Phi^{\pm}$ of the Hopf differential. 
Then, every point of $M$ has a neighbourhood $U$ on which, there exists an adapted to $F$ local orthonormal frame field with $\varphi_{34}=0$. 
Corollary \ref{Trivial Inf} and \eqref{ff23} imply that $\varphi_{23}\neq0$ on the open and dense subset of $U$, on which the bending field of $F$ is nontrivial.
Then, from the first equation in \eqref{logLOm} it follows that $d\star\Omega^{\mp}=0$ on this subset and thus, on $U$. This shows that $f$ is $\mp$ isotropically isothermic.

Conversely, assume that $f$ is $\mp$ isotropically isothermic. 
Then, there exists a smooth positive function $L$ on $M$, satisfying the first equation in \eqref{logLOm}.
Let $\{e_1,e_2\}$ be a positively oriented orthonormal frame field of $TM$.
Since $M_0(f)=\emptyset$, the frame field $\{e_1,e_2\}$ determines the orthonormal frame fields 
$\{e_3^\mp,e_4^\mp\}$ and $\{e_3^\pm,e_4^\pm\}$ of $N_fM$, given by \eqref{HI}.
By setting $e_a=e_a^\pm$, $a=3,4$, we consider $\phi\in\mathcal{C}^\infty(M)$ satisfying the second equation in \eqref{logLOm}.
Then, we define $\varphi_{23}$ by \eqref{f23L}, $\varphi_{34}=0$, and the remaining $\varphi_{kl}$, $1\leq k,l\leq4$,
from equations \eqref{ff23} and \eqref{vf12}. It is straightforward to check that $\{\varphi_{kl}\}_{1\leq k,l\leq4}$
satisfy the fundamental system with respect to the connection forms $\{\w_{kl}\}_{1\leq k,l\leq4}$
of the dual frame field of $\{e_k\}_{1\leq k\leq4}$. From Theorem \ref{FTHMINF}(ii), it follows that 
the solution $\{\varphi_{kl}\}_{1\leq k,l\leq4}$
determines a unique bending field $\mathcal{T}$. In particular, since $\varphi_{34}=0\neq\varphi_{23}$  
everywhere on $M$, Corollary \ref{Trivial Inf} implies that $\mathcal{T}$ is nontrivial.
Moreover, for the infinitesimal isometric deformation determined by $\mathcal{T}$, 
Theorem \ref{FTHMINF}(ii) implies that $\{e_k\}_{1\leq k\leq4}$ locally extends to an adapted to $F$ local 
orthonormal frame field, such that the variations of the connection forms of its corresponding coframe
are the $\{\varphi_{kl}\}_{1\leq k,l\leq4}$. From Proposition \ref{IDP}(i) it follows that $F$ preserves parallelly in the normal bundle,
the mean curvature vector field and the differential $\Phi^\pm$. 
The rest of the proof follows immediately from Proposition \ref{LoC}.
\qed
\medskip

For the proof of Theorem \ref{IIS} we need the following lemma. We recall from Proposition \ref{holomorphic Gl} that
the Gauss lift $G_\pm$ of a superconformal surface $f\colon M\to \R^4$  with $\pm K_N\geq0$, is holomorphic. 

\begin{lemma} \label{LemSup}
Let $f\colon M\to \R^4$ be an oriented superconformal surface with $\pm K_N\geq0$ and nowhere-vanishing mean curvature vector field.
Assume that $F\colon I\times M\to \R^4$ is an infinitesimal isometric deformation that preserves parallelly in the normal bundle the mean curvature vector field.
Then, $F$ preserves parallelly in the normal bundle the differential $\Phi^\pm$ if and only if it preserves the holomorphicity
of the Gauss lift $G_\pm\colon M\to (\mathcal{Z},g_1)$ of $f$.
\end{lemma}

\begin{proof}
Since $F$ preserves parallelly in the normal bundle the mean curvature vector field, around every point of $M$, there exists
an adapted to $F$ local orthonormal frame field $\{e_k(t)\}_{1\leq k\leq 4}$ with $\varphi_{34}=0$. 
In particular, from the proof of Proposition \ref{IDP}(i) it follows that $\varphi_{1a}=\star\varphi_{2a}$, $a=3,4$.

Since $\pm K_N\geq0$, Lemma \ref{pseudo}(ii) implies that $\Phi^\pm\equiv0$. Therefore, using \eqref{ajrt} for $t=0$,
from \eqref{Hopf Invariants} and \eqref{Bpm} we obtain that
$(\w_{23}\pm\w_{14})=\star (\w_{13}\mp\w_{24}).$
Moreover, since $H\neq0$ everywhere on $M$, a simple computation shows that $\w_{13}\mp\w_{24}$ is nowhere-vanishing.

Let $G_\pm(t)$ be the Gauss lift of $f_t$ into $(\mathcal{Z},g_1)$. From \eqref{Gconf} we have that
$$G_{\pm}^*(t)(g_1)=ds_t^2 + \frac{1}{4}\left((\w_{13}(t)\mp\w_{24}(t))^2 + (\w_{23}(t)\pm\w_{14}(t))^2\right).$$
Proposition \ref{holomorphic Gl} implies that $F$ preserves the holomorphicity of $G_\pm$ if and only if $\delta G_{\pm}^*(t)(g_1)=0$.
Differentiating the above with respect to $t$, and using that $\varphi_{1a}=\star\varphi_{2a}$, $a=3,4$, and that
$(\w_{23}\pm\w_{14})=\star (\w_{13}\mp\w_{24})\neq0$ everywhere on $M$, we obtain that 
$\delta G_{\pm}^*(t)(g_1)=0$ is equivalent to \eqref{ff23}. The proof follows from Proposition \ref{IDP}(i).
\qed
\end{proof}

\medskip

\noindent{\emph{Proof of Theorem \ref{IIS}:}}
The equivalence of (i) and (ii) has been proved in Proposition \ref{mdi}.

We argue that (i) is equivalent to (iii).
Since $M_1(f)=\emptyset$, Lemma \ref{pseudo}(ii) yields that $K_N\neq0$ everywhere on $M$ and therefore,
it also implies that either $\Phi^+\equiv0$, or $\Phi^-\equiv0$ on $M$. 
Assume that $\Phi^\pm\equiv0$ on $M$. Since $M_1(f)=\emptyset$, from Lemma \ref{pseudo}(i) it follows that $M_0^\mp(f)=\emptyset$.
Hence, every positively oriented local orthonormal frame field $\{e_1,e_2\}$ of $TM$, determines the local orthonormal frame field 
$\{e_3^\mp,e_4^\mp\}$ of $N_fM$, given by \eqref{HI}.
By virtue of Lemma \ref{LemSup}, the equivalence of (i) and (iii) follows by repeating the proof of Theorem \ref{IIID}, 
using the frame field $\{e_3=H/\|H\|, e_4=J^\perp e_3\}$ instead of $\{e_3^\pm, e_4^\pm\}$, to show the converse implication.
\qed
\medskip

\section{The Moduli Space of Isometric Surfaces with the Same Mean Curvature} \label{s5}

We recall briefly some facts from \cite{PV}, about isometric surfaces in $\Q^4_c$ with the same mean curvature.
Let $M$ be a 2-dimensional oriented Riemannian manifold, and $f,\tilde{f}\colon M\to \Q^4_c$ isometric immersions
with mean curvature vector fields $H$ and $\tilde{H}$, respectively. 
The surfaces $f,\tilde{f}$ are said to have {\emph{the same 
mean curvature}}, if there exists a parallel vector bundle isometry $T\colon N_{f}M\to N_{\tilde f}M$
such that $TH=\tilde{H}$. 
If $f$ and $\tilde{f}$ have the same mean curvature and they are noncongruent, then
the pair $(f,\tilde{f})$ is called a \emph{pair of Bonnet mates}.

Assume that $f,\tilde{f}\colon M \to \Q^4_c$ have the same mean curvature and let $T\colon N_{f}M\to N_{\tilde{f}}M$
be a parallel vector bundle isometry  satisfying $TH=\tilde{H}$.
After an eventual composition of $\tilde{f}$ with an
orientation-reversing isometry of $\Q^4_c$, we may hereafter suppose that $T$ is orientation-preserving.
Let $\a, \tilde{\a}$ be the second fundamental forms of $f$ and $\tilde{f}$, respectively.
The section of $\text{Hom}(TM\times TM,N_{f}M)$ given by 
$$D^T_{f,\tilde{f}}= \a-T^{-1}\circ \tilde{\a}$$ 
is traceless and measures how far the surfaces deviate from being congruent.
Its $\mathbb{C}$-bilinear extension decomposes into its $(k,l)$-components, $k+l=2$, 
and the $(2,0)$-part is given by
$$Q^T_{f,\tilde{f}}=(D^T_{f,\tilde{f}})^{(2,0)}=\Phi-T^{-1}\circ \tilde{\Phi},$$
where $\Phi, \tilde{\Phi}$ are the Hopf differentials of $f$ and $\tilde{f}$, respectively.
The following has been proved in \cite[Lemma 12]{PV}.

\begin{lemma}\label{qke}
Let $f,\tilde{f}\colon M\to \Q^4_c$ be non-minimal 
surfaces and $T\colon N_{f}M\to N_{\tilde{f}}M$ an
orientation-preserving parallel vector bundle isometry satisfying $TH=\tilde{H}$. Then:
\begin{enumerate}[topsep=0pt,itemsep=-1pt,partopsep=1ex,parsep=0.5ex,leftmargin=*, label=(\roman*), align=left, labelsep=-0.4em]
\item The quadratic differential $Q^T_{f,\tilde{f}}$ is holomorphic and independent of $T$. 
\item The normal curvatures of the surfaces are equal and
the curvature ellipses $\mathcal{E}_{f}$, $\mathcal{E}_{\tilde{f}}$ are congruent at any point of $M$.
In particular, $M^{\pm}_0(f)=M^{\pm}_0(\tilde{f})$.
\end{enumerate}
\end{lemma}

By virtue of Lemma \ref{qke}(i), we assign to each pair of non-minimal surfaces $(f,\tilde{f})$ with the same mean curvature,
a holomorphic quadratic differential denoted by $Q_{f,\tilde{f}}$, which is called \emph{the distortion differential of the pair} and is given by
$$Q_{f,\tilde{f}}=\Phi-T^{-1}\circ \tilde{\Phi}.$$
The distortion differential of such a pair is simply denoted by $Q$, whenever there is no danger of confusion.

Let $f,\tilde{f}\colon M \to \Q^4_c$ be non-minimal surfaces with the same mean curvature.
It is clear that $Q\equiv 0$ if and only if $f$ and $\tilde{f}$ are congruent.
If $(f,\tilde{f})$ is a pair of Bonnet mates, then according to Lemmas \ref{zeros} and \ref{qke}(i), the zero-set $Z$ of $Q$ consists of isolated points only.
With respect to the decomposition $N_fM\otimes\mathbb{C}=N_f^{-}M\oplus N_f^{+}M$, the distortion differential splits as 
\begin{eqnarray}
Q=Q^{-}+Q^{+},\;\; \mbox{where }\;\;Q^{\pm}=\pi^{\pm}\circ Q.\nonumber
\end{eqnarray}
From Lemma \ref{qke}(i) it follows that both differentials $Q^{-}$ and $Q^{+}$ are holomorphic, and $Q^{\pm}$ is given by
\be\label{qpr}
Q^{\pm}=\Phi^{\pm}-T^{-1}\circ \tilde{\Phi}^{\pm}.
\ee
Lemma \ref{zeros} implies that either $Q^{\pm}\equiv 0$, or the zero-set $Z^{\pm}$ of $Q^{\pm}$ consists of isolated points only.

For an oriented surface $f\colon M\to \Q^4_c$,
we denote by ${\mathcal M}(f)$ {\emph {the moduli space of congruence classes of all isometric immersions 
of $M$ into $\Q^4_c$, that have the same mean curvature with $f$}}.

Assume that $f\colon M\to \Q^4_c$ is a non-minimal oriented surface. Since the distortion differential of a pair of Bonnet mates does not vanish identically, the moduli space
can be written as 
$$\mathcal{M}(f)=\mathcal{N}^{-}(f)\cup \mathcal{N}^{+}(f)\cup\{f\},$$
where
$$\mathcal{N}^{\pm}(f)=\{\tilde{f}:\; Q^{\pm}_{f, \tilde{f}}\not \equiv 0\}/\text{Isom}^+(\Q^4_c),$$
$\{f\}$ is the trivial congruence class, and $\text{Isom}^+(\Q^4_c)$ is the group of orientation-preserving isometries of $\Q^4_c$.
Moreover, the moduli space decomposes into disjoint components as
$$\mathcal{M}(f)= \mathcal{M}^*(f)\cup \mathcal{M}^{-}(f)\cup \mathcal{M}^{+}(f)\cup\{f\},$$ 
where 
$$\mathcal{M}^{\pm}(f)=\mathcal{N}^{\pm}(f)\smallsetminus \mathcal{N}^{\mp}(f)= \{\tilde{f}:\; Q_{f, \tilde{f}}\equiv Q^{\pm}_{f, \tilde{f}}\}/\text{Isom}^+(\Q^4_c),$$
and
$$\mathcal{M}^*(f)= \mathcal{N}^{-}(f)\cap \mathcal{N}^{+}(f)=
\{\tilde{f}:\; Q^{-}_{f, \tilde{f}}\not\equiv 0\;\; \mbox{and}\;\; Q^{+}_{f, \tilde{f}}\not\equiv 0\}/\text{Isom}^+(\Q^4_c).$$
In order to simplify the notation in the sequel, we set $\bar{\mathcal{M}}^{\pm}(f)=\mathcal{M}^{\pm}(f)\cup\{f\}$.

Hereafter, whenever we refer to a surface in the moduli space we mean its congruence class.
A surface $f\colon M\to \Q^4_c$ is called a {\emph{Bonnet surface}} if $\mathcal{M}(f)\smallsetminus \{f\}\neq \emptyset$.
Any $\tilde{f}\in \mathcal{M}(f)\smallsetminus \{f\}$ is called a {\emph{Bonnet mate}} of $f$. 
A Bonnet surface $f$ is called {\emph{proper Bonnet}} if it admits infinitely many Bonnet mates.

\subsection{Bonnet Mates}

In view of Lemma \ref{qke}(ii), we denote by $M_0=M_0^{-}\cup M_0^{+}$ and $M_1$, the set of pseudo-umbilic and umbilic points of 
a pair of non-minimal Bonnet mates, respectively.

\begin{proposition}\label{dd}
If $\tilde{f}\in \mathcal{N}^{\pm}(f)$, then there exists $\th^{\pm} \in \mathcal{C}^{\infty}(M\smallsetminus M_0^{\pm};(0,2\pi))$,
such that the distortion differential of the pair $(f,\tilde{f})$ satisfies on $M\smallsetminus M_0^{\pm}$ the relation
\be \label{general1}
Q^{\pm}=(1-e^{\mp i\th^{\pm}})\Phi^{\pm}.
\ee
Moreover, $Q^{\pm}$ vanishes precisely on $M_0^{\pm}$, which consists of isolated points only.
\end{proposition}

\begin{proof}
From \cite[Lemma 14, Prop. 15]{PV} it follows that $M_0^{\pm}\subset Z^{\pm}$ is isolated, and 
there exists $\th^{\pm} \in \mathcal{C}^{\infty}(M\smallsetminus Z^{\pm};(0,2\pi))$ 
such that \eqref{general1} is valid on $M\smallsetminus Z^{\pm}$. 
It remains to prove that $M_0^{\pm}=Z^{\pm}$.

Arguing indirectly, assume that there exists $p\in Z^{\pm}\smallsetminus M_0^{\pm}$.
Then, Lemma \ref{pseudo}(i) implies that $\Phi^{\pm}(p)\neq0$.
Since $Q^{\pm}$ and $\Phi^{\pm}$ are smooth and $\Phi^{\pm}(p)\neq0$, from \eqref{general1}
it follows that the function $k=\exp{(\mp i\th^{\pm})}$ extends smoothly at $p$, with $k(p)=1$.

We claim that $\th^{\pm}$ extends smoothly at $p$. We first show that the limit of $\th^{\pm}$ at $p$ exists;
assume to the contrary that there exist sequences
$p_n,q_n \in M\smallsetminus Z^{\pm}$, $n\in \mathbb{N}$, converging at $p$, such that $\th^{\pm}(p_n)\to0$ and $\th^{\pm}(q_n)\to2\pi.$
Since $\th^{\pm}$ is continuous on $M\smallsetminus Z^{\pm}$, 
for every $r>0$
there exists $s_r\in B_r(p)\smallsetminus\{p\}$ such that $\th^{\pm}(s_r)=\pi$, or equivalently, $k(s_r)=-1$. 
On the other hand, since $k$ is continuous at $p$, there exists $\tilde{r}>0$ such that $|k-1|<1/2$ on $B_{\tilde{r}}(p)$.
This is a contradiction and thus, the limit of $\th^{\pm}$ at $p$ exists. 
Since $k$ is smooth and $\th^{\pm}$ extends continuously at $p$, the claim follows.

Let $(U,z)$ be a complex chart with $U\cap Z^{\pm}=\{p\}$. From Lemmas \ref{qke}(i) and \ref{zeros} it follows that there exists
a positive integer $m$ such that $Q^{\pm}=z^m\Psi^{\pm}$ on $U$, and $\Psi^{\pm}(p)\neq0$. Using \eqref{general1}, this is equivalent to
\be \label{neq0}
(1-e^{\mp i\th^{\pm}})\phi^{\pm}=z^m\psi^{\pm},\;\;\; \psi^{\pm}(p)\neq0,
\ee
where $\phi^{\pm}$ is given by \eqref{phipm}, and $\Psi^{\pm}=\psi^{\pm}dz^2$ on $U$.
Differentiating (\ref{general1}) with respect to $\bar \d$ in the normal connection and using the holomorphicity of $Q^{\pm}$, we obtain
\begin{eqnarray*}
\left(h^{\pm}(1-e^{\mp i\th^{\pm}})\pm ie^{\mp i\th^{\pm}}\th^{\pm}_{\bar z}\right)\phi^{\pm}=0,
\end{eqnarray*}
where $h^{\pm}$ is given by \eqref{hpm}. Since $\phi^{\pm}\neq0$ everywhere on $U$, the above implies that
$$\th^{\pm}_{\bar z}=\mp ih^{\pm}(1-e^{\pm i\th^{\pm}}),\;\;\; \th^{\pm}_{z}=\pm i\overline{h^{\pm}}(1-e^{\mp i\th^{\pm}}).$$
Using that $\th^{\pm}(p)=0$ or $2\pi$, from the above relation we obtain that all derivatives of $\th^{\pm}$ vanish at $p$.
Therefore, differentiation of \eqref{neq0} $m$-times with respect to $\d$ in the normal connection yields that $m!\psi^{\pm}(p)=0$.
This is a contradiction, and the proof follows.
\qed
\end{proof}
\medskip

The following lemma is essential for our results.

\begin{lemma} \label{Sys}
Let $M$ be a simply-connected, oriented 2-dimensional Riemannian manifold with a global complex coordinate $z$, and 
$f\colon M\to \Q^4_c$ a surface with $M_0^{\pm}(f)$ isolated.
Consider the differential equation
\be \label{sys}
\th^{\pm}_{\bar z}=\mp ih^{\pm}(1-e^{\pm i\th^{\pm}}),\;\;\; \th^{\pm}_{z}=\pm i\overline{h^{\pm}}(1-e^{\mp i\th^{\pm}}),
\ee
where $h^{\pm}$ is given by \eqref{hpm} on $M\smallsetminus M_0^{\pm}(f)$, and $\th^{\pm} \in \mathcal{C}^{\infty}(M\smallsetminus M_0^{\pm}(f);\R)$. 
Then, the graph of any solution  of \eqref{sys} is an integral surface of the  
distribution $D^{\pm}$ on $\R\times (M\smallsetminus M_0^{\pm}(f))$, defined by the 1-form
\be \label{system}
\rho^{\pm}= d\th^{\pm}\mp i\overline{h^{\pm}}(1-e^{\mp i\th^{\pm}})dz\pm ih^{\pm}(1-e^{\pm i\th^{\pm}})d{\bar z}.
\ee
We have that:
\begin{enumerate}[topsep=0pt,itemsep=-1pt,partopsep=1ex,parsep=0.5ex,leftmargin=*, label=(\roman*), align=left, labelsep=0em]
\item Any solution $\th^{\pm} \in \mathcal{C}^{\infty}(M\smallsetminus M_0^{\pm}(f);\R)$ of \eqref{sys}, satisfies the equations
\be \label{B}
A^{\pm}e^{\pm2i\th^{\pm}}-2i(\Imag A^{\pm})e^{\pm i\th^{\pm}}-{\overline{A^{\pm}}}=0,
\ee
\be \label{harmonic}
\th^{\pm}_{z\overline{z}}=\mp A^{\pm}(1-e^{\pm i\th^{\pm}}),
\ee
where 
\be \label{Ah}
A^{\pm}=i\left(h^{\pm}_z-|h^{\pm}|^2\right)=-\Imag h^{\pm}_z+i(\Real h^{\pm}_z-|h^{\pm}|^2).
\ee
\item Assume that $h^{\pm}$ extends smoothly on $M$.  
Then, $D^{\pm}$ is involutive on $\R\times M$ if and only if $A^{\pm}\equiv0$ on $M$.
If $D^{\pm}$ is involutive, then its maximal integral surfaces are graphs of solutions of \eqref{sys} on $M$.
In particular, any solution of \eqref{sys} on $M$
is equivalent modulo $2\pi$, either to a harmonic function $\th^{\pm} \in \mathcal{C}^{\infty}(M;(0,2\pi))$, 
or  to the constant function $\th^{\pm}\equiv0$, and the space of the distinct modulo $2\pi$ solutions can be smoothly
parametrized by $\mathbb{S}^1\simeq \R/2\pi\mathbb{Z}$.
\item If \eqref{sys} has a harmonic solution $\th^{\pm} \in \mathcal{C}^{\infty}(M\smallsetminus M_0^{\pm}(f);(0,2\pi))$, 
then $h^{\pm}$ extends smoothly on $M$, and $A^{\pm}\equiv0$. 
\end{enumerate}
\end{lemma}

\begin{proof}
It is clear that the graph of any solution of \eqref{sys} is an integral surface of $D^{\pm}$.

(i) Assume that $\th^{\pm} \in \mathcal{C}^{\infty}(M\smallsetminus M_0^{\pm}(f);\R)$ satisfies \eqref{sys}.
From \eqref{sys} it follows that $$\th^{\pm}_{\overline{z}z}=\mp A^{\pm}(1-e^{\pm i\th^{\pm}})\;\;\; \mbox{and}\;\;\; 
\th^{\pm}_{z\overline{z}}=\mp \overline{A^{\pm}}(1-e^{\mp i\th^{\pm}}),$$
where $A^{\pm}$ is given by \eqref{Ah}. Since $\th^{\pm}_{\overline{z}z}=\th^{\pm}_{z\overline{z}}$, the above implies \eqref{B} and \eqref{harmonic}.

(ii) From \eqref{system} and \eqref{Ah} it follows that $\rho^{\pm}$ and $A^{\pm}$ can be smoothly extended on $\R\times M$ and $M$, respectively. 
The Frobenius Theorem yields that $D^{\pm}$ is involutive if and only if $\rho^{\pm}\wedge d\rho^{\pm}\equiv0$ on $\R\times M$, 
or equivalently, $A^{\pm}\equiv0$ on $M$.

Assume that $D^{\pm}$ is involutive on $\R\times M$ and let $\Sigma$ be a maximal integral surface.
Then $\rho^{\pm}=0$ on $\Sigma$. Since $M$ is simply-connected and $\rho^{\pm}$ is defined globally on $\R\times M$, from
\eqref{system} it follows that $\Sigma$ is the graph of a solution of \eqref{sys} on $M$.

Let $\th^{\pm} \in \mathcal{C}^{\infty}(M;\R)$ be a solution of \eqref{sys} on $M$.
Since $A^{\pm}\equiv0$ on $M$, from \eqref{harmonic} it follows that $\th^{\pm}$ is harmonic.
Clearly, $\th^{\pm}+2k\pi$ also satisfies \eqref{sys} for every $k\in \mathbb{Z}$.
Therefore, if $\th^{\pm}\not \equiv 0\mod 2\pi$, we may assume that $\th^{\pm}(p)\in (0,2\pi)$ at some $p\in M$.
Then, the graph of $\theta^{\pm}$ must lie between the graphs of the constant solutions 0 and $2\pi$ and thus, $\th^{\pm}$ takes values in $(0,2\pi)$.
Therefore, any solution of \eqref{sys} on $M$ is equivalent modulo $2\pi$, either to a harmonic function 
$\th^{\pm} \in \mathcal{C}^{\infty}(M;(0,2\pi))$, or  to the function $\th^{\pm}\equiv0$.

Since $\R\times M$ is foliated by maximal integral surfaces of $D^{\pm}$, 
which are graphs over $M$ of solutions of \eqref{sys}, it follows that
the space of these surfaces can be parametrized by a smooth curve $\gamma(t)=(t,p), t\in \R$, where $p\in M$ is an arbitrary point.
Obviously, the space of the distinct modulo $2\pi$ solutions of \eqref{sys} is smoothly parametrized by $\mathbb{S}^1\simeq \R/2\pi\mathbb{Z}$. 

(iii) Let $\th^{\pm} \in \mathcal{C}^{\infty}(M\smallsetminus M_0^{\pm}(f);(0,2\pi))$ be a harmonic function satisfying \eqref{sys}. 
Since $\th^{\pm}$ is bounded with isolated singularities, it extends
to a harmonic function $\th^{\pm} \in \mathcal{C}^{\infty}(M;[0,2\pi])$.
We claim that $\th^{\pm}$ does not attain the values $0$ and $2\pi$ on $M$. Arguing indirectly, assume that there exists a point 
at which $\th^{\pm}$ attains the value $0$ or $2\pi$. Then $\th^{\pm}$ has an interior minimum or maximum, respectively,
and the maximum principle implies that $\th^{\pm}\equiv0$ or $2\pi$, respectively, on $M$. 
This is a contradiction, since $\th^{\pm}(p)\in (0,2\pi)$ for every $p\in M\smallsetminus M_0^{\pm}(f)$.
Therefore, $\th^{\pm} \in \mathcal{C}^{\infty}(M;(0,2\pi))$. From \eqref{sys}, it follows that $h^{\pm}$ extends smoothly at 
every point of $M_0^{\pm}(f)$. Since $\th^{\pm}$ is harmonic, \eqref{harmonic} implies that $A^{\pm}\equiv0$ on $M$.
\qed
\end{proof}
\medskip

\begin{proposition}\label{Npm}
If $\tilde{f}\in \mathcal{N}^{\pm}(f)$, then the function $\th^{\pm}$ of Proposition \ref{dd} satisfies \eqref{sys} 
on $U\smallsetminus M_0^{\pm}$ for every simply-connected complex chart $(U,z)$ on $M$. 
Moreover, if one of the following holds, then it extends to a harmonic function $\th^{\pm}\in \mathcal{C}^{\infty}(M;(0,2\pi))$.
\begin{enumerate}[topsep=0pt,itemsep=-1pt,partopsep=1ex,parsep=0.5ex,leftmargin=*, label=(\roman*), align=left, labelsep=-0.4em]
\item There exists $\hat{f}\in\mathcal{N}^{\pm}(f)\cap\mathcal{N}^{\pm}(\tilde{f})$.
\item The surface $f$ is $\pm$ isotropically isothermic on $M\smallsetminus M_0^{\pm}$.
\end{enumerate}
\end{proposition}

\begin{proof}
Let $(U,z)$ be a simply-connected complex chart on $M$. 
In the proof of Proposition \ref{dd} it has been shown that $\th^{\pm}$ satisfies 
\eqref{sys} on $U\smallsetminus M_0^{\pm}$. We claim that if (i) or (ii) holds, then $\th^{\pm}$ is harmonic on $U\smallsetminus M_0^{\pm}$.

(i) To unify the notation, set $f_1=\tilde{f}$, $\th^{\pm}_1=\th^{\pm}$ and $f_2=\hat{f}$.
Proposition \ref{dd} implies that there exists 
$\th^{\pm}_j\in \mathcal{C}^{\infty}(M\smallsetminus M_0^{\pm};(0,2\pi))$ such that the distortion differential $Q_j$ of the pair $(f,f_j)$
satisfies  
\bea
Q^{\pm}_{j}= (1-e^{\mp i\th^{\pm}_{j}})\Phi^{\pm} \;\;\;\mbox{on}\;\;\; M\smallsetminus M_0^{\pm}
\eea
for $j=1,2$, where $\Phi$ is the Hopf differential of $f$.
Moreover (cf. \cite[Lemma 17]{PV} and its proof), the distortion differential $Q$ of the pair $(f_1,f_2)$ satisfies
$$Q^{\pm}=T\circ (Q^{\pm}_1-Q^{\pm}_2),$$
where $T\colon N_fM\to N_{f_1}M$ is an orientation and mean curvature vector field-preserving, parallel vector bundle isometry.
From the above two relations it follows that
$$Q^{\pm}=(e^{\mp i\th^{\pm}_2}-e^{\mp i\th^{\pm}_1})T\circ \Phi^{\pm}\;\;\; \mbox{on}\;\;\; M\smallsetminus M_0^{\pm}.$$
Since $f_2\in \mathcal{N}^{\pm}(f_1)$, it is clear that $f_1\in \mathcal{N}^{\pm}(f_2)$.
Proposition \ref{dd} implies that $Q^{\pm}$ vanishes precisely on $M_0^{\pm}$ and from the above it follows that 
$\th^{\pm}_1\neq \th^{\pm}_2$ everywhere on $M\smallsetminus M_0^{\pm}$. 
Since $\th^{\pm}_j, j=1,2,$ satisfies \eqref{sys} on $U\smallsetminus M_0^{\pm}$, from Lemma \ref{Sys}(i) it follows that it also satisfies \eqref{B}.
At every point of $U\smallsetminus M_0^{\pm}$, equation \eqref{B} viewed as a polynomial equation, has the distinct roots 
$1,e^{\mp i\th^{\pm}_1},e^{\mp i\th^{\pm}_2}$. Hence, $A^{\pm}\equiv0$ on $U\smallsetminus M_0^{\pm}$ and the claim follows by virtue of \eqref{harmonic}. 

(ii) Arguing indirectly, assume that $\th^{\pm}$ is not harmonic on $U\smallsetminus M_0^{\pm}$. Appealing to Lemma \ref{Sys}(i), 
equation \eqref{harmonic} implies that there exists $p\in U\smallsetminus M_0^{\pm}$ such that $A^{\pm}(p)\neq0$. On the other hand,
Lemma \ref{qiz} and \eqref{Ah} yield that $\Real A^{\pm}\equiv0$ on $U\smallsetminus M_0^{\pm}$.  
Since $\Real A^{\pm}(p)=0\neq\Imag A^{\pm}(p)$, equation \eqref{B} implies that $\exp{(\pm i\th^{\pm}(p))}=1$. 
This is a contradiction since $\th^{\pm}$ takes values in $(0,2\pi)$, and the claim follows.

Since $\th^{\pm}$ is a harmonic function satisfying \eqref{sys} on $U\smallsetminus M_0^{\pm}$, Lemma \ref{Sys}(iii) implies that
$h^\pm$ extends smoothly on $U$ and $A^\pm\equiv0$ on $U$. From Lemma \ref{Sys}(ii) it follows that $\th^{\pm}$ extends to a harmonic function
on $U$ with values in $(0,2\pi)$, satisfying \eqref{sys} on $U$. Since $U$ is arbitrary, this completes the proof.
\qed
\end{proof}

\section{Simply-Connected Surfaces} \label{s6}

\subsection{The Structure of the Moduli Space}

We study here the moduli space $\mathcal{M}(f)$ for simply-connected surfaces $f\colon M\to \Q^4_c$.
The following proposition determines the structure of $\mathcal{M}(f)$ for such compact surfaces. 

\begin{proposition}\label{mssc}
Let $f\colon M\to \Q^4_c$ be an oriented surface. If $M$ is homeomorphic to $\mathbb{S}^2$, then $f$ admits at most one Bonnet mate.
\end{proposition}

\begin{proof}
If both Gauss lifts of $f$ are not vertically harmonic, then \cite[Thm. 2]{PV} implies that $f$ admits at most one Bonnet mate. 
Assume that $f$ has a vertically harmonic Gauss lift. We claim that $f$ is superconformal. Indeed,
if $f$ is non-minimal then \cite[Thm. 3]{PV} yields that it is superconformal.
If $f$ is minimal, the claim follows by a well-known result of Calabi \cite{Cal}.  
Then, \cite[Thm. 5(i)]{PV} implies that $f$ admits at most one Bonnet mate. 
\qed
\end{proof}
\medskip

By virtue of the above proposition, in the sequel we focus on non-compact surfaces. The following theorem provides information
about the structure of the moduli space of non-minimal such surfaces.

\begin{theorem} \label{SC}
Let $M$ be a non-compact, simply-connected, oriented 2-dimensional Riemannian manifold, and 
$f\colon M\to \Q^4_c$ a non-minimal surface. Then:
\begin{enumerate}[topsep=0pt,itemsep=-1pt,partopsep=1ex,parsep=0.5ex,leftmargin=*, label=(\roman*), align=left, labelsep=0em]
\item Either there exists at most one Bonnet mate of $f$ in $\mathcal{M}^{\pm}(f)$, or the 
component $\bar{\mathcal{M}}^{\pm}(f)$ is diffeomorphic to $\mathbb{S}^1\simeq \R/2\pi\mathbb{Z}$.
\item We have that $\mathcal{M}^*(f)\neq \emptyset$ if and only if $\mathcal{M}^-(f)\neq \emptyset \neq \mathcal{M}^+(f)$.
If $\mathcal{M}^*(f)\neq \emptyset$, then there is a one-to-one correspondence between Bonnet mates $\tilde{f}\in \mathcal{M}^*(f)$ and pairs $f^-,f^+$ 
with $f^{\pm}\in \mathcal{M}^{\pm}(f)$, such that the distortion differential of the pair $(f,\tilde{f})$ is given by
$$Q=Q_{f,f^-}+Q_{f,f^+},$$ where $Q_{f,f^{\pm}}$ is the distortion differential of the pair $(f,f^{\pm})$.
\item The surface $f$ is proper Bonnet if and only if either $\bar{\mathcal{M}}^{-}(f)=\mathbb{S}^1$, or $\bar{\mathcal{M}}^{+}(f)=\mathbb{S}^1$.
\item The moduli space $\mathcal{M}(f)$ can be parametrized by the product $\bar{\mathcal{M}}^{-}(f)\times \bar{\mathcal{M}}^{+}(f)$.
In particular, if $f$ is proper Bonnet then $\mathcal{M}(f)$ is a smooth manifold.
\end{enumerate}
\end{theorem}

For the proof of the above theorem, we need the following.

\begin{proposition} \label{correspond}
Let $M$ be a simply-connected, oriented 2-dimensional Riemannian manifold with a global complex coordinate $z$, and 
$f\colon M\to \Q^4_c$ a non-minimal surface with $M_0^{\pm}(f)$ isolated.
\begin{enumerate}[topsep=0pt,itemsep=-1pt,partopsep=1ex,parsep=0.5ex,leftmargin=*, label=(\roman*), align=left, labelsep=-0.4em]
\item If $\tilde{f}\in \mathcal{M}^{\pm}(f)$ and $\mathcal{M}^{\pm}(f)\smallsetminus\{\tilde{f}\}\neq\emptyset$, then there exists 
a harmonic $\th^{\pm} \in \mathcal{C}^{\infty}(M;(0,2\pi))$ satisfying \eqref{sys} on $M$, such that
the distortion differential of the pair $(f,\tilde{f})$ is given by \eqref{general1} on $M$.
\item If $h^{\pm}$ can be smoothly extended on $M$, then the distinct modulo $2\pi$ solutions of \eqref{sys} on $M$
determine noncongruent surfaces in $\bar{\mathcal{M}}^{\pm}(f)$.
In particular, any solution $\th^{\pm} \in \mathcal{C}^{\infty}(M;(0,2\pi))$ 
determines a unique 
$\tilde{f}\in \mathcal{M}^{\pm}(f)$ such that 
the distortion differential of the pair $(f,\tilde{f})$ is given by \eqref{general1} on $M$.
\end{enumerate}
\end{proposition}

\begin{proof}
(i) Propositions \ref{dd} and \ref{Npm} imply that there exists $\th^{\pm} \in \mathcal{C}^{\infty}(M\smallsetminus M_0^\pm;(0,2\pi))$ 
satisfying \eqref{sys} on $M\smallsetminus M_0^\pm$, such that the distortion differential $Q$ of the pair $(f,\tilde{f})$ is given by 
\eqref{general1} on $M\smallsetminus M_0^\pm$.
Let $\hat{f}\in \mathcal{M}^{\pm}(f)\smallsetminus\{\tilde{f}\}$. Then, \cite[Lemma 17(ii)]{PV} implies that  
$\hat{f}\in \mathcal{M}^{\pm}(f)\cap \mathcal{M}^{\pm}(\tilde{f})$. From Proposition \ref{Npm}(i) it follows that $\th^\pm$ extends to a harmonic function
$\th^{\pm} \in \mathcal{C}^{\infty}(M;(0,2\pi))$. In particular, from the proof of Proposition \ref{Npm} it follows that $\th^\pm$ satisfies \eqref{sys} on $M$.
From Lemma \ref{pseudo}(i) and Proposition \ref{dd}, it follows that $Q$ and $\Phi^\pm$ vanish precisely on $M_0^\pm$.
Since $\th^{\pm}$ is defined on the whole $M$, it is clear that $Q$ is given by \eqref{general1} on $M$.

(ii) Assume that $h^{\pm}$ extends smoothly on $M$. For a solution 
$\th^{\pm}$ of \eqref{sys} on $M$, consider the quadratic differential 
\bea
\Psi=\Phi^{\mp}+e^{\mp i\th^{\pm}}\Phi^{\pm}.
\eea
By using \eqref{phipm}, it is straightforward to check that $\Psi$ satisfies equations \eqref{Gauss} and \eqref{Ricci} with respect to $\nap,R^{\perp},H$. 
Since $\th^{\pm}$ satisfies \eqref{sys}, by using \eqref{phipm} 
it follows that $\Phi-\Psi$ is holomorphic. Therefore, $\Psi$ satisfies the Codazzi equation.
By the fundamental theorem of submanifolds, there exists a
unique (up to congruence) isometric immersion $\tilde{f}\colon M\to \Q^4_c$ and an orientation-preserving
parallel vector bundle isometry $T\colon N_fM\to N_{\tilde{f}}M$,
such that the Hopf differential $\tilde{\Phi}$ and the mean curvature vector field $\tilde{H}$ of $\tilde{f}$ are given by
$\tilde{\Phi}=T\circ \Psi$ and $\tilde{H}=T H$, respectively.
Clearly, $\tilde{f}$ is congruent to $f$ if and only if $\th^{\pm}\equiv0\mod 2\pi$. 
If $\tilde{f}$ is noncongruent to $f$, then the distortion differential of the pair 
$(f,\tilde{f})$ satisfies $Q^{\mp}\equiv0$ and thus, $\tilde{f}\in \mathcal{M}^{\pm}(f)$.
In particular, if $\th^{\pm} \in \mathcal{C}^{\infty}(M;(0,2\pi))$, then $\tilde{f}\in \mathcal{M}^{\pm}(f)$ and 
from the definition of $\Psi$ it follows that the distortion differential of the pair $(f,\tilde{f})$ is given by \eqref{general1} on $M$.
\qed
\end{proof}
\medskip

\noindent{\emph{Proof of Theorem \ref{SC}:}}
Since $M$ is non-compact, the Uniformization Theorem implies that it is conformally equivalent either to the complex plane, or to the unit disk.
Therefore, $M$ admits a global complex coordinate $z$.

(i) Assume that there exist at least two Bonnet mates of $f$ in $\mathcal{M}^{\pm}(f)$, and let $\tilde{f}\in \mathcal{M}^{\pm}(f)$.
Proposition \ref{dd} implies that $M_0^{\pm}$ is isolated.
Since $\mathcal{M}^{\pm}(f)\smallsetminus\{\tilde{f}\}\neq\emptyset$, from Proposition \ref{correspond}(i) it follows that \eqref{sys}
has a harmonic solution  $\th^{\pm} \in \mathcal{C}^{\infty}(M\smallsetminus M_0^{\pm};(0,2\pi))$.
Then, Lemma \ref{Sys}(iii-ii) implies 
that the space of the distinct modulo $2\pi$ solutions of \eqref{sys} can be smoothly parametrized
by $\mathbb{S}^1$. The proof follows by virtue of Proposition \ref{correspond}(ii).

(ii) Assume that there exists $\tilde{f}\in \mathcal{M}^*(f)$ and consider the quadratic differentials 
$$\Psi_{f^-}=\Phi-Q^-\;\;\; \mbox{and}\;\;\; \Psi_{f^+}=\Phi-Q^+,$$
where $\Phi$ is the Hopf differential of $f$, and $Q$ is the distortion differential of the pair $(f,\tilde{f})$.
We argue that $\Psi_{f^-}$ and $\Psi_{f^+}$ satisfy the compatibility equations with respect to $\nap,R^{\perp},H$.
From Lemma \ref{qke}(i), it follows that $Q^{\pm}$ is holomorphic and thus, the differential $\Psi_{f^{\pm}}$ satisfies the Codazzi equation.
Lemma \ref{pseudo}(i) and Proposition \ref{dd} yield that $\Phi^{\pm}$ and $Q^{\pm}$ vanish precisely on $M_0^{\pm}$.
Therefore, $\Psi_{f^{\pm}}(p)=\Phi(p)$ at any $p\in M_0^{\pm}$ and thus, $\Psi_{f^{\pm}}$ satisfies the algebraic equations
\eqref{Gauss} and \eqref{Ricci} on $M_0^{\pm}$. Moreover, since $\tilde{f}\in \mathcal{M}^*(f)$,
Proposition \ref{dd} implies that there exist $\th^-,\th^+$ with $\th^{\pm} \in \mathcal{C}^{\infty}(M\smallsetminus M_0^{\pm};(0,2\pi))$ such that
$Q^{\pm}$ is given by \eqref{general1} on $M\smallsetminus M_0^{\pm}$. Using \eqref{general1} and \eqref{phipm} it follows that
$\Psi_{f^{\pm}}$ satisfies the equations \eqref{Gauss} and \eqref{Ricci} on $M\smallsetminus M_0^{\pm}$. The fundamental theorem
of submanifolds implies that there exist unique Bonnet mates $f^-,f^+\colon M\to \Q^4_c$ of $f$, such that the Hopf
differential $\Phi_{f^{\pm}}$ of $f^{\pm}$ is given by $\Phi_{f^{\pm}}=T^{\pm}\circ \Psi_{f^{\pm}}$, where 
$T^{\pm}\colon N_fM\to N_{f^{\pm}}M$ is an orientation and mean curvature vector field-preserving, parallel vector bundle isometry.
From Lemma \ref{qke}(i), it follows that the distortion differential of the pair $(f,f^{\pm})$ is $Q^{\pm}$ and thus,
$f^{\pm}\in \mathcal{M}^{\pm}(f)$.

Conversely, assume that there exist $f^-,f^+$ with $f^{\pm}\in \mathcal{M}^{\pm}(f)$ and consider the quadratic differential $\Psi=\Psi^-+\Psi^+$ with
$$\Psi^-=\Phi^--Q_{f,f^-}\;\;\;\mbox{and}\;\;\; \Psi^+=\Phi^+-Q_{f,f^+},$$
where 
$Q_{f,f^{\pm}}$ is the distortion differential of the pair $(f,f^{\pm})$.
Lemma \ref{qke}(i) implies that $Q_{f,f^-}$ and $Q_{f,f^+}$ are both holomorphic and thus, $\Psi$ satisfies the Codazzi equation.
From Lemma \ref{pseudo}(i) and Proposition \ref{dd} it follows that $\Psi^{\pm}$ vanishes precisely on $M_0^{\pm}$.
Furthermore, Proposition \ref{dd} implies that there exist $\th^-$, $\th^+$ with
$\th^{\pm} \in \mathcal{C}^{\infty}(M\smallsetminus M_0^{\pm};(0,2\pi))$ such that
$$Q_{f,f^{\pm}}=(1-e^{\mp i\th^{\pm}})\Phi^{\pm}\;\;\;\mbox{on}\;\;\; M\smallsetminus M_0^{\pm}.$$
Using the above and \eqref{phipm}, it follows that $\Psi$ satisfies \eqref{Gauss} and \eqref{Ricci} on $M\smallsetminus M_0$.
Taking into account that $\Psi^{\pm}(p)=0$ at any $p\in M_0^{\pm}$, from the above and \eqref{phipm} we obtain that
$\Psi$ also satisfies \eqref{Gauss} and \eqref{Ricci} at any point of $M_0$.
The fundamental theorem of submanifolds and Lemma \ref{qke}(i) imply that there exists a unique Bonnet mate $\tilde{f}$ of $f$,
such that the distortion differential of the pair $(f,\tilde{f})$ is $Q=Q_{f,f^-}+Q_{f,f^+}.$
Clearly, $\tilde{f}\in \mathcal{M}^{*}(f)$. The rest of the proof is now obvious.

(iii) Assume that $f$ is proper Bonnet. Then at least one of the disjoint components of $\mathcal{M}(f)$ is infinite.
From part (ii) it follows that at least one of $\mathcal{M}^{-}(f)$ and $\mathcal{M}^{+}(f)$ is infinite. 
If $\mathcal{M}^{\pm}(f)$ is infinite, then part (i) implies that $\bar{\mathcal{M}}^{\pm}(f)=\mathbb{S}^1$.
The converse is obvious.

(iv) From Proposition \ref{dd} and the proof of part (i), it follows that if $\mathcal{M}^{\pm}(f)\neq\emptyset$, then
there exists a one to one correspondence between Bonnet mates of $f$ in $\mathcal{M}^{\pm}(f)$, and 
solutions $\th^{\pm}\in \mathcal{C}^\infty (M\smallsetminus M^{\pm}_0,(0,2\pi))$ of \eqref{sys}.
Using part (ii), we deduce that the moduli space is parametrized by the pairs
$(\th^-,\th^+)$, for those solutions $\th^{\pm}\in \mathcal{C}^\infty (M\smallsetminus M^{\pm}_0,[0,2\pi))$ of \eqref{sys}, 
that correspond to surfaces in $\bar{\mathcal{M}}^{\pm}(f)$.
Obviously, according to this parametrization, $\th^{\mp}\equiv0$ correspond to $\bar{\mathcal{M}}^{\pm}(f)$. 
It is now clear that $\mathcal{M}(f)$ can be parametrized by $\bar{\mathcal{M}}^{-}(f)\times \bar{\mathcal{M}}^{+}(f)$. 
In particular, if $f$ is proper Bonnet then parts (iii) and (i) imply that the moduli space is a smooth manifold.
\qed
\medskip

\begin{remark}\label{Param}
{\emph{
From the proof of Theorem \ref{SC}(i) it follows that if $\bar{\mathcal{M}}^{\pm}(f)$  
is diffeomorphic to $\mathbb{S}^1$, 
then its parametrization is induced by the parametrization of the space of the distinct modulo $2\pi$ solutions of \eqref{sys}.
In the proof of Lemma \ref{Sys}(ii) the parametrization $\th^{\pm}_t, t\in \mathbb{S}^1$, of these 
solutions is such that
\be \label{param}
\th^{\pm}_t(p)=t,\;\; t\in \mathbb{S}^1,
\ee
at some $p\in M$. Obviously, this parametrization depends on $p$ and is not unique, unless the solutions of \eqref{sys}
are constant. In this case, from \eqref{sys} it follows that $h^\pm\equiv0$ on $M$. Then, \eqref{hpm} and Proposition \ref{glphi}
imply that the Gauss lift $G_{\pm}$ of $f$ is vertically harmonic.}}
\end{remark}

\noindent{\emph{Proof of Theorem \ref{MS}:}}
If $M$ is compact, the proof follows from Proposition \ref{mssc}. 
Moreover, if $f$ is minimal then it is known (cf. \cite{DG2,ET2, Vl1}) that either $\mathcal{M}(f)=\{f\}$, or $\mathcal{M}(f)=\mathbb{S}^1$. 
Assume that $M$ is non-compact and $f$ is non-minimal.

(i) If $f$ is not proper Bonnet, then Theorem \ref{SC}(iii) and (i) imply that $f$ admits at most one Bonnet mate in
each one of $\mathcal{M}^{-}(f)$ and $\mathcal{M}^{+}(f)$. If $\mathcal{M}^{-}(f)\neq\emptyset\neq\mathcal{M}^{+}(f)$, 
then Theorem \ref{SC}(ii) yields that $f$ admits exactly three Bonnet mates.

(ii) If $f$ is proper Bonnet, then Theorem \ref{SC}(iii) implies that either $\mathcal{M}^{-}(f)=\mathbb{S}^1$, or $\mathcal{M}^{+}(f)=\mathbb{S}^1$.
Assume that $\bar{\mathcal{M}}^{\pm}(f)=\mathbb{S}^1$.
From Theorem \ref{SC}(i) and (iv) it follows that $f$ is either tight, or flexible, if there exist either at most one, or infinitely many
Bonnet mates of $f$ in $\mathcal{M}^{\mp}(f)$, respectively.
\qed

\subsection{Bonnet Surfaces in $\Q^3_c\subset\Q^4_c$}

We study here Bonnet surfaces lying in totally geodesic hypersurfaces of the ambient space.

\begin{lemma} \label{comp}
Let $f\colon M\to \Q^4_c$ be an oriented surface, which is 
the composition of a non-minimal Bonnet surface $F\colon M\to \Q^3_c$ with a totally geodesic inclusion $j\colon \Q^3_c\to \Q^4_c$.
For every Bonnet mate $\tilde{F}$ of $F$ in $\Q^3_c$ we have that $\tilde{f}=j\circ \tilde{F}\in \mathcal{M}^*(f)$.
\end{lemma}

\begin{proof}
Let $\tilde{F}\colon M\to \Q^3_c$ be a Bonnet mate of $F$.
Denote by $\xi$, $\tilde{\xi}$ the unit normal vector fields of $F$ and $\tilde{F}$ in $\Q^3_c$, respectively, and by $h$ their common mean curvature function.
Then, the mean curvature vector fields of $f$, $\tilde{f}$, are given by $H=hj_*\xi$ and $\tilde{H}=hj_*\tilde{\xi}$, respectively.
The parallel vector bundle isometry $T\colon N_fM\to N_{\tilde{f}}M$ given by $Tj_*\xi=j_*\tilde{\xi}$, $T(\Jp j_*\xi)=\tilde{J}^{\perp}j_*\tilde{\xi}$
preserves the mean curvature vector fields, where $\Jp$, $\tilde{J}^{\perp}$ are the complex structures of $N_fM$ and $N_{\tilde{f}}M$, respectively.
Therefore, $\tilde{f}\in \mathcal{M}(f)$. Since the image of the second fundamental form of $f, \tilde{f}$ is contained in the line bundle spanned by
$j_*\xi, j_*\tilde{\xi}$, respectively, from Lemma \ref{qke}(i) and the definition of $T$ it follows that the zeros of the distortion differential of the pair
$(f, \tilde{f})$ satisfy $Z^-=Z^+=Z$. Hence, $\tilde{f}\in \mathcal{M}^*(f)$.
\qed
\end{proof}
\medskip

\noindent{\emph{Proof of Theorem \ref{Q3}:}}
Let $f=j\circ F$, where $j\colon \Q^3_c\to \Q^4_c$ is a totally geodesic inclusion, and denote by $\xi$ the unit normal of $F$ in $\Q^3_c$.
Since $M$ is simply-connected and $F$ is a Bonnet surface, the result of Lawson-Tribuzy \cite{LT} implies that $M$ is non-compact.
Let $\tilde{F}\colon M\to \Q^3_c$ be a Bonnet mate of $F$. From Lemma \ref{comp} it follows that $j\circ \tilde{F}\in \mathcal{M}^*(f)$
and Theorem \ref{SC}(ii) implies that there exist Bonnet mates $f^-$ and $f^+$ of $f$,
with $f^{\pm}\in \mathcal{M}^{\pm}(f)$. In particular, since any Bonnet mate of $f$ lying in some totally geodesic $\Q^3_c\subset \Q^4_c$ belongs to
$\mathcal{M}^*(f)$, the surface $f^\pm$ does not lie in any totally geodesic hypersurface of $\Q^4_c$.

Assume that $f^{\pm}$ lies in some totally umbilical $\Q^3_{\tilde c}\subset \Q^4_c, \tilde{c}>c$. 
Proposition \ref{dd} implies that $M_1$ is isolated. Let $(U,z)$ be a complex chart with $U\cap M_1=\emptyset$. 
Then, there exist $\varphi, \varphi^{\pm}\in \mathcal{C}^{\infty}(U)$ such that the Hopf differentials $\Phi$, $\Phi_{f^{\pm}}$ of $f$ and $f^{\pm}$, respectively, 
are given on $U$ by
\be \label{phitphi}
\Phi=\frac{\lambda^2}{2}e^{i\varphi}\sqrt{\|H\|^2-K}e_3 dz^2\;\;\; \mbox{and}\;\;\; \Phi_{f^{\pm}}=\frac{\lambda^2}{2}e^{i\varphi^{\pm}}\sqrt{\|H\|^2-K}
\tilde{\varepsilon}_3^{\pm} dz^2,
\ee
where $\lambda>0$ is the conformal factor, $e_3=j_*\xi$, and $\tilde{\varepsilon}_3^{\pm}\in N_{f^{\pm}}M$ is a smooth unit vector field, parallel to the
line segment that the ellipse of curvature of $f^{\pm}$ degenerates. 
Consider an orientation and mean curvature vector field-preserving,
parallel vector bundle isometry $T_{\pm}\colon N_fM \to N_{f^{\pm}}M$.
Appealing to Lemma \ref{qke}(i) and using \eqref{phitphi}, it follows that the distortion differential 
$Q_{f,f^{\pm}}$ of the pair $(f,f^{\pm})$ is given on $U$ by
\be \label{qffpm}
Q_{f,f^{\pm}}\equiv Q^{\pm}_{f,f^{\pm}}=\frac{\lambda^2}{4}\sqrt{\|H\|^2-K}
\left(e^{i\varphi}(e_3\pm ie_4)-e^{i\varphi^{\pm}}(\varepsilon_3^{\pm}\pm i\varepsilon_4^{\pm})\right)dz^2,
\ee
where $e_4=\Jp e_3$, $\varepsilon_3^{\pm}=T_{\pm}^{-1}\tilde{\varepsilon}_3^{\pm}$ and $\varepsilon_4^{\pm}=\Jp\varepsilon_3^{\pm}$. 
On the other hand, according to Proposition \ref{dd} there exists
$\th^{\pm}\in \mathcal{C}^{\infty}(U;(0,2\pi))$ such that $Q_{f,f^{\pm}}$ is given by \eqref{general1} on $U$. Substituting $\Phi^{\pm}$ from
\eqref{phitphi} into \eqref{general1}, and using \eqref{qffpm} we obtain that
$$\varepsilon_3^{\pm}\pm i\varepsilon_4^{\pm}=e^{i(\varphi-\varphi^{\pm}\mp \th^{\pm})}(e_3\pm ie_4)\;\;\; \mbox{on}\;\;\; U.$$
Moreover, since $Q^{\mp}_{f,f^{\pm}}\equiv 0$, from Lemma \ref{qke}(i) and \eqref{phitphi} it follows that
$$\varepsilon_3^{\pm}\mp i\varepsilon_4^{\pm}=e^{i(\varphi-\varphi^{\pm})}(e_3\mp ie_4)\;\;\; \mbox{on}\;\;\; U.$$
From the last two equations we obtain that $\th^{\pm}=\pm 2(\varphi-\varphi^{\pm})\mod 2\pi$. Then, the above implies that
\bea 
\w_{34}^{\pm}=\frac{1}{2}d\th^{\pm}+\w_{34},
\eea
where $\w_{34}$ and $\w_{34}^{\pm}$ are the connection forms corresponding to the dual frame fields of 
$\{e_3,e_4\}$ and $\{\varepsilon_3^{\pm},\varepsilon_4^{\pm}\}$, respectively. 
Since $f$ and $f^{\pm}$ lie in totally umbilical hypersurfaces and $T_{\pm}$ is parallel, 
it follows that the vector fields $e_3$ and $\varepsilon_3^{\pm}$ are parallel in the normal connection of $f$. 
Therefore, the last relation yields that $\th^{\pm}$ is constant on $U$. 
Proposition \ref{Npm} implies that $\th^{\pm}$ satisfies \eqref{sys} on $U$. 
From \eqref{sys} it follows that $h^{\pm}\equiv0$ on $U$. 
Then, \eqref{hpm} and Proposition \ref{glphi} yield that the section $H^{\pm}$ is anti-holomorphic on $U$. 
Since $H=he_3$, where $h$ is the mean curvature function of $F$, this implies that $h$ is constant on $U$. 
Since $U$ is arbitrary and $M_1$ is isolated, it follows that 
$h$ is constant on $M$.

Conversely, if $F$ has constant mean curvature function, then $f$ and its Bonnet mates have nonvanishing parallel mean curvature vector field. 
From \cite{Chen, Yau} it follows that $f^\pm$ lies in some totally umbilical hypersurface of $\Q^4_c$.

Moreover, from Theorem \ref{SC}(i) and (iv) it is clear that either $f$ admits exactly three Bonnet mates, or it is flexible proper Bonnet. 
\qed

\subsection{Proper Bonnet Surfaces}

We study here non-minimal proper Bonnet surfaces $f\colon M\to \Q^4_c$. 
From Proposition \ref{mssc} it follows that if $f\colon M\to \Q^4_c$ is a simply-connected proper Bonnet surface, 
then $M$ is non-compact and therefore it admits a global complex coordinate $z$. 
By virtue of Theorem \ref{SC}(iii-iv), we focus on surfaces with $\bar{\mathcal{M}}^{\pm}(f)=\mathbb{S}^1$.
For such a surface, Proposition \ref{dd} implies that $M^{\pm}_0(f)$ consists of isolated points only.

We need some facts about absolute value type functions (cf. \cite{EGT} or \cite{ET}).
Let $M$ be a 2-dimensional oriented Riemannian manifold. A function $u\in\mathcal{C}^\infty (M;[0,+\infty))$ is called
of {\emph{absolute value type}} if for all $p\in M$ and any complex coordinate $z$ around $p$, there exists a nonnegative 
integer $m$ and a smooth positive function $u_0$ on a neighbourhood $U$ of $p$, such that 
$$u=|z-z(p)|^mu_0\;\;\; \mbox{on}\;\;\; U.$$
If $m>0$, then $p$ is called {\emph{a zero of $u$ of multiplicity $m$}}.
It is clear that if an absolute value type function $u$ does not vanish identically, 
then its zeros are isolated and they have well-defined multiplicities. 
Furthermore, the Laplacian $\Delta\log u$ is still defined and smooth at the zeros of $u$.

\begin{proposition} \label{IPU}
Let $f\colon M\to \Q^4_c$ be a simply-connected oriented surface with $\bar{\mathcal{M}}^{\pm}(f)=\mathbb{S}^1$.
Consider a complex chart $(U,z)$ on $M$, with $U\cap M^{\pm}_0(f)=\{p\}$ and $z(p)=0$. Then:
\begin{enumerate}[topsep=0pt,itemsep=-1pt,partopsep=1ex,parsep=0.5ex,leftmargin=*, label=(\roman*), align=left, labelsep=-0.4em]
\item There exists a positive integer $m$, such that differential $\Phi^{\pm}$ satisfies
\bea
\Phi^{\pm}=z^m\hat{\Phi}^{\pm}\;\;\; \mbox{on}\;\;\; U,\;\;\; \hat{\Phi}^{\pm}(p)\neq0.
\eea
\item The function $\|\mathcal{H}^{\pm}\|$ is of absolute value type on $M$. 
The multiplicity of its zero $p\in M_0^{\pm}(f)$ is the integer $m$.
\end{enumerate}
\end{proposition}

\begin{proof}
(i) Let $\tilde{f}\in \mathcal{M}^{\pm}(f)$.
From Proposition \ref{correspond}(i) it follows that there exists $\th^{\pm}\in \mathcal{C}^{\infty}(U;(0,2\pi))$
such that the distortion differential of the pair $(f,\tilde{f})$ is given by
$$Q=(1-e^{\mp i\th^{\pm}})\Phi^{\pm}\;\;\; \mbox{on}\;\;\; U.$$
Proposition \ref{dd} implies that $p$ is the only zero of $Q$ in $U$. Lemmas \ref{qke}(i) and \ref{zeros} yield that
there exists a positive integer $m$ such that
$$Q=z^m \hat{\Psi}^{\pm}\;\;\; \mbox{on}\;\;\; U,\;\;\; \hat{\Psi}^{\pm}(p)\neq0.$$
The proof follows from the above expressions of $Q$,
by setting $\hat{\Phi}^{\pm}=(1-e^{\mp i\th^{\pm}})^{-1}\hat{\Psi}^{\pm}$.

(ii) Let $z=x+iy$ and set $e_1=\d_x/\lambda, e_2=\d_y/\lambda$, where $\lambda>0$ is the conformal factor.
If $\hat{\Phi}^{\pm}=\hat{\phi}^{\pm}dz^2$ on $U$, then part (i) implies that $\phi^{\pm}=z^m\hat{\phi}^{\pm}$, where
$\phi^{\pm}$ is given by \eqref{phipm} on $U$.
Consequently, from \eqref{fpm} it follows that
$$\|\mathcal{H}^{\pm}\|=|z|^m u,\;\; \mbox{where} \;\;u=\sqrt{2}\lambda^{-2}\|\hat{\phi}^{\pm}\|\;\; \mbox{is smooth and positive}.$$
Clearly, the multiplicity of $p$ is $m$.
\qed
\end{proof}
\medskip

\begin{lemma}\label{relations}
Let $M$ be an oriented 2-dimensional Riemannian manifold with a global complex coordinate $z$, and 
$f\colon M\to \Q^4_c$ a surface with $M_0^{\pm}(f)=\emptyset$. The 1-forms $a_1^{\pm}, a_2^{\pm}$ on $M$ given by 
$$a_1^{\pm}=d\log \|\mathcal{H}^{\pm}\|-\star\Omega^{\pm},\;\;\; a_2^{\pm}=\star a_1^{\pm},$$
vanish precisely at the points where the Gauss lift $G_{\pm}$ of $f$ is vertically harmonic.
Moreover:
\begin{enumerate}[topsep=0pt,itemsep=-1pt,partopsep=1ex,parsep=0.5ex,leftmargin=*, label=(\roman*), align=left, labelsep=0em]
\item $$ da_2^{\pm}=\left(\Delta\log\|\mathcal{H}^{\pm}\|-2K \mp K_N\right)dM=\frac{4}{\lambda^2}\Real h^{\pm}_zdM,$$
\item $$ a_1^{\pm}\wedge a_2^{\pm}=\frac{\|\tau^{v}(G_{\pm})\|^2}{4\|\mathcal{H}^{\pm}\|^2}dM=\frac{4}{\lambda^2}|h^{\pm}|^2dM,$$
\end{enumerate}
where 
$\lambda>0$ is the conformal factor, and $h^{\pm}$ is given by \eqref{hpm} on $M$.
\end{lemma}

\begin{proof}
Let $z=x+iy$ and set $e_1=\d_x/\lambda$, $e_2=\d_y/\lambda$. 
Consider the frame field $\{e^{\pm}_3, e^{\pm}_4\}$ of $N_fM$ determined by $\{e_1,e_2\}$ from \eqref{HI}.
Then \eqref{fpm} and \eqref{lambda} hold on $M$, and as in the proof of Lemma \ref{qiz}, we obtain \eqref{phiH} and \eqref{hpmB}.
Using \eqref{lambda}, from \eqref{hpmB} it follows that
\be \label{Chern}
a_1^{\pm}=\frac{2}{\lambda}\left(\Real h^{\pm}\w_1+\Imag h^{\pm}\w_2 \right),
\ee
where $\{\w_1,\w_2\}$ is the dual frame field of $\{e_1,e_2\}$.
Proposition \ref{glphi} and \eqref{hpm} imply that $h^{\pm}(p)=0$ if and only if the Gauss lift $G_{\pm}$ of $f$ is vertically harmonic at $p$.
Therefore, from \eqref{Chern} it follows that $a_1^{\pm}$ vanishes precisely at the points where $G_{\pm}$ is vertically harmonic.

(i) Differentiating $\w_{34}^{\pm}=\w_{34}^{\pm}(e_1)\w_1+\w_{34}^{\pm}(e_2)\w_2$ and using \eqref{normcf} and 
that $\w_{12}=\star d\log\lambda$, we obtain 
$$K_N=\mp\left(e_1(\log\lambda)\w_{34}^{\pm}(e_2)-e_2(\log\lambda)\w_{34}^{\pm}(e_1)+e_1(\w_{34}^{\pm}(e_2))-e_2(\w_{34}^{\pm}(e_1)) \right).$$
Differentiating \eqref{hpmB} with respect to $z$, taking the real part, and using the above and that $\Delta\log \lambda=-K$, yields
$(4/\lambda^2)\Real h^{\pm}_z=\Delta\log\|\mathcal{H}^{\pm}\|-2K \mp K_N$. On the other hand, taking into account \eqref{dW}, 
exterior differentiation of $a_2^{\pm}$ gives 
$da_2^{\pm}=\left(\Delta\log\|\mathcal{H}^{\pm}\|-2K \mp K_N\right)dM$,
and this completes the proof.

(ii) Assume that the mean curvature vector field of $f$ is given by $H=H^{3\pm}e^{\pm}_3+H^{4\pm}e^{\pm}_4$. Then,
\bea
H^{\pm}=\frac{1}{2}(H\pm i\Jp H)=\frac{1}{2}(H^{3\pm}\mp iH^{4\pm})(e^{\pm}_3\pm ie^{\pm}_4).
\eea
Differentiating the above with respect to $\d$ in the normal connection, we obtain from \eqref{Codazzi} that
\bea
\nap_{\dzb}\phi^{\pm}=\frac{\lambda^2}{4}\left(\d (H^{3\pm}\mp iH^{4\pm})\mp i\w_{34}^{\pm}(\d)(H^{3\pm}\mp iH^{4\pm})\right)(e^{\pm}_3\pm ie^{\pm}_4).
\eea
From \eqref{phiH} and the above it follows that
$$h^{\pm}=\frac{\lambda}{2}\left(\frac{H^{3\pm}_1 \mp H^{4\pm}_2}{\|\mathcal{H}^{\pm}\|}-i\frac{H^{3\pm}_2\pm H^{4\pm}_1}{\|\mathcal{H}^{\pm}\|}\right),$$
where $H^{a\pm}_j$, $j=1,2,\ a=3,4$, is given by \eqref{Hij}.
Then, \eqref{Chern} implies that
\be\label{uvpm}
a_1^{\pm}=u^{\pm}\w_1-v^{\pm}\w_2,\;\;\;\; \mbox{where}\;\;\;\;
u^{\pm}= \frac{H^{3\pm}_1 \mp H^{4\pm}_2}{\|\mathcal{H}^{\pm}\|},\;\; v^{\pm}= \frac{H^{3\pm}_2\pm H^{4\pm}_1}{\|\mathcal{H}^{\pm}\|}.
\ee
From the above two relations 
it follows that
$$a_1^{\pm}\wedge a_2^{\pm}=\left((u^{\pm})^2+(v^{\pm})^2\right)dM =\frac{4}{\lambda^2}|h^{\pm}|^2dM,$$
where $dM=\w_1\wedge\w_2$. On the other hand, from Proposition \ref{tension field} and \eqref{uvpm}, we obtain that
$$\|\tau^{v}(G_{\pm})\|^2=4\|\mathcal{H}^{\pm}\|^2\left((u^{\pm})^2+(v^{\pm})^2\right),$$
and the proof follows from the last two equations.
\qed
\end{proof}

\begin{theorem} \label{local deformability}
Let $f\colon M\to \Q^4_c$ be a simply-connected oriented surface. If $\bar{\mathcal{M}}^{\pm}(f)=\mathbb{S}^1$, then:
\begin{enumerate}[topsep=0pt,itemsep=-1pt,partopsep=1ex,parsep=0.5ex,leftmargin=*, label=(\roman*), align=left, labelsep=0em]
\item The Gauss lift $G_{\pm}$ of $f$ is vertically harmonic at any point of $M_0^{\pm}(f)$.
\item The surface $f$ is $\pm$ isotropically isothermic on $M\smallsetminus M_0^{\pm}(f)$, 
and the following differential equation is valid on the whole $M$
\be \label{fund}
\Delta\log\|\mathcal{H}^{\pm}\|-2K \mp K_N=\frac{\|\tau^{v}(G_{\pm})\|^2}{4\|\mathcal{H}^{\pm}\|^2}.
\ee
\item The forms $a^{\pm}_1, a^{\pm}_2$ of Lemma \ref{relations} satisfy on $M\smallsetminus M_0^{\pm}(f)$ the relations
\begin{equation}\label{C1}
da^{\pm}_1 = 0\;\;\;\;\;\mbox{and}\;\;\;\;\; da^{\pm}_2 = a^{\pm}_1\wedge a^{\pm}_2.
\end{equation}
\end{enumerate}
Conversely, if $f$ is non-minimal with $M_0^{\pm}(f)=\emptyset$, and (ii) or (iii) holds, then $\bar{\mathcal{M}}^{\pm}(f)=\mathbb{S}^1$.
\end{theorem}

\begin{proof}
Let $\tilde{f}\in \mathcal{M}^{\pm}(f)$. Proposition \ref{dd} yields that $M_0^{\pm}$ is isolated.
From Proposition \ref{correspond}(i) it follows that there exists a harmonic function 
$\th^{\pm} \in \mathcal{C}^{\infty}(M;(0,2\pi))$ satisfying \eqref{sys} on $M$. Lemma \ref{Sys}(iii) implies that $h^{\pm}$
extends smoothly on $M$ and $A^{\pm}\equiv 0$.
Then, from \eqref{Ah} it follows that 
\be \label{reimh}
\Imag h^{\pm}_z\equiv0 \;\;\; \mbox{and}\;\;\; |h^{\pm}|^2\equiv \Real h^{\pm}_z \;\;\; \mbox{on}\;\;\; M.
\ee

(i) Since $h^{\pm}$ extends smoothly on $M$, equation \eqref{hpm} holds on $M$. From Lemma \ref{pseudo}(i) and \eqref{hpm} we obtain
that $$\nap_{\dzb}\phi^{\pm}(p)=0\;\;\; \mbox{for any}\;\;\; p\in M_0^{\pm}(f).$$ Appealing to Proposition \ref{glphi}, this is equivalent with the vertical
harmonicity of $G_{\pm}$ at $p$.

(ii) By virtue of Lemma \ref{qiz}, the first equation in \eqref{reimh} implies that $f$ is $\pm$ isotropically isothermic on $M\smallsetminus M_0^{\pm}$.
Using Lemma \ref{relations}, the second equation in \eqref{reimh} yields that \eqref{fund} holds on $M\smallsetminus M_0^{\pm}$.
From Proposition \ref{IPU}(ii) it follows that the left-hand side of \eqref{fund} can be smoothly extended on $M$.
Therefore, \eqref{fund} is valid on the whole $M$.

(iii) From the definition of $a^{\pm}_1$, it follows that the first equation in \eqref{C1} is equivalent 
with the fact that $f$ is $\pm$ isotropically isothermic on $M\smallsetminus M_0^{\pm}$.
Moreover, Lemma \ref{relations} implies that the second equation in \eqref{C1} is equivalent with the second equation in \eqref{reimh} on $M\smallsetminus M_0^{\pm}$.

Conversely, assume that $f$ is non-minimal and $M_0^{\pm}(f)=\emptyset$. 
As above, we obtain that each one of (ii) and (iii) is equivalent to \eqref{reimh}. 
From \eqref{reimh} and \eqref{Ah} it follows that $A^{\pm}\equiv 0$ on $M$, and Lemma \ref{Sys}(ii) implies that 
the space of the distinct modulo $2\pi$ solutions of \eqref{sys} on $M$ is parametrized by $\mathbb{S}^1$. 
From Proposition \ref{correspond}(ii) it follows that $\bar{\mathcal{M}}^{\pm}(f)=\mathbb{S}^1$.
\qed
\end{proof}
\medskip

\begin{corollary}\label{CM}
Let $f\colon M\to \Q^4_c$ be a simply-connected surface. 
If $\bar{\mathcal{M}}^{\pm}(f)=\mathbb{S}^1$ and $\tau^{v}(G_{\pm})\neq 0$ everywhere, then the conformal metric
\bea
d\hat{s}^2=\frac{\|\tau^{v}(G_{\pm})\|^2}{4\|\mathcal{H}^{\pm}\|^2} ds^2
\eea
has Gaussian curvature $\hat{K}=-1$.
\end{corollary}

\begin{proof}
By virtue of Theorem \ref{local deformability}(i), it follows that $M_0^\pm(f)=\emptyset$.
Consider the forms $a_1^{\pm}, a_2^{\pm}$ of Lemma \ref{relations}.
Proposition \ref{tension field} and \eqref{uvpm} yield that 
$$d\hat{s}^2=a_1^{\pm}\otimes a_1^{\pm}+ a_2^{\pm}\otimes a_2^{\pm}\;\;\; \mbox{on}\;\;\; M.$$ 
Let $a_{12}^{\pm}$ be the connection form corresponding to the coframe $\{a_1^{\pm},a_2^{\pm}\}$. 
Then, $$da_2^{\pm}=a_1^{\pm}\wedge a_{12}^{\pm}\;\;\; \mbox{and} \;\;\; da_{12}^{\pm}=-\hat{K}a_1^{\pm}\wedge a_{2}^{\pm}.$$ 
Since $\bar{\mathcal{M}}^{\pm}(f)=\mathbb{S}^1$, the first equation of the above and the second relation in \eqref{C1} yield that 
$a_{12}^{\pm}=a_{2}^{\pm}$.
Using the second equation of the above, this implies that
$da_{2}^{\pm}=-\hat{K}a_1^{\pm}\wedge a_{2}^{\pm}$, and the proof follows by virtue of the second relation in \eqref{C1}.
\qed
\end{proof}

\begin{remark}\label{Rem}
{\emph{
\begin{enumerate}[topsep=0pt,itemsep=-1pt,partopsep=1ex,parsep=0.5ex,leftmargin=*, label=(\roman*), align=left, labelsep=-0.2em]
\item Theorems \ref{SC}(i) and \ref{local deformability}(i) imply that a surface $f$ admits at most one Bonnet mate 
in $\mathcal{M}^{\pm}(f)$, if there exists a point $p\in M_0^{\pm}(f)$ at which the Gauss lift
$G_{\pm}$ of $f$ is not vertically harmonic. From Theorem \ref{SC}(iv), it follows that $f$ is not proper
Bonnet if there exists an umbilic point at which $H$ is non-parallel.
This extends a result of Roussos-Hernandez \cite[Thm. 1B]{RH}.
\item Equation \eqref{fund} extends the Ricci-like condition satisfied by the non-superconformal
surfaces whose Gauss lift $G_{\pm}$ is vertically harmonic (cf. \cite[Prop. 23(iii)]{PV}).
\item For umbilic-free surfaces in $\R^3$, the analogues of \eqref{fund} and \eqref{C1}, are
due to Colares-Kenmotsu \cite{CK} and Chern \cite{Ch2}, respectively.
\end{enumerate}
}}\end{remark}

From Theorems \ref{SC}(iv) and \ref{local deformability}(ii) it follows that a flexible proper Bonnet surface is strongly isotropically
isothermic away from its isolated pseudo-umbilic points.
The following proposition shows that a Bonnet, strongly isotropically isothermic surface is proper Bonnet.
The analogous result for isothermic Bonnet surfaces in $\Q^3_c$ is due to Graustein \cite{Grau}.

\begin{proposition} \label{IP}
Let $f\colon M\to \Q^4_c$ be a non-minimal, simply-connected oriented surface. 
If $f$ is $\pm$ isotropically isothermic, then either $\bar{\mathcal{M}}^{\pm}(f)=\{f\}$, or
$\bar{\mathcal{M}}^{\pm}(f)=\mathbb{S}^1$. In particular, if $f$ is Bonnet and strongly isotropically isothermic then either 
$\mathcal{M}(f)=\mathbb{S}^1$, or $\mathcal{M}(f)=\mathbb{S}^1\times\mathbb{S}^1$.
\end{proposition}

\begin{proof}
Assume that there exists a Bonnet mate $\tilde{f}\in \mathcal{M}^{\pm}(f)$. From Proposition \ref{dd} it follows that there 
exists $\th^{\pm} \in \mathcal{C}^{\infty}(M;(0,2\pi))$, such that the distortion differential 
of the pair $(f,\tilde{f})$ is given by \eqref{general1} on $M$. Proposition \ref{Npm}(ii) yields that $\th^{\pm}$ is harmonic.

If $M$ is compact, then the maximum principle implies that $\th^{\pm}$ is constant. From Lemma \ref{qke}(i) and \eqref{general1}
it follows that $\Phi^\pm$ is holomorphic, and Proposition \ref{glphi} implies that the Gauss lift $G_{\pm}$ of $f$ is vertically harmonic.
From \cite[Thm. 3]{PV} it follows that $f$ is superconformal. Then, \cite[Prop. 10]{PV} yields that $\Phi^\pm\equiv0$.
Lemma \ref{pseudo}(i) implies that $M_0^\pm(f)=M$ and this contradicts the fact that $f$ is $\pm$ isotropically isothermic.
Therefore, if $M$ is homeomorphic to $\mathbb{S}^2$, then $\bar{\mathcal{M}}^{\pm}(f)=\{f\}$.

Assume that $M$ is non-compact and let $z$ be a global complex coordinate on $M$.
Proposition \ref{Npm} yields that $\th^{\pm}$ satisfies \eqref{sys} on $M$.
Then, Lemma \ref{Sys}(iii-ii) implies that the space of the distinct modulo $2\pi$ solutions of \eqref{sys}
is parametrized by $\mathbb{S}^1$. From Proposition \ref{correspond}(ii) we obtain that $\bar{\mathcal{M}}^{\pm}(f)=\mathbb{S}^1$.
In particular, if $f$ is Bonnet and strongly isotropically isothermic, the proof follows by using Theorem \ref{SC}(iv).
\qed
\end{proof}
\medskip

\begin{example} \label{Has}
Tight proper Bonnet surfaces in $\R^4$, that are strongly isotropically isothermic and isothermic, and they have a vertically harmonic Gauss lift.

{\emph{We consider the product in $\R^4$ of two plane curves $\gamma_1, \gamma_2$, as in Example \ref{ISTNII},
and we adopt the notation used there. Assume that the curvature of $\gamma_j$ is $k_j(s_j)=cs_j$, $j=1,2$, with $0\neq c\in \R$, and 
we restrict the product surface $f$ such that $f\colon M\to \R^4$ is simply-connected and umbilic-free. Clearly, $f$ has flat normal bundle
and does not lie in any totally umbilical hypersurface of $\R^4$. Moreover, from \eqref{curves} it follows that $f$ is 
strongly isotropically isothermic.}}

{\emph{It has been proved by Hasegawa \cite[Example 1]{Ha} that the Gauss lift $G_-$ of $f$ is vertically harmonic.
Since $f$ is neither minimal, nor superconformal, from \cite[Thm. 4, Prop. 25]{PV} it follows that $\bar{\mathcal{M}}^{-}(f)=\mathbb{S}^1$.
}}

{\emph{Since $f$ is $+$ isotropically isothermic, Proposition \ref{IP} implies that either $\bar{\mathcal{M}}^{+}(f)=\{f\}$, 
or $\bar{\mathcal{M}}^{+}(f)=\mathbb{S}^1$.
We claim that $\bar{\mathcal{M}}^{+}(f)=\{f\}$. Arguing indirectly, assume that $\bar{\mathcal{M}}^{+}(f)=\mathbb{S}^1$.
Then, from Theorem \ref{local deformability}(ii) it follows that
$$\Delta\log\|\mathcal{H}^{+}\|-2K=\frac{\|\tau^{v}(G_{+})\|^2}{4\|\mathcal{H}^{+}\|^2}.$$
On the other hand, since $\bar{\mathcal{M}}^{-}(f)=\mathbb{S}^1$, Theorem \ref{local deformability}(ii) yields that $$\Delta\log\|\mathcal{H}^-\|-2K=0.$$
Since $K_N=0$ everywhere on $M$, it follows that $\|\mathcal{H}^-\|=\|\mathcal{H}^+\|$ and the above two relations imply
that the Gauss lift $G_+$ of $f$ is vertically harmonic. Therefore, the mean curvature vector field of $f$ is parallel in the normal connection and thus 
(cf. \cite{Chen, Yau}), $f$ lies in some totally umbilical hypersurface of $\R^4$. 
This is a contradiction and the claim follows. From Theorem \ref{SC}(iv) we deduce that
$\mathcal{M}(f)=\mathbb{S}^1$.}}
\end{example}

\section{Compact Surfaces} \label{s7}

\subsection{The Effect of Isotropic Isothermicity}

We study here the effect of isotropic isothermicity on the structure of the moduli space $\mathcal{M}(f)$ for compact surfaces.

\begin{theorem} \label{QICC}
Let $f\colon M\to \Q^4_c$ be a compact oriented surface, and $V$ an open and dense subset of $M$. 
If one of the following holds, then $\mathcal{N}^{\pm}(f)=\emptyset$.
\begin{enumerate}[topsep=0pt,itemsep=-1pt,partopsep=1ex,parsep=0.5ex,leftmargin=*, label=(\roman*), align=left, labelsep=-0.2em]
\item The Gauss lift $G_{\pm}$ of $f$ is not vertically harmonic and $f$ is $\pm$ isotropically isothermic on $V$.
\item The set $V$ is connected and  $f$ is totally non $\pm$ isotropically isothermic on $V$.
\end{enumerate}
\end{theorem}

\begin{proof}
If (i) or (ii) holds, then from Proposition \ref{glphi} or Examples \ref{examples}(ii-iii), respectively, it follows that $f$ is non-minimal.
Arguing indirectly, assume that there exists $\tilde{f}\in\mathcal{N}^{\pm}(f)$. Proposition \ref{dd} implies that
$M_0^{\pm}$ is isolated and that there exists
$\th^{\pm} \in \mathcal{C}^{\infty}(M\smallsetminus M_0^{\pm};(0,2\pi))$, such that the distortion differential $Q$
of the pair $(f,\tilde{f})$ satisfies \eqref{general1} on $M\smallsetminus M_0^{\pm}$.

(i) Since $V$ is dense, it follows that $f$ is $\pm$ isotropically isothermic on $M\smallsetminus M_0^{\pm}$. 
Then, Proposition \ref{Npm}(ii) implies that $\th^\pm$ extends to a bounded harmonic function on $M$, which has to be constant by the maximum principle.
By virtue of Lemma \ref{qke}(i), from \eqref{general1} it follows that $\Phi^{\pm}$ is holomorphic. 
Proposition \ref{glphi} yields that the Gauss lift $G_{\pm}$ of $f$ is vertically harmonic, and this is a contradiction.

(ii) From the definition of non $\pm$ isotropically isothermic points it follows that $M_0^{\pm}\subset M\smallsetminus V$. 
Therefore, $\th^{\pm}$ is defined everywhere on $V$. Let $(U,z)$ be a complex chart with $U\subset V$.
Proposition \ref{Npm} implies that $\th^{\pm}$ satisfies \eqref{sys} on $U$.
From Lemma \ref{qiz} it follows that $\Imag h^{\pm}_z\neq0$ everywhere on $U$. 
Appealing to Lemma \ref{Sys}(i), \eqref{Ah} and \eqref{harmonic}
yield that $\Delta\th^{\pm}$ is nowhere-vanishing on $U$. 
Since $U$ is an arbitrary subset of the connected $V$,
we deduce that either $\Delta\th^{\pm}>0$, or $\Delta\th^{\pm}<0$, on $V$. 
Since $V$ is dense in $M\smallsetminus M_0^{\pm}$, it follows by continuity that either $\Delta\th^{\pm}\geq0$,
or $\Delta\th^{\pm}\leq0$, on $M\smallsetminus M_0^{\pm}$. As in the proof of \cite[Thm. 2]{JMN}, it can be shown that either $\th^{\pm}$, or $-\th^{\pm}$
can be extended to a subharmonic function on $M$ which attains a maximum and thus, it has to be constant by the maximum principle for subharmonic functions.
As in the proof of part (i), it follows that the Gauss lift $G_{\pm}$ of $f$ is vertically harmonic. 
Then, Example \ref{examples}(i) implies that $f$ is $\pm$ isotropically isothermic on $V$, which is a contradiction.
\qed
\end{proof}
\medskip

\noindent{\emph{Proof of Theorem \ref{QIC}:}}
Since $G_{\pm}$ is not vertically harmonic and $f$ is either $\pm$ isotropically isothermic, 
or totally non $\pm$ isotropically isothermic, 
on $V$,
Theorem \ref{QICC} implies that
$\mathcal{N}^{\pm}(f)=\emptyset$. On the other hand, since $G_{\mp}$ is not vertically harmonic, from \cite[Thm. 13(i)]{PV}
it follows that there exists at most one Bonnet mate of $f$ in $\mathcal{M}^{\mp}(f)$. 
Therefore, $f$ admits at most one Bonnet mate.
In particular, if $f$ is either strongly isotropically isothermic, or strongly totally non isotropically isothermic, on $V$, 
then Theorem \ref{QICC} implies that $\mathcal{N}^{-}(f)=\mathcal{N}^{+}(f)=\emptyset$ and thus, $f$ does not admit any Bonnet mate.
\qed
\medskip

The following consequence of Theorem \ref{QIC} shows that the result of \cite{JMN} can be strengthened.

\begin{corollary}
Let $F\colon M\to \Q^3_c$ be a compact oriented surface and $j\colon \Q^3_c\to\Q^4_c$ a totally geodesic inclusion.
If the mean curvature of $F$ is not constant and $F$ is either isothermic, or totally non isothermic,
on an open dense and connected subset $V$ of $M$, then $f=j\circ F$ does not admit any Bonnet mate in $\Q^4_c$.
\end{corollary}

\begin{proof}
From Proposition \ref{MCF-PCF} it follows that $f$ is strongly (totally non) isotropically isothermic on $V$ if and 
only if $F$ is (totally non) isothermic on $V$. The proof follows immediately from Theorem \ref{QIC}.
\qed
\end{proof}

\subsection{Locally Proper Bonnet Surfaces}

An oriented surface $f\colon M\to \Q^4_c$ is called \emph{locally proper Bonnet}, if every point of $M$ has a simply-connected neighbourhood $U$
such that $f|_U$ is proper Bonnet. Notice that if $f|_U$ is non-minimal, then Theorem \ref{SC}(iv) implies that $\mathcal{M}(f|_{U})$ is 
a smooth manifold.

\begin{proposition} \label{non minimal}
Let $f\colon M\to \Q^4_c$ be a locally proper Bonnet surface. Then:
\begin{enumerate}[topsep=0pt,itemsep=-1pt,partopsep=1ex,parsep=0.5ex,leftmargin=*, label=(\roman*), align=left, labelsep=-0.2em]
\item Either $f$ is minimal, or $\inter\{p\in M:H(p)=0\}=\emptyset$.
\item If $f$ is non-minimal, then for every $p\in M$ there exists a submanifold $L^n(p), 1\leq n\leq2$, of the torus $\mathbb{S}^1\times \mathbb{S}^1$,
$\mathbb{S}^1\simeq \R/2\pi\mathbb{Z}$,
with the property that $L^n(p)$ is also a submanifold of 
$\mathcal{M}(f|_{U})$ for every sufficiently small simply-connected neighbourhood $U$ of $p$.
In particular, for every point of $M$, a submanifold of the torus with this property is either $\mathbb{S}^1_{-}=\mathbb{S}^1\times\{0\}$,
or $\mathbb{S}^1_{+}=\{0\}\times \mathbb{S}^1$. 
\end{enumerate}
\end{proposition}

\begin{proof}
(i) Arguing indirectly, assume that $f$ is 
non-minimal and $\inter\{p\in M:H(p)=0\}\neq\emptyset$. 
Then, for 
a boundary point $\bar p$ of $\{p\in M:H(p)=0\}$, there exists a simply-connected
complex chart $(U,z)$ around $\bar p$ such that $f|_U$ is proper Bonnet and non-minimal. 
By virtue of Theorem \ref{SC}(iii), we may assume that $\bar{\mathcal{M}}^{\pm}(f|_U)=\mathbb{S}^1$.
Let $\tilde{f}\in \mathcal{M}^{\pm}(f|_U)$. From Proposition \ref{dd} it follows that $M_0^{\pm}(f|_U)$ is isolated.
Since $M_0^{\pm}(f|_U)=M_0^{\pm}(f)\cap U$, we may assume that $\bar p$ and $U$ are such that $M_0^{\pm}(f|_U)=\emptyset$.
Then, the Codazzi equation and \eqref{hpm} imply that 
$$h^{\pm}\equiv0\;\;\;\mbox{on}\;\;\; U\cap \inter\{p\in M:H(p)=0\}.$$
According to Proposition \ref{correspond}(i), there exists a harmonic function $\th^{\pm} \in \mathcal{C}^{\infty}(U;(0,2\pi))$ satisfying \eqref{sys} on $U$, 
such that the distortion differential of the pair $(f|_U,\tilde{f})$ is given by \eqref{general1} on $U$.
From \eqref{sys} and the above, it follows that the harmonic function $\th^{\pm}$ is constant on $U\cap \inter\{p\in M:H(p)=0\}$ and thus, 
constant on $U$. Then, \eqref{sys} yields that $h^{\pm}\equiv0$ on $U$. Proposition \ref{glphi} and \eqref{hpm}
imply that the Gauss lift $G_{\pm}$ of $f$ is vertically harmonic on $U$.
From \cite[Prop. 23(ii)]{PV} we know that $\|H\|^2$ is an absolute value type function on $U$.
Since $\|H\|^2$ vanishes on an open subset of $U$, it follows that $H\equiv0$ on $U$.
This is a contradiction, since $f|_U$ is non-minimal.

(ii) Assume that $f$ is non-minimal and let $p\in M$. There exists a simply-connected complex chart $(V,z)$ around $p$ such that $f|_V$ is
proper Bonnet. From part (i) it follows that $f|_V$ is non-minimal and Theorem \ref{SC}(iii) implies that either $\bar{\mathcal{M}}^{-}(f|_V)=\mathbb{S}^1$, 
or $\bar{\mathcal{M}}^{+}(f|_V)=\mathbb{S}^1$.
Assume that $\bar{\mathcal{M}}^{\pm}(f|_V)=\mathbb{S}^1$. 
By virtue of Remark \ref{Param}, we parametrize $\bar{\mathcal{M}}^{\pm}(f|_V)$ such that \eqref{param} is valid at $p$ and we write 
$\bar{\mathcal{M}}^{\pm}_p(f|_V)=\mathbb{S}^1$. For every sufficiently small simply-connected neighbourhood $U$ of $p$ we have that $U\subset V$
and therefore, $\bar{\mathcal{M}}^{\pm}_p(f|_{U})=\mathbb{S}^1$. 
Appealing to Theorem \ref{SC}(iv), it is clear that $\mathbb{S}^1_{\pm}$ is a submanifold of $\mathcal{M}(f|_U)$.
\qed
\end{proof}
\medskip

Let $f\colon M\to \Q^4_c$ be a non-minimal locally proper Bonnet surface. By virtue of Proposition \ref{non minimal}(ii) we
give the following definition; the surface $f$ is called {\emph {uniformly locally proper Bonnet}} if there exists a submanifold 
$L^n, 1\leq n\leq2$, of the torus $\mathbb{S}^1\times \mathbb{S}^1$, $\mathbb{S}^1\simeq \R/2\pi\mathbb{Z}$, with the property that for 
every $p\in M$, $L^n$ is also a submanifold of $\mathcal{M}(f|_{U})$ for every sufficiently small simply-connected neighbourhood $U$ of $p$.
In this case, $L^n$ is called {\emph{a deformation manifold for $f$}}.
Moreover, $f$ is called {\emph {locally flexible proper Bonnet}}
if the torus $\mathbb{S}^1\times \mathbb{S}^1$ is a deformation manifold for $f$.

\begin{lemma} \label{unm}
A surface $f\colon M\to \Q^4_c$ is uniformly locally proper Bonnet with deformation manifold $\mathbb{S}^1_{\pm}$ if and only if 
every point of $M$ has a simply-connected neighbourhood $U$ such that $\bar{\mathcal{M}}^{\pm}(f|_U)=\mathbb{S}^1$.
If $\mathbb{S}^1_{\pm}$ is a deformation manifold for $f$, then the set $M^{\pm}_0(f)$ is isolated.
\end{lemma}

\begin{proof}
Assume that $\mathbb{S}^1_{\pm}$ is a deformation manifold for $f$. Then, every point of $M$ has a simply-connected neighbourhood $U$ 
such that $\mathbb{S}^1_{\pm}$ is a submanifold of $\mathcal{M}(f|_{U})$.
From Theorem \ref{SC}(iv) it follows that $\bar{\mathcal{M}}^{\pm}(f|_U)=\mathbb{S}^1$. The converse follows 
in a similar manner with the proof of Proposition \ref{non minimal}(ii).

Suppose now that that $\mathbb{S}^1_{\pm}$ is a deformation manifold for $f$ and arguing indirectly, 
assume that $M^{\pm}_0(f)$ has an accumulation point $p$. Then, there exists a neighbourhood $U$ of $p$ such that $\bar{\mathcal{M}}^{\pm}(f|_U)=\mathbb{S}^1$. 
Proposition \ref{dd} implies that $M_0^{\pm}(f|_U)$ is isolated. This is a contradiction, since $M_0^{\pm}(f|_U)=M_0^{\pm}(f)\cap U$.
\qed
\end{proof}
\medskip

For the proof of the following theorem, we recall (cf. \cite{EGT}) that if $M$ is compact and $u\not\equiv0$ 
is an absolute value function on $M$, then
$$\int_{M} \Delta\log u=-2\pi N(u),$$
where $N(u)$ is the number of zeros of $u$, counted with multiplicities.

\begin{theorem} \label{ess}
Let $f\colon M\to \Q^4_c$ be a non-minimal, compact oriented surface.
The surface $f$ is uniformly locally proper Bonnet with deformation manifold $\mathbb{S}^1_{\pm}$ 
if and only if the Gauss lift $G_{\pm}$ of $f$ is vertically harmonic and non-holomorphic.
\end{theorem}

\begin{proof}
Assume that $\mathbb{S}^1_{\pm}$ is a deformation manifold for $f$.
Lemma \ref{unm} yields that $M_0^{\pm}(f)$ is isolated. 
Then, Lemma \ref{pseudo}(ii) and Proposition \ref{holomorphic Gl} imply that the Gauss lift $G_{\pm}$ of $f$ is non-holomorphic.
From Lemma \ref{unm} and Theorem \ref{local deformability}(ii), it follows that \eqref{fund} is valid on the whole $M$.
Integrating \eqref{fund} on $M$, yields
\be \label{integrals}
\int_{M}\Delta\log \|\mathcal{H}^{\pm}\| -\int_{M}(2K\pm K_N)=\int_{M}\frac{\|\tau^{v}(G_{\pm})\|^2}{4\|\mathcal{H}^{\pm}\|^2}.
\ee
From Lemma \ref{unm} and Proposition \ref{IPU}(ii), it follows that $\|\mathcal{H}^{\pm}\|$ is an absolute value function on $M$
with isolated zeros. Therefore, we have
$$\int_{M}\Delta\log \|\mathcal{H}^{\pm}\|=-2\pi N(\|\mathcal{H}^{\pm}\|).$$
On the other hand, Theorem \ref{Hopf type} and Propositions \ref{IndPU} and \ref{IPU}(i), imply that 
$$\int_{M}(2K\pm K_N)=-2\pi N(\|\mathcal{H}^{\pm}\|).$$
From the above two relations it follows that the left hand side of \eqref{integrals} vanishes identically. 
Therefore, \eqref{integrals} implies that $\|\tau^{v}(G_{\pm})\|\equiv0$ on $M$, and this shows that the Gauss lift $G_{\pm}$ of $f$ is vertically harmonic. 

Conversely, assume that the Gauss lift $G_{\pm}$ of $f$ is vertically harmonic and non-holomorphic. 
By virtue of Lemma \ref{pseudo}(ii),
Proposition \ref{holomorphic Gl} implies that $M\neq M_0^{\pm}(f)$. 
From \cite[Prop. 23(ii)]{PV} it follows that $M_0^{\pm}(f)$ is isolated, and that the mean curvature vector field of $f$ 
does not vanish on any open subset of $M$. 
Then \cite[Thm. 4, Prop. 25]{PV} imply that every point of $M$ has a simply-connected
neighbourhood $U$ such that $\bar{\mathcal{M}}^{\pm}(f|_U)=\mathbb{S}^1$. The proof now follows from Lemma \ref{unm}.
\qed
\end{proof}
\medskip

\noindent{\emph{Proof of Theorem \ref{ULPB}:}}
By virtue of Lemma \ref{non minimal}(ii), either $\mathbb{S}^1_-$, or $\mathbb{S}^1_+$, is a deformation manifold for $f$.
The proof follows immediately from Theorem \ref{ess}.
\qed
\medskip

\noindent{\emph{Proof of Theorem \ref{SPB}:}}
For minimal surfaces, the result is known (cf. \cite{Joh, Vl1}).
Let $f\colon M\to \Q^4_c$ be a non-minimal, compact superconformal surface and arguing indirectly, assume that $f$ is locally proper Bonnet. 

We claim that the normal curvature of $f$ does not change sign. By virtue of Lemma \ref{non minimal}(ii) and
Theorem \ref{SC}(iii), every point of $M$ has a neighbourhood $U$ such that either $\bar{\mathcal{M}}^{-}(f|_U)=\mathbb{S}^1$,
or $\bar{\mathcal{M}}^{+}(f|_U)=\mathbb{S}^1$. Then, Proposition \ref{dd} implies that either $M_0^-(f|_U)$, or $M_0^+(f|_U)$ is isolated.
Since $M_0^{\pm}(f|_U)=M_0^{\pm}(f)\cap U$ and $M_1(f)= M_0^-(f)\cap M_0^+(f)$, we deduce that $M_1(f)$ is isolated.
From Lemma \ref{pseudo}(ii) it follows that the normal curvature of $f$ vanishes at isolated points only, and this proves the claim.

Assume that $\pm K_N\geq0$. Lemma \ref{pseudo}(ii) implies that $\Phi^{\pm}\equiv0$. Therefore, $\mathcal{M}^{\pm}(f|_U)=\emptyset$
for every $U\subset M$. Since $f$ is locally proper Bonnet, from Theorem \ref{SC}(iii) and Lemma \ref{unm} it follows that $f$ 
is uniformly locally proper Bonnet with deformation manifold $\mathbb{S}^1_\mp$.
Then, Theorem \ref{ess} implies that the Gauss lift $G_{\mp}$ is vertically harmonic and non-holomorphic. On the other hand,
since $\Phi^{\pm}\equiv0$, from Proposition \ref{glphi} it follows that $G_{\pm}$ is vertically harmonic.
Since both Gauss lifts of $f$ are vertically harmonic, the mean curvature vector field of $f$ is parallel in
the normal connection. Therefore, $K_N\equiv0$ on $M$. Proposition \ref{holomorphic Gl} then implies that 
$G_{\mp}$ is holomorphic, which is a contradiction.
\qed
\medskip

\begin{corollary} \label{genus zero}
There do not exist uniformly locally proper Bonnet surfaces in $\Q^4_c$ of genus zero.
\end{corollary}

\begin{proof}
Arguing indirectly, assume that $M$ is homeomorphic to $\mathbb{S}^2$ and let $f\colon M\to \Q^4_c$ 
be a uniformly locally proper Bonnet surface. By virtue of Lemma \ref{non minimal}(ii), 
assume that $\mathbb{S}^1_\pm$ is a deformation manifold for $f$.
Theorem \ref{ess} implies that the Gauss lift $G_{\pm}$ of $f$ is vertically harmonic.
Then, from \cite[Thm. 3]{PV} it follows that $f$ is superconformal. This contradicts Theorem \ref{SPB}.
\qed
\end{proof}
\medskip

\noindent{\emph{Proof of Theorem \ref{LFPB}:}}
Assume that $f$ is locally flexible proper Bonnet. From \cite{ET2} it follows that $f$ is non-minimal.
Since both $\mathbb{S}^1_-$ and $\mathbb{S}^1_+$ are
deformation manifolds for $f$, Theorem \ref{ess} implies that both Gauss lifts of $f$ are vertically harmonic. 
Therefore, $f$ has nonvanishing parallel mean curvature vector field. Moreover, Corollary \ref{genus zero} yields that $genus(M)>0$.

Conversely, assume that $f$ has nonvanishing parallel mean curvature vector field and that $genus(M)>0$. 
Since $M$ is not homeomorphic to $\mathbb{S}^2$, it follows that $f$ is not totally umbilical 
and thus, Lemma \ref{pseudo}(i) yields that the Hopf differential $\Phi$ of $f$ does not vanish identically on $M$.
On the other hand, the Codazzi equation implies that $\Phi$ is holomorphic. Therefore, from
Lemmas \ref{zeros} and \ref{pseudo}(i) it follows that the umbilic points of $f$ are isolated.
Then, \cite[Prop. 26(iii)]{PV} implies that every point of $M$ has a simply-connected neighbourhood $U$ such that
$\mathcal{M}(f|_U)=\mathbb{S}^1\times \mathbb{S}^1$. 
This completes the proof.
\qed
\medskip

An immediate consequence of Theorems \ref{Q3} and \ref{LFPB} is the following result due to Umehara \cite{U}.

\begin{theorem}
Let $F\colon M\to \Q^3_c$ be a non-minimal, compact oriented surface with $genus(M)>0$. 
The surface $F$ is locally proper Bonnet if and only if it has constant mean curvature.
\end{theorem}

\begin{proof}
Let $j\colon \Q^3_c\to \Q^4_c$ be a totally geodesic inclusion and set $f=j\circ F$.
From Theorem \ref{Q3} it follows that $F$ is locally proper Bonnet if and only if $f$ is locally flexible. 
Theorem \ref{LFPB} implies that $f$ is locally flexible if and only if it has parallel mean curvature vector field, or equivalently, 
if the mean curvature of $F$ is constant.
\qed
\end{proof}

\bigskip

\noindent Mathematics Department, University of Ioannina\\
45110, Ioannina, Greece \\
E-mail address: kpolymer@cc.uoi.gr


\begin{bibdiv}
\begin{biblist}




\bib{As}{article}{
   author={Asperti, A.C.},
   title={Immersions of surfaces into 4-dimensional spaces with nonzero normal curvature},
   journal={Ann. Mat. Pura Appl. (4)},
   volume={125},
   date={1980},
   number={1},
   pages={313–-328},
}


\bib{BE}{article}{
   author={Bobenko, A.},
   author={Eitner, U.},
   title={Bonnet surfaces and Painlev\'{e} equations},
   journal={J. Reine Angew. Math.},
   volume={499},
   date={1998},
   pages={47--79},
}



\bib{BWW}{article}{
   author={Bolton, J.},
   author={Willmore, T.J.},
   author={Woodward, L.M.},
   title={Immersions of surfaces into space forms},
   conference={
      title={Global differential geometry and global analysis 1984 (Berlin,
      1984)},
   },
   book={
      series={Lecture Notes in Math.},
      volume={1156},
      publisher={Springer, Berlin},
   },
   date={1985},
   pages={46--58},
}

\bib{B}{article}{
   author={Bonnet, O.},
   title={M\'{e}moire sur la th\'{e}orie des surfaces applicables sur une surface don\'{n}e, deuxi\`{e}me partie},
   journal={J. \'{E}c. Polyt.},
   volume={42},
   date={1867},
   pages={1--151},
}



\bib{Cal}{article}{
   author={Calabi, E.},
   title={Minimal immersions of surfaces in Euclidean spheres},
   journal={J. Differential Geom.},
   volume={1},
   date={1967},
   pages={111--125},
}

\bib{Ca}{article}{
   author={Cartan, {\'E}.},
   title={Sur les couples de surfaces applicables avec conservation des
   courbures principales},
   journal={Bull. Sci. Math. (2)},
   volume={66},
   date={1942},
   pages={55--72, 74--85},
}




\bib{CU}{article}{
   author={Castro, I.},
   author={Urbano, F.},
   title={Lagrangian surfaces in the complex Euclidean plane with conformal
   Maslov form},
   journal={Tohoku Math. J. (2)},
   volume={45},
   date={1993},
   number={4},
   pages={565--582},
}



\bib{Chen}{article}{
   author={Chen, B.-Y.},
   title={On the surface with parallel mean curvature vector},
   journal={Indiana Univ. Math. J.},
   volume={22},
   date={1972/73},
   pages={655--666},
}


\bib{Ch0}{article}{
   author={Chern, S.S.},
   title={La g\'eometri\'e des sous-vari\'et\'es d'un espace euclidien \`a plusieurs dimensions},
   journal={Enseign. Math.},
   volume={40},
   date={1955},
   pages={26--46},
}

\bib{Ch}{article}{
   author={Chern, S.S.},
   title={On the minimal immersions of the two-sphere in a space of constant
   curvature},
   conference={
      title={Problems in analysis},
      address={Lectures at the Sympos. in honor of Salomon Bochner,
      Princeton Univ., Princeton, N.J.},
      date={1969},
   },
   book={
      publisher={Princeton Univ. Press, Princeton, N.J.},
   },
   date={1970},
   pages={27--40},
}



\bib{Ch2}{article}{
   author={Chern, S.S.},
   title={Deformation of surfaces preserving principal curvatures},
   conference={
      title={Differential geometry and complex analysis},
   },
   book={
      publisher={Springer, Berlin},
   },
   date={1985},
   pages={155--163},
}

\bib{CGS}{article}{
   author={Cie\'{s}li\'{n}ski, J.},
   author={Goldstein, P.},
   author={Sym, A.},
   title={Isothermic surfaces in $\bold E^3$ as soliton surfaces},
   journal={Phys. Lett. A},
   volume={205},
   date={1995},
   number={1},
   pages={37--43},
}

\bib{CK}{article}{
   author={Colares, A.G.},
   author={Kenmotsu, K.},
   title={Isometric deformations of surfaces in $\R^3$ preserving the mean curvature function},
   journal={Pacific J. Math.},
   volume={136},
   date={1989},
   number={1},
   pages={71--80},
}



\bib{DG2}{article}{
   author={Dajczer, M.},
   author={Gromoll, D.},
   title={Real Kaehler submanifolds and uniqueness of the Gauss map},
   journal={J. Differential Geom.},
   volume={22},
   date={1985},
   number={1},
   pages={13--28},
}



\bib{DJ}{article}{
    author={Dajczer, M.},
    author={Jimenez, M.I.},    
    title={Infinitesimal bendings of submanifolds},
    year={2019},
    eprint={1911.01863},
    archivePrefix={arXiv},
    primaryClass={math.DG}
}




\bib{DT}{article}{
   author={Dajczer, M.},
   author={Tojeiro, R.},
   title={All superconformal surfaces in $\Bbb R^4$ in terms of minimal
   surfaces},
   journal={Math. Z.},
   volume={261},
   date={2009},
   number={4},
   pages={869--890},
}




\bib{DTB}{book}{
   author={Dajczer, M.},
   author={Tojeiro, R.},
   title={Submanifold theory beyond an introduction},
   series={Universitext},
   publisher={Springer, New York},
   date={2019},
   pages={xx+628},
}




\bib{DV2}{article}{
   author={Dajczer, M.},
   author={Vlachos, Th.},
   title={The infinitesimally bendable Euclidean hypersurfaces},
   journal={Ann. Mat. Pura Appl. (4)},
   volume={196},
   date={2017},
   number={6},
   pages={1961--1979},
}



\bib{ES}{article}{
   author={Eells, J.},
   author={Salamon, S.},
   title={Twistorial construction of harmonic maps of surfaces into
   four-manifolds},
   journal={Ann. Scuola Norm. Sup. Pisa Cl. Sci. (4)},
   volume={12},
   date={1985},
   number={4},
   pages={589--640 (1986)},
}

\bib{EGT}{article}{
   author={Eschenburg, J.H.},
   author={Guadalupe, I.V.},
   author={Tribuzy, R.},
   title={The fundamental equations of minimal surfaces in ${\bf C}{\rm
   P}^2$},
   journal={Math. Ann.},
   volume={270},
   date={1985},
   number={4},
   pages={571--598},
}


\bib{ET}{article}{
   author={Eschenburg, J.H.},
   author={Tribuzy, R.},
   title={Branch Points of Conformal Mappings of Surfaces},
   journal={Math. Ann.},
   volume={279},
   date={1988},
   number={4},
   pages={621--633},
}


\bib{ET2}{article}{
   author={Eschenburg, J.H.},
   author={Tribuzy, R.},
   title={Constant mean curvature surfaces in $4$-space forms},
   journal={Rend. Sem. Mat. Univ. Padova},
   volume={79},
   date={1988},
   pages={185--202},
}





\bib{Fr}{article}{
   author={Friedrich, T.},
   title={On surfaces in four-spaces},
   journal={Ann. Global Anal. Geom.},
   volume={2},
   date={1984},
   number={3},
   pages={257--287},
}


\bib{Fu}{article}{
   author={Fujioka, A.},
   title={Bonnet surfaces with non-flat normal bundle in the hyperbolic
   four-space},
   journal={Far East J. Math. Sci. (FJMS)},
   volume={30},
   date={2008},
   number={2},
}


\bib{GMM}{article}{
   author={G{\'a}lvez, J.A.},
   author={Mart{\'{\i}}nez, A.},
   author={Mira, P.},
   title={The Bonnet problem for surfaces in homogeneous 3-manifolds},
   journal={Comm. Anal. Geom.},
   volume={16},
   date={2008},
   number={5},
   pages={907--935},
}


\bib{GS}{article}{
   author={Garcia, R.},
   author={Sotomayor, J.},
   title={Lines of axial curvature on surfaces immersed in $\bold R^4$},
   journal={Differential Geom. Appl.},
   volume={12},
   date={2000},
   number={3},
   pages={253--269},
}


\bib{Grau}{article}{
   author={Graustein, W.C.},
   title={Applicability with preservation of both curvatures},
   journal={Bull. Amer. Math. Soc.},
   volume={30},
   date={1924},
   pages={19--23},
}



\bib{GR}{article}{
   author={Guadalupe, I.V.},
   author={Rodriguez, L.},
   title={Normal curvature of surfaces in space forms},
   journal={Pacific J. Math.},
   volume={106},
   date={1983},
   number={1},
   pages={95--103},
}

\bib{GGTG}{article}{
   author={Gutierrez, C.},
   author={Guadalupe, I.},
   author={Tribuzy, R.},
   author={Gu\'{\i}\~{n}ez, V.},
   title={Lines of curvature on surfaces immersed in ${\bf R}^4$},
   journal={Bol. Soc. Brasil. Mat. (N.S.)},
   volume={28},
   date={1997},
   number={2},
   pages={233--251},
}



\bib{Ha}{article}{
   author={Hasegawa, K.},
   title={On surfaces whose twistor lifts are harmonic sections},
   journal={J. Geom. Phys.},
   volume={57},
   date={2007},
   number={7},
   pages={1549--1566},
}



\bib{H-J}{book}{
   author={Hertrich-Jeromin, U.},
   title={Introduction to M\"{o}bius differential geometry},
   series={London Mathematical Society Lecture Note Series},
   volume={300},
   publisher={Cambridge University Press, Cambridge},
   date={2003},
   pages={xii+413},
   isbn={0-521-53569-7},
}




\bib{IMS}{article}{
   author={Ivanova-Karatopraklieva, I.},
   author={Markov, P. E.},
   author={Sabitov, I.},
   title={Bending of surfaces. III},
   language={Russian, with English and Russian summaries},
   journal={Fundam. Prikl. Mat.},
   volume={12},
   date={2006},
   number={1},
   pages={3--56},
   issn={1560-5159},
   translation={
      journal={J. Math. Sci. (N.Y.)},
      volume={149},
      date={2008},
      number={1},
      pages={861--895},
      issn={1072-3374},
   },
}



\bib{IS}{article}{
   author={Ivanova-Karatopraklieva, I.},
   author={Sabitov, I.},
   title={Deformation of surfaces. I},
   language={Russian},
   note={Translated in J. Math. Sci. {\bf 70} (1994), no. 2, 1685--1716},
   conference={
      title={Problems in geometry, Vol. 23 (Russian)},
   },
   book={
      series={Itogi Nauki i Tekhniki},
      publisher={Akad. Nauk SSSR, Vsesoyuz. Inst. Nauchn. i Tekhn. Inform.,
   Moscow},
   },
   date={1991},
   pages={131--184, 187},
}
	



\bib{JMNB}{book}{
   author={Jensen, G.R.},
   author={Musso, E.},
   author={Nicolodi, L.},
   title={Surfaces in classical geometries: a treatment by moving frames},
   series={Universitext},
   publisher={Springer, Cham},
   date={2016},
}


\bib{JMN}{article}{
   author={Jensen, G.R.},
   author={Musso, E.},
   author={Nicolodi, L.},
   title={Compact surfaces with no Bonnet mate},
   journal={J. Geom. Anal.},
   volume={28},
   date={2018},
   number={3},
   pages={2644–-2652},
}


\bib{JR}{article}{
   author={Jensen, G.R.},
   author={Rigoli, M.},
   title={Twistor and Gauss lifts of surfaces in four-manifolds},
   conference={
      title={Recent developments in geometry},
      address={Los Angeles, CA},
      date={1987},
   },
   book={
      series={Contemp. Math.},
      volume={101},
      publisher={Amer. Math. Soc., Providence, RI},
   },
   date={1989},
   pages={197--232},
}


\bib{Jim}{article}{
   author={Jimenez, M.I.},
   title={Infinitesimal bendings of complete Euclidean hypersurfaces},
   journal={Manuscripta Math.},
   volume={157},
   date={2018},
   number={3-4},
   pages={513--527},
}


\bib{Joh}{article}{
   author={Johnson, G.D.},
   title={An intrinsic characterization of a class of minimal surfaces in
   constant curvature manifolds},
   journal={Pacific J. Math.},
   volume={149},
   date={1991},
   number={1},
   pages={113--125},
}


\bib{KPP}{article}{
   author={Kamberov, G.},
   author={Pedit, F.},
   author={Pinkall, U.},
   title={Bonnet pairs and isothermic surfaces},
   journal={Duke Math. J.},
   volume={92},
   date={1998},
   number={3},
   pages={637--644},
}



\bib{K}{article}{
   author={Kenmotsu, K.},
   title={An intrinsic characterization of $H$-deformable surfaces},
   journal={J. London Math. Soc. (2)},
   volume={49},
   date={1994},
   number={3},
   pages={555--568},
}



\bib{LP}{article}{
   author={Lam, W.Y.},
   author={Pinkall, U.},
   title={Isothermic triangulated surfaces},
   journal={Math. Ann.},
   volume={368},
   date={2017},
   number={1-2},
   pages={165–-195},
}



\bib{L}{article}{
   author={Lawson, H.B.},
   title={Complete minimal surfaces in $S^{3}$},
   journal={Ann. of Math. (2)},
   volume={92},
   date={1970},
   pages={335--374},
}

\bib{LT}{article}{
   author={Lawson, H.B.},
   author={Tribuzy, R.},
   title={On the mean curvature function for compact surfaces},
   journal={J. Differential Geom.},
   volume={16},
   date={1981},
   number={2},
   pages={179--183},
}


\bib{LMW}{article}{
   author={Li, C.},
   author={Miao, P.},
   author={Wang, Z.},
   title={Uniqueness of isometric immersions with the same mean curvature},
   journal={J. Funct. Anal.},
   volume={276},
   date={2019},
   number={9},
   pages={2831--2855},
}


\bib{Little}{article}{
   author={Little, J.A.},
   title={On singularities of submanifolds of higher dimensional Euclidean
   spaces},
   journal={Ann. Mat. Pura Appl. (4)},
   volume={83},
   date={1969},
   pages={261--335},
}


\bib{Ma}{article}{
   author={Markov, P. E.},
   title={Infinitesimal bendings of certain multidimensional surfaces},
   language={Russian},
   journal={Mat. Zametki},
   volume={27},
   date={1980},
   number={3},
   pages={469--479, 495},
}

\bib{Me}{article}{
   author={Mello, L.F.},
   title={Mean directionally curved lines on surfaces immersed in ${\Bbb
   R}^4$},
   journal={Publ. Mat.},
   volume={47},
   date={2003},
   number={2},
   pages={415--440},
}

\bib{Mo}{article}{
   author={Moriya, K.},
   title={Super-conformal surfaces associated with null complex holomorphic curves},
   journal={Bull. London Math. Soc.},
   volume={41},
   date={2009},
   number={2},
   pages={327--331},
}



\bib{Pa}{article}{
   author={Palmer, B.},
   title={Isothermic Surfaces and the Gauss Map},
   journal={Proc. Amer. Math. Soc.},
   volume={104},
   date={1988},
   number={3},
   pages={876--884},
}


\bib{PT}{article}{
   author={Peng, C.},
   author={Tang, Z.},
   title={On surfaces immersed in Euclidean space $\R^4$},
   journal={Sci. China Math.},
   volume={53},
   date={2010},
   number={1},
   pages={261--335},
}


\bib{PV}{article}{
   author={Polymerakis, K.},
   author={Vlachos, Th.},
   title={On the moduli space of isometric surfaces with the same mean curvature in 4-dimensional space forms},
   journal={J. Geom. Anal.},
   volume={29},
   date={2019},
   number={2},
   pages={1320–-1355},
}


\bib{Ra}{article}{
   author={Raffy, L.},
   title={Sur certaines surfaces, dont les rayons de courbure sont li\'{e}s par
   une relation},
   language={French},
   journal={Bull. Soc. Math. France},
   volume={19},
   date={1891},
   pages={158--169},
}


\bib{RH}{article}{
   author={Roussos, I.M.},
   author={Hern{\'a}ndez, G.E.},
   title={On the number of distinct isometric immersions of a Riemannian
   surface into ${\bf R}^3$ with given mean curvature},
   journal={Amer. J. Math.},
   volume={112},
   date={1990},
   number={1},
   pages={71--85},
}



\bib{ST}{article}{
   author={Smyth, B.},
   author={Tinaglia, G.},
   title={The number of constant mean curvature isometric immersions of a
   surface},
   journal={Comment. Math. Helv.},
   volume={88},
   date={2013},
   number={1},
   pages={163--183},
}



\bib{Tr}{article}{
   author={Tribuzy, R.},
   title={A characterization of tori with constant mean curvature in space
   form},
   journal={Bol. Soc. Brasil. Mat.},
   volume={11},
   date={1980},
   number={2},
   pages={259--274},
}


\bib{U}{article}{
   author={Umehara, M.},
   title={A characterization of compact surfaces with constant mean curvature},
   journal={Proc. Amer. Math. Soc.},
   volume={108},
   date={1990},
   number={2},
   pages={483--489},
}


\bib{Vl1}{article}{
   author={Vlachos, Th.},
   title={Congruence of minimal surfaces and higher fundamental forms},
   journal={Manuscripta Math.},
   volume={110},
   date={2003},
   number={1},
   pages={77--91},
}






\bib{Yau}{article}{
   author={Yau, S.T.},
   title={Submanifolds with constant mean curvature. I, II},
   journal={Amer. J. Math.},
   volume={96},
   date={1974},
   pages={346--366; ibid. 97 (1975), 76--100},
}

\end{biblist}
\end{bibdiv}
\end{document}